\documentclass[12pt, reqno]{amsart}

\usepackage{amsmath, amsthm, amssymb}
\usepackage{enumerate}
\usepackage[margin=1.0in]{geometry}
\usepackage{xcolor}
\definecolor{cite}{rgb}{0.30,0.60,1.00}
\definecolor{url}{rgb}{0.00,0.00,0.80}
\definecolor{link}{rgb}{0.40,0.10,0.20}
\usepackage[pdfusetitle,colorlinks,linkcolor=link,urlcolor=url,citecolor=cite,pagebackref,breaklinks]{hyperref}
\usepackage{graphicx}
\usepackage{cleveref}
\usepackage{mathdots}
\usepackage{tikz-cd}
\usepackage{comment}
\usepackage{xypic}
\usepackage{mathtools}

\newtheorem{theorem}{Theorem}[section]

\newtheorem{proposition}[theorem]{Proposition}
\newtheorem{lemma}[theorem]{Lemma}

\newtheorem{corollary}[theorem]{Corollary}

\theoremstyle{definition}

\theoremstyle{definition}
\newtheorem{remark}[theorem]{Remark}
\theoremstyle{definition}

\newcommand{\zIntegers}{\mathbb{Z}}

\newcommand{\cComplex}{\mathbb{C}}
\newcommand{\multiplicativegroup}[1]{#1^{\times}}
\newcommand{\RealPart}{\mathrm{Re}}
\newcommand{\Hom}{\mathrm{Hom}}
\newcommand{\EndomorphismRing}{\operatorname{End}}

\newcommand{\idmap}{\mathrm{id}}

\newcommand{\abs}[1]{\left|#1\right|}
\newcommand{\sizeof}[1]{\left|#1\right|}

\newcommand{\innerproduct}[2]{\left(#1,#2\right)}
\newcommand{\Norm}[1]{\left\Vert #1\right\Vert }
\newcommand{\standardForm}[2]{\left\langle #1,#2\right\rangle}

\newcommand{\fieldCharacter}{\psi}
\newcommand{\centralCharacter}[1]{\omega_{#1}}
\newcommand{\Ind}[3]{\mathrm{Ind}_{#1}^{#2}\left(#3\right)}
\newcommand{\ind}[3]{\mathrm{ind}_{#1}^{#2}\left(#3\right)}
\newcommand{\Whittaker}{\mathcal{W}}
\newcommand{\Contragradient}[1]{#1^{\vee}}

\newcommand{\SpehRepresentation}[2]{\Delta\left(#1, #2\right)}
\newcommand{\SteinbergRepresentation}[2]{\mathrm{St}\left(#1, #2\right)}

\newcommand{\besselSpehFunction}[2]{\mathcal{BS}_{\SpehRepresentation{#1}{#2}, \fieldCharacter}}
\newcommand{\fourierTransform}[2]{\mathcal{F}_{#1}#2}
\newcommand{\GKGammaFactor}[3]{\gamma^{\mathrm{GK}}\left(#1 \times #2, #3\right)}
\newcommand{\LocalGKGammaFactor}[4]{\gamma^{\mathrm{GK}}\left(#1, #2 \times #3, #4\right)}

\newcommand{\GKPreGammaFactor}[3]{\Gamma^{\mathrm{GK}}\left(#1 \times #2, #3\right)}

\newcommand{\grpIndex}[2]{\left[#1:#2\right]}

\newcommand{\transpose}[1]{\, {}^{t}#1}

\newcommand{\IdentityMatrix}[1]{I_{#1}}
\newcommand{\diag}{\mathrm{diag}}

\newcommand{\trace}{\operatorname{tr}}
\newcommand{\GL}{\mathrm{GL}}
\newcommand{\UnipotentSubgroup}{U}
\newcommand{\UnipotentRadicalForWss}[2]{N_{\left(#2^{#1}\right)}}
\newcommand{\UnipotentRadicalForWssRecursion}[2]{\mathcal{Y}_{c,k}}
\newcommand{\UnipotentRadical}{N}
\newcommand{\ParabolicSubgroup}{P}

\newcommand{\finiteField}{\mathbb{F}}

\newcommand{\squareMatrix}{\operatorname{Mat}}
\newcommand{\Mat}[2]{\operatorname{Mat}_{#1 \times #2}}

\newcommand{\Steinberg}{\operatorname{St}}
\newcommand{\ProjectionOperator}{\operatorname{pr}}
\newcommand{\SymmetricGroup}{\mathfrak{S}}

\newcommand{\ParabolicForSpeh}[2]{P_{\left({#1}^{#2}\right)}}
\newcommand{\UnipotentForSpeh}[2]{N_{\left({#1}^{#2}\right)}}
\newcommand{\PoincarePolynomial}[2]{P_{#2}}

\newcommand{\localField}{F}
\newcommand{\ringOfIntegers}{\mathfrak{o}}

\newcommand{\maximalIdeal}{\mathfrak{p}}
\newcommand{\depthZeroRepresentation}{\mathcal{T}}
\newcommand{\differential}{\mathrm{d}}
\newcommand{\mdifferential}{\differential^{\times}}
\newcommand{\quotientMap}{\nu}
\newcommand{\Lift}{\mathcal{L}}
\newcommand{\uniformizer}{\varpi}
\newcommand{\VolumeOf}{\operatorname{Vol}}

\newcommand{\parabolicSection}{\Phi^{\left(z_1, \dots, z_c\right)}}
\newcommand{\intertwiningOperator}{M^{\left(z_1, \dots, z_c\right)}}
\newcommand{\holomorphicRepresentation}{\depthZeroRepresentation^{\left(z_1, \dots, z_c\right)}}
\newcommand{\WhittakerFunctional}[1]{\ell_{#1, \fieldCharacter}}

\newcommand{\gShortSpehWhittakerFunctional}[3]{\ell_{\SpehRepresentation{#1}{#3}}}

\newcommand{\GKGaussSum}[3]{\mathcal{G}\left(#1 \times #2, #3\right)}
\newcommand{\fieldCharacterkc}[2]{\fieldCharacter_{\left({#2}^{#1}\right)}}

\newcommand{\weylElement}[1]{w_{({#1})}}
\newcommand{\zetaOperator}{\mathrm{Z}}
\newcommand{\dualZetaOperator}{\mathrm{Z}^{\ast}}
\newcommand{\specialBesselSpeh}[2][\fieldCharacter]{\mathcal{K}_{{#2}, {#1}}}
\newcommand{\firstSpecialFunctional}{S_0}
\newcommand{\firstDualSpecialFunctional}{S_0^{\ast}}
\newcommand{\secondSpecialFunctional}{S}
\newcommand{\secondDualSpecialFunctional}{S^{\ast}}

\newcommand{\Moeglin}{M{\oe}glin}
\crefformat{footnote}{#2\footnotemark[#1]#3}

\hypersetup{pdfauthor={Elad Zelingher},
	pdfsubject={Representation theory},
	pdfkeywords={Speh representations, Level zero representations, Gamma factors}}

\title[Speh representations for level zero representations]{On Speh representations for level zero supercuspidal representations and Ginzburg--Kaplan gamma factors}

\author{Elad Zelingher}
\address{Department of Mathematics, University of Michigan, 1844 East Hall, 530 Church Street, Ann Arbor, MI 48109-1043 USA}
\email{eladz@umich.edu}

\keywords{Speh representations, Level zero representations, Gamma factors}
\subjclass[2020]{11F70, 11F66, 20C33, 22E50}

\begin{document}

\begin{abstract}
	We establish a relation between Speh representations of $\GL_n\left(\finiteField_q\right)$ and Speh representations of $\GL_n\left(\localField\right)$, where $\localField$ is a non-archimedean local field. We use irreducible level zero supercuspidal representations to show that these two notions of Speh representations associated to cuspidal representations are related via a commutative diagram, and that their corresponding $(k,c)$ $\fieldCharacter$-Whittaker models are also related. We use these results to relate the local Ginzburg--Kaplan integrals \cite{ginzburg2019tensor, Kaplan2023}, \cite[Appendix A]{kaplan2018} for level zero supercuspidal representations to their finite field counterparts \cite{CarmonZelingher2024}.
\end{abstract}
	\dedicatory{\bf In memory of Fabian Sharristh}
\maketitle

\section{Introduction}\label{sec:introduction}

In the theory of automorphic representations of general linear groups, Speh representations are important objects. They were introduced by Jacquet in \cite{Jacquet1984}, where he conjectured that the discrete spectrum of $L^2\left(\left[\GL_n\right]\right)$ corresponding to the automorphic quotient $\left[\GL_n\right]$ decomposes as the direct sum of Speh representations. This was later proved by \Moeglin{} and Waldspurger in their celebrated work \cite{MoeglinWaldspurger1989}. This result serves as a base case for Arthur's conjectures regarding the decomposition of the discrete spectrum of $L^2\left(\left[G\right]\right)$, where $G$ is a reductive group. See \cite[Page 23]{Arthur2013}.

A notion of (generalized) Speh representations also exists in the context of local fields \cite[Section 2.2]{Jacquet1984}, \cite{CaiFriedbergGourevitchKaplan2023} and, unsurprisingly, the local components of a global Speh representation are local Speh representations.

In the past recent years, a new use for Speh representations was found: they serve as an ingredient in certain new integral constructions of various $L$-functions \cite{CaiFriedbergGinzburgKaplan2019,CaiFriedbergKaplan2022,kaplan2018,GourevitchKaplan2023,ginzburg2019tensor,GinzburgSoudry2020}. A fundamental property of Speh representations utilized in these constructions is a local uniqueness property of a generalized Whittaker functional. We move to explain this now.

Let $\localField$ be a non-archimedean local field with ring of integers $\ringOfIntegers$, maximal ideal $\maximalIdeal$ and residual field $\finiteField = \ringOfIntegers \slash \maximalIdeal$ of cardinality $q$, and denote by $\quotientMap \colon \ringOfIntegers \to \finiteField$ the quotient map. Choose a uniformizer $\uniformizer \in \maximalIdeal$. Let $\fieldCharacter \colon \localField \to \multiplicativegroup{\cComplex}$ be a non-trivial character. Given positive integers $k$ and $c$, let $\UnipotentRadicalForWss{k}{c}$ be the unipotent radical of the standard parabolic subgroup $\ParabolicSubgroup_{(c^k)}$ of $\GL_{kc}\left(\localField\right)$ corresponding to the composition $\left(c^k\right)$. We define a character $\fieldCharacterkc{k}{c} \colon \UnipotentRadicalForWss{k}{c} \to \multiplicativegroup{\cComplex}$ by the formula
$$\fieldCharacterkc{k}{c} \begin{pmatrix}
	\IdentityMatrix{c} & X_1 & \ast & \ast  & \ast \\
	& \IdentityMatrix{c} & X_2 & \ast & \ast \\
	& & \ddots & \ddots & \ast  \\
	& & & \IdentityMatrix{c} &  X_{k-1} \\
	& & & & \IdentityMatrix{c}
\end{pmatrix} = \fieldCharacter\left( \sum_{j=1}^{k-1} \trace X_j \right).$$
Then a fundamental result of \cite{CaiFriedbergGourevitchKaplan2023} is that for any Speh representation $\SpehRepresentation{\depthZeroRepresentation}{c}$ (where $\depthZeroRepresentation$ is a generic representation of $\GL_k\left(\localField\right)$) there exists a unique (up to scalar multiplication) $(k,c)$ $\fieldCharacter$-Whittaker functional, that is, a non-zero element $$\gShortSpehWhittakerFunctional{\depthZeroRepresentation}{k}{c} \in \Hom_{\UnipotentRadicalForWss{k}{c}}\left(\SpehRepresentation{\depthZeroRepresentation}{c}, \fieldCharacterkc{k}{c}\right).$$

Recently, in his master's thesis \cite{Carmon2023}, Carmon defined the notion of (generalized) Speh representations of finite general linear groups, and showed that a similar uniqueness property holds.

It is very natural to ask what the relation between Speh representations over local fields and over finite fields is. In this paper, we establish a relation between these two representations using level zero representations. 

Given any irreducible cuspidal representation $\tau$ of $\GL_k\left(\finiteField\right)$, we may construct an irreducible level zero supercuspidal representation $\depthZeroRepresentation$ of $\GL_k\left(\localField\right)$. This construction is not unique and depends on a choice of a central character $\centralCharacter{\depthZeroRepresentation}$ for $\depthZeroRepresentation$. The Speh representation $\SpehRepresentation{\tau}{c}$ of $\GL_{kc}\left(\finiteField\right)$ is the irreducible subrepresentation of minimal dimension of the parabolically induced representation $\tau^{\circ c} = \tau \circ \dots \circ \tau$. It occurs with multiplicity one in $\tau^{\circ c}$. On the hand, the Speh representation $\SpehRepresentation{\depthZeroRepresentation}{c}$ of $\GL_{kc}\left(\localField\right)$ is defined as the image of the standard intertwining operator $\intertwiningOperator_{\weylElement{1^c}} \colon \holomorphicRepresentation \to \depthZeroRepresentation^{(z_c,\dots,z_1)}$ at the point $(z_1,\dots,z_c) = \left(\frac{c-1}{2}, \frac{c-3}{2}, \dots, -\frac{c-1}{2}\right)$, where for a choice of complex numbers $(z_1,\dots,z_c) \in \cComplex^c$, the representation $\holomorphicRepresentation$ is the (normalized) parabolically induced representation $$\holomorphicRepresentation = \abs{\det}^{z_1} \depthZeroRepresentation \times \dots \times \abs{\det}^{z_c} \depthZeroRepresentation = \Ind{\ParabolicForSpeh{k}{c}}{\GL_{kc}\left(\localField\right)}{\abs{\det}^{z_1} \depthZeroRepresentation \otimes \dots \otimes \abs{\det}^{z_c} \depthZeroRepresentation}.$$ For any element $v_{\tau} \in \tau$, we define a natural lift $\Lift v_{\tau} \in \depthZeroRepresentation$ corresponding to $v_{\tau}$. We then extend this notion to define for an element $v_{\tau^{\otimes c}} \in \tau^{\otimes c}$ a lift $\Lift v_{\tau^{\otimes c}} \in \depthZeroRepresentation^{\otimes c}$, and then using the Iwasawa decomposition, we define for an element $f_{\tau^{\circ c}} \in \tau^{\circ c}$ a lift $\Lift^{(z_1,\dots,z_c)} f_{\tau^{\circ c}} \in \holomorphicRepresentation$. Our first result establishes a relation between $\SpehRepresentation{\tau}{c}$ and $\SpehRepresentation{\depthZeroRepresentation}{c}$ by showing that a certain diagram is commutative (\Cref{cor:intertwining-operator-of-depth-zero-representation-of-speh}).

\begin{theorem}
	The following diagram commutes:
		$$\xymatrix{
			\tau^{\circ c} \ar[d]_{\Lift^{\left(\frac{c-1}{2},\dots,-\frac{c-1}{2}\right)}} \ar[rr]^{\ProjectionOperator_{\SpehRepresentation{\tau}{c}}}  & & \SpehRepresentation{\tau}{c} \ar[d]^{\Lift^{\left(-\frac{c-1}{2},\dots,\frac{c-1}{2}\right)}}\\
			\depthZeroRepresentation^{\left(\frac{c-1}{2},\dots,-\frac{c-1}{2}\right)} \ar[rr]^{M^{\ast}_{\weylElement{1^c}}} & & \depthZeroRepresentation^{\left(-\frac{c-1}{2},\dots,\frac{c-1}{2}\right)},}$$
	Hence, for any $f_{\SpehRepresentation{\tau}{c}} \in \SpehRepresentation{\tau}{c}$, we have that $\Lift f_{\SpehRepresentation{\tau}{c}} \in \SpehRepresentation{\depthZeroRepresentation}{c}$.
	Here, $$\ProjectionOperator_{\SpehRepresentation{\tau}{c}} \colon \tau^{\circ c} \to \SpehRepresentation{\tau}{c}$$ is the projection operator and $M^{\ast}_{\weylElement{1^c}}$ is a certain non-zero scalar multiple of $M^{\left(\frac{c-1}{2},\dots, -\frac{c-1}{2}\right)}_{\weylElement{1^c}}$.
\end{theorem}
Our proof involves an explicit computation of simple reflection intertwining operators, which utilizes Cartan decompositions of matrices. As a byproduct, we also obtain a similar relation between the Steinberg representations $\SteinbergRepresentation{\tau}{c}$ and $\SteinbergRepresentation{\depthZeroRepresentation}{c}$ (\Cref{rem:commutative-diagram-for-steinberg-representations}).

Our next results establish a relation between $(k,c)$ $\fieldCharacter$-Whittaker functions of $\SpehRepresentation{\tau}{c}$ and of $\SpehRepresentation{\depthZeroRepresentation}{c}$. Suppose that $\fieldCharacter \colon \multiplicativegroup{\localField} \to \multiplicativegroup{\cComplex}$ has conductor $\maximalIdeal$, and regard $\fieldCharacter$ also as a non-trivial character $\fieldCharacter \colon \finiteField \to \multiplicativegroup{\cComplex}$. Let $\WhittakerFunctional{\tau}$ be a $(k,1)$ $\fieldCharacter$-Whittaker functional of $\tau$. We may use $\WhittakerFunctional{\tau}$ to construct a $(k,1)$ $\fieldCharacter$-Whittaker functional $\WhittakerFunctional{\depthZeroRepresentation}$ for $\depthZeroRepresentation$ as in \cite[Proposition 3.7]{YeZeligher18}. Following the recursive process in \cite[Section 3.2]{CaiFriedbergGourevitchKaplan2023}, we use $\WhittakerFunctional{\tau}$ and $\WhittakerFunctional{\depthZeroRepresentation}$ to construct $(k,c)$ $\fieldCharacter$-Whittaker functionals $\gShortSpehWhittakerFunctional{\tau}{k}{c}$ and $\gShortSpehWhittakerFunctional{\depthZeroRepresentation}{k}{c}$ for $\SpehRepresentation{\tau}{c}$ and $\SpehRepresentation{\depthZeroRepresentation}{c}$, respectively. We show the following relations between these functionals (\Cref{prop:relation-between-k-c-whittaker-models}, \Cref{prop:support-of-diagonal-whittaker-elements} part \ref{item:support-of-W-diag-t-I-kc} and \Cref{thm:special-value-of-diag-c(c-1)-t}).

\begin{theorem}
	Let $f_{\SpehRepresentation{\tau}{c}} \in \SpehRepresentation{\tau}{c}$ and define for $g \in \GL_{kc}\left(\finiteField\right)$, $$W\left(g\right) = \standardForm{\SpehRepresentation{\tau}{c}\left(g\right) f_{\SpehRepresentation{\tau}{c}}}{\gShortSpehWhittakerFunctional{\tau}{k}{c}}$$ and for $g \in \GL_{kc}\left(\localField\right)$, $$\Lift W\left(g\right) = \standardForm{\SpehRepresentation{\depthZeroRepresentation}{c}\left(g\right) \Lift f_{\SpehRepresentation{\tau}{c}}}{\gShortSpehWhittakerFunctional{\depthZeroRepresentation}{k}{c}}.$$
	\begin{enumerate}
		\item For any $k_0 \in \GL_{kc}\left(\ringOfIntegers\right)$, we have the relation $$\Lift W\left(k_0\right) = W\left(\quotientMap\left(k_0\right)\right).$$
		\item Suppose that $g = \diag\left(t, \IdentityMatrix{(k-1)c}\right)$ where $t = \diag\left(\uniformizer^{i_1},\dots,\uniformizer^{i_c}\right)$, and $i_1,\dots,i_c \in \zIntegers$. Then we have $\Lift W\left(g\right) = 0$ if $\min\left(i_1,\dots,i_c\right) < 0$. If $c < k$, we have that $\Lift W\left(g\right) = 0$ unless $i_1 = i_2 = \dots = i_c = 0$.
		\item Suppose that $k=c$ and that $g = \diag\left(t, \IdentityMatrix{(c-1)c}\right)$ where $t$ is as above and $\min\left(i_1,\dots,i_c\right) > 0$. Then $\Lift W\left(g\right) = 0$ unless $i_1 = i_2 = \dots = i_c = m > 0$, in which case $$\Lift W\left(g\right) = \centralCharacter{\depthZeroRepresentation}\left(\uniformizer\right)^m q^{-\left(m+1\right)\left(c-1\right)\binom{c}{2}} \standardForm{f_{\SpehRepresentation{\tau}{c}}\restriction_{\diag\left(\IdentityMatrix{c}, \GL_{(k-1)c}\left(\finiteField\right)\right)}}{\firstSpecialFunctional},$$
		where $\firstSpecialFunctional \colon \tau \otimes \SpehRepresentation{\tau}{c-1} \to \cComplex$ is a certain functional given in \Cref{thm:special-value-of-diag-t-c(c-1)}.
	\end{enumerate}
\end{theorem}
We also have similar results for elements of the form $\diag\left(\IdentityMatrix{(k-1)c}, t\right)$, where $t$ is as in the theorem, see \Cref{prop:support-of-diagonal-whittaker-elements} part \ref{item:support-of-W-diag-I-kc-t} and \Cref{thm:special-value-of-diag-c(c-1)-t}. It is an interesting problem to find an explicit formula for the values of $\Lift W\left(g\right)$ as in the theorem for general $g$, but at this moment this seems too ambitious. Let us mention that Luo \cite{Luo2024} has recently given an asymptotic expansion for $(k,c)$ $\fieldCharacter$-Whittaker functions of Speh representations at elements of the form $\diag\left(t, \IdentityMatrix{(k-1)c}\right)$, but for our purposes we need more explicit formulas.

We use our results to link the finite field Ginzburg--Kaplan gamma factor defined by Carmon and the author \cite{CarmonZelingher2024} to its local field counterpart defined by Kaplan \cite[Appendix A]{kaplan2018}, \cite{Kaplan2023}. This gamma factor arises from an integral representation that uses a Speh representation as an ingredient. This integral represents the $L$-function associated to the tensor product representation of a pair of representations $\Pi$ and $\depthZeroRepresentation$ of $\GL_c$ and $\GL_k$, respectively, where $\depthZeroRepresentation$ is assumed to be generic, but nothing is assumed on $\Pi$. We show that in the non-exceptional case, the local Ginzburg--Kaplan gamma factor associated with a pair of irreducible level zero supercuspidal representations agrees with the finite field Ginzburg--Kaplan gamma factor associated with the corresponding pair of irreducible cuspidal representations (\Cref{cor:ginzburg-kaplan-gamma-factor-equality}). In the exceptional case, we use the equality of the local Ginzburg--Kaplan gamma factor and of the local Jacquet--Piatetski-Shapiro--Shalika gamma factor \cite{Jacquet1983rankin} to compute the finite field Ginzburg--Kaplan gamma factor (\Cref{thm:pi-equals-tau-dual}). We also obtain an identity that shows that for any pair of irreducible cuspidal representations of finite general linear groups, the Ginzburg--Kaplan gamma factor agrees with the tensor product $\varepsilon_0$-factor \cite{ye2021epsilon} (\Cref{cor:gamma-factor-equality-with-epsilon-factor}). This identity is \cite[Theorem 4.24]{CarmonZelingher2024}, and it is crucial for proving the results of \cite[Section 5]{CarmonZelingher2024}, which link special values of Bessel--Speh functions to exotic matrix Kloosterman sums.

\begin{theorem}
	Let $\pi$ and $\tau$ be irreducible cuspidal representations of $\GL_c\left(\finiteField\right)$ and $\GL_k\left(\finiteField\right)$ respectively, and let $\Pi$ and $\depthZeroRepresentation$ be irreducible level zero supercuspidal representations of $\GL_c\left(\localField\right)$ and $\GL_k\left(\localField\right)$ constructed from $\pi$ and $\tau$, respectively. Then \begin{enumerate}
		\item If $\pi \ncong \Contragradient{\tau}$ then $$\LocalGKGammaFactor{s}{\Pi}{\depthZeroRepresentation}{\fieldCharacter} = \GKGammaFactor{\pi}{\tau}{\fieldCharacter}.$$
		\item If $\pi \cong \Contragradient{\tau}$ then $$\GKGammaFactor{\pi}{\tau}{\fieldCharacter} = -\centralCharacter{\tau}\left(-1\right)^{c-1} q^{-\frac{c}{2}}.$$
		\item In all cases we have $$\GKGammaFactor{\pi}{\tau}{\fieldCharacter} = \varepsilon_0\left(\pi \times \tau, \fieldCharacter\right).$$
	\end{enumerate}
\end{theorem}

Let us mention that such relation between a finite field gamma factor and its local field counterpart for level zero representation has been established many times before for various other integrals \cite{Ye18, NienZhang18, YeZeligher18, Jo2024, Chai2024}. However, the case in this article is considerably more involved, since the Ginzburg--Kaplan zeta integrals do not use the representation $\depthZeroRepresentation$ directly, but instead use the $(k,c)$ $\fieldCharacter$-Whittaker model of the Speh representation $\SpehRepresentation{\depthZeroRepresentation}{c}$ associated to it. Hence it is crucial to first establish a relation between the representations $\SpehRepresentation{\tau}{c}$ and $\SpehRepresentation{\depthZeroRepresentation}{c}$ and between their corresponding $(k,c)$ $\fieldCharacter$-Whittaker models.

We also use our results to give a modified functional equation for the finite field Ginzburg--Kaplan zeta integrals in the exceptional case (\Cref{cor:modified-functional-equation}).
\begin{theorem}
	Suppose that $\pi = \Contragradient{\tau}$. Then for any $v \in \Contragradient{\tau}$, $v^{\vee} \in \tau$ and $W \in \Whittaker\left(\SpehRepresentation{\tau}{c}, \fieldCharacterkc{c}{c}\right)$ as above, we have
	\begin{align*}
		\standardForm{\dualZetaOperator\left(W, \pi \times \tau \right) v}{v^{\vee}} =& \GKPreGammaFactor{\pi}{\tau}{\fieldCharacter} \left(\standardForm{\zetaOperator\left(W, \pi \times \tau \right) v}{v^{\vee}}\right) \\
		&+ q^{-\frac{c^2}{2}} \cdot  \standardForm{v \otimes v^{\vee} \otimes f_{\SpehRepresentation{\tau}{c}}\restriction_{\diag\left(\IdentityMatrix{c}, \GL_{(k-1)c}\left(\finiteField\right)\right)}}{\secondSpecialFunctional},
	\end{align*}
	where $\secondSpecialFunctional \colon \Contragradient{\tau} \otimes \tau \otimes \tau \otimes \SpehRepresentation{\tau}{c-1} \to \cComplex$ is a certain functional defined in \Cref{subsubsec:computation-for-level-zero-representations}.
\end{theorem}

We hope to utilize these computations in the future in order to compute the generalized doubling method integrals \cite{CaiFriedbergGinzburgKaplan2019, CaiFriedbergKaplan2022, GourevitchKaplan2023, kaplan2018} for level zero representations of $G \times \GL_k$, where $G$ is a classical group. We mention that this was done for several cases of the classical doubling method \cite{piatetski1986varepsilon, piatetski1987functions, lapid2011local} in \cite{Kim2000}.

\subsection*{Acknowledgments}
I would like to thank Nir Elber, Hahn Lheem and Jialiang Zou for many helpful discussions about lifting intertwining operators, at an REU held by Jialiang and me in Summer 2023 at the University of Michigan (supported by NSF Grant DMS-1840234). I would also like to thank Oded Carmon for many discussions about Speh representations and $(k,c)$ $\fieldCharacter$-Whittaker models.

This work is dedicated to my beloved grandfather who passed away while this manuscript was under preparation.

\tableofcontents

\subsection{Notation}
Suppose that $R$ is a ring. Given a composition $(n_1,\dots,n_r)$ of a positive integer $n$, let $$\UnipotentRadical_{(n_1,\dots,n_r)}\left(R\right) = \left\{ \begin{pmatrix}
	\IdentityMatrix{n_1} & \ast & \ast & \ast \\
	& \IdentityMatrix{n_2} & \ast & \ast \\
	& & \ddots & \ast \\
	& & & \IdentityMatrix{n_r}
\end{pmatrix} \right\}$$ and let $$P_{(n_1,\dots,n_r)}\left(R\right) = \left\{\diag\left(h_1,\dots,h_r\right) \mid h_j \in \GL_{n_j}\left(R\right)\right\} \ltimes \UnipotentRadical_{(n_1,\dots,n_r)}\left(R\right).$$
The group $\ParabolicSubgroup_{(n_1,\dots,n_r)}\left(R\right)$ is \emph{the standard parabolic subgroup of $\GL_n\left(R\right)$ corresponding to the composition $(n_1,\dots,n_r)$}, and $\UnipotentRadical_{(n_1,\dots,n_r)}\left(R\right)$ is the \emph{unipotent radical of $\ParabolicSubgroup_{(n_1,\dots,n_r)}\left(R\right)$}. We will mostly take $R$ to be a field, and we will often omit this field from the notation when there is no fear of confusion.

If $R = \finiteField$ is a finite field, given representations $\tau_1$, $\dots$, $\tau_r$ of $\GL_{n_1}\left(\finiteField\right)$, $\dots$, $\GL_{n_r}\left(\finiteField\right)$, respectively, we denote their parabolic induction by $\tau_1 \circ \dots \circ \tau_r$. This is the representation of $\GL_n\left(\finiteField\right)$ acting by right translations on the space of all functions $f \colon \GL_n\left(\finiteField\right) \to \tau_1 \otimes \dots \otimes \tau_r$, satisfying the condition that for any $g \in \GL_n\left(\finiteField\right)$, and any $p \in \ParabolicSubgroup_{(n_1,\dots,n_r)}\left(\finiteField\right)$ of the form $p = \diag\left(h_1,\dots,h_r\right) u$, where $u \in \UnipotentRadical_{(n_1,\dots,n_r)}\left(\finiteField\right)$ and $h_1 \in \GL_{n_1}\left(\finiteField\right)$, $\dots$, $h_r \in \GL_{n_r}\left(\finiteField\right)$, $$f\left(p g\right) = \tau_1\left(h_1\right) \otimes \dots \otimes \tau_r\left(h_r\right) f\left(g\right).$$
By \cite[Theorem 2.4]{Gelfand70}, for every irreducible representation $\sigma$ of $\GL_n\left(\finiteField\right)$, there exist a composition $(n_1,\dots,n_r)$ of $n$ and irreducible cuspidal representations $\tau_1$, $\dots$, $\tau_r$ of $\GL_{n_1}\left(\finiteField\right)$, $\dots$, $\GL_{n_r}\left(\finiteField\right)$, respectively, such that $\sigma$ is a subrepresentation of the parabolically induced representation $\tau_1 \circ \dots \circ \tau_r$. The equivalence classes of the representations $\tau_1$, $\dots$, $\tau_r$ are uniquely determined, up to permutation. We define the \emph{cuspidal support of $\sigma$} to be the multiset consisting of these equivalence classes.

If $R = \localField$ is a non-archimedean local field with absolute value $\abs{\cdot}$, given smooth representations $\depthZeroRepresentation_1$, $\dots$, $\depthZeroRepresentation_r$ of $\GL_{n_1}\left(\localField\right)$, $\dots$, $\GL_{n_r}\left(\localField\right)$, respectively, we denote their (normalized) parabolic induction by $\depthZeroRepresentation_1 \times \dots \times \depthZeroRepresentation_r$. This is the representation of $\GL_n\left(\localField\right)$ acting by right translations on the space of all right-smooth functions $\Phi \colon \GL_n\left(\localField\right) \to \depthZeroRepresentation_1 \otimes \dots \otimes \depthZeroRepresentation_r$, satisfying the condition that for any $g \in \GL_n\left(\localField\right)$, and any $p \in \ParabolicSubgroup_{(n_1,\dots,n_r)}\left(\localField\right)$ of the form $p = \diag\left(h_1,\dots,h_r\right) u$, where $u \in \UnipotentRadical_{(n_1,\dots,n_r)}\left(\localField\right)$, and $h_1 \in \GL_{n_1}\left(\localField\right)$, $\dots$, $h_r \in \GL_{n_r}\left(\localField\right)$, $$\Phi\left(pg\right) = \delta^{\frac{1}{2}}_{\ParabolicSubgroup_{(n_1,\dots,n_r)}}\left(\diag\left(h_1,\dots,h_r\right)\right) \depthZeroRepresentation_1\left(h_1\right) \otimes \dots \otimes \depthZeroRepresentation_r\left(h_r\right) \Phi\left(g\right),$$
where $\delta_{\ParabolicSubgroup_{(n_1,\dots,n_r)}}$ is the modulus character.

For the reader's convenience, we include here the following formula for a modulus character commonly used in the paper. For $h_1 \in \GL_{n_1}\left(\localField\right)$ and $h_2 \in \GL_{n_2}\left(\localField\right)$,
$$\delta_{\ParabolicSubgroup_{(n_1,n_2)}}\left(\diag\left(h_1, h_2\right)\right) = \abs{\det h_1}^{n_2} \cdot \abs{\det h_2}^{-n_1}.$$

\section{Speh representations}\label{sec:speh-representations}

This section is devoted to explaining the notion of Speh representations of general linear groups in the context of finite fields and of non-archimedean local fields, and to establishing a relation between these two notions.

\subsection{Speh representation for cuspidal representations over finite fields}
Let $\finiteField$ be a field with $q$ elements. We recall the definition of an important class of irreducible representations of $\GL_{kc}\left(\finiteField\right)$ called Speh representations. These are also known in the literature as ``generalized trivial representations''. Our discussion is based on \cite[Section 4]{SilbergerZink00}.

Let $\tau$ be an irreducible cuspidal representation of $\GL_k\left(\finiteField\right)$. Denote $\SpehRepresentation{\tau}{1} = \tau$. When $c > 1$, the parabolic induction $\tau^{\circ c} = \tau \circ \dots \circ \tau$, a representation of $\GL_{kc}\left(\finiteField\right)$, is not irreducible. The equivalence classes of irreducible subrepresentations of $\tau^{\circ c}$ are in bijection with the (equivalence classes of) irreducible representations of the symmetric group on $c$ elements $\SymmetricGroup_c$, which in turn are in bijection with partitions of $c$. There are two irreducible subrepresentations of $\tau^{\circ c}$ that appear with multiplicity one. These are the Steinberg representation $\Steinberg\left(\tau, c\right)$, which corresponds to the partition $\left(c\right)$ and the Speh representation $\SpehRepresentation{\tau}{c}$, which corresponds to the partition $\left(1^c\right)$. The Steinberg representation $\Steinberg\left(\tau, c\right)$ is the irreducible subrepresentation of $\tau^{\circ c}$ with maximal dimension, while the Speh representation $\SpehRepresentation{\tau}{c}$ is the irreducible subrepresentation of $\tau^{\circ c}$ with minimal dimension (to show this, use the fact that the dimension of the irreducible subrepresentation of $\tau^{\circ}$ corresponding to the partition $\lambda \vdash c$ is a scalar multiple of the unipotent representation of $\GL_{c}\left(\finiteField_k\right)$ corresponding to $\lambda$ (see \cite[Corollary 5.3 (ii)]{Howe1985} or \cite[Lemma 7.4]{Green55}), and then use Corollary 3.4 or Proposition 3.5 of \cite{larsen2013largest}).

It turns out that if $c = c_1 + c_2$, then the Speh representation $\SpehRepresentation{\tau}{c}$ is a subrepresentation of the parabolic induction $\SpehRepresentation{\tau}{c_1} \circ \SpehRepresentation{\tau}{c_2}$.

We remark that in \cite{Carmon2023}, the Speh representation $\SpehRepresentation{\tau}{c}$ was defined for any irreducible generic representation $\tau$. See also \cite[Section 2.4]{CarmonZelingher2024}. However, for our purposes we only need to consider cuspidal representations.

\subsubsection{Projection operator}\label{subsec:speh-projection-operator}
By \cite[Proposition 4.2]{SilbergerZink00}, we have the following projection operator $\ProjectionOperator_{\SpehRepresentation{\tau}{c}} \colon \tau^{\circ c} \to \SpehRepresentation{\tau}{c}$:
$$ \ProjectionOperator_{\SpehRepresentation{\tau}{c}} = \frac{1}{\PoincarePolynomial{k}{c}\left(q^k\right)} \sum_{w \in \SymmetricGroup_c} h^0_w,$$
where $$\PoincarePolynomial{k}{c}\left(X\right) = \left(1+X\right)\left(1 + X + X^2\right) \dots \left(1+ X + \dots + X^{c-1}\right)$$ is the Poincare polynomial and $h^0_w$ is given by the Howe isomorphism. 
We explain the last point now. 
For any permutation $w \in \SymmetricGroup_c$, we may realize it as a $kc \times kc$ block matrix, with blocks of size $k$, in the following way. 
First represent $w$ as $c \times c$ matrix, and then replace each zero entry in the matrix with the matrix $0_k$, and each 1 entry in the matrix with $\IdentityMatrix{k}$. 
Now let $h_w \colon \tau^{\circ c} \to \tau^{\circ c}$ be the operator corresponding to the characteristic function of the double coset $\ParabolicForSpeh{k}{c} w \ParabolicForSpeh{k}{c}$ by Mackey theory, that is, for $f \in \tau^{\circ c}$ and $x \in \GL_{kc}\left(\finiteField\right)$,
$$ \left(h_w f\right)\left(x\right) = \sum_{g \in (\ParabolicForSpeh{k}{c} w \ParabolicForSpeh{k}{c}) \slash \ParabolicForSpeh{k}{c}} f^w \left(g^{-1} x\right) = \sum_{u \in \UnipotentForSpeh{k}{c} \slash  (\UnipotentForSpeh{k}{c} \cap w \UnipotentForSpeh{k}{c} w^{-1}) } f^w \left(w^{-1} u^{-1} x\right),$$
where for $w \in \SymmetricGroup_c$ and $g \in \GL_{kc}\left(\finiteField\right)$, we denote $f^{w}\left(g\right) = w f\left(g\right)$, where $\SymmetricGroup_c$ acts on $\tau^{\otimes c}$ by permuting the components of pure tensors. 
We then normalize $h_w$ as Howe \cite[Page 12]{Howe1985}, by setting $$h_w^0 = q^{-\frac{k\left(k-1\right)}{2} \ell\left(w\right)} \centralCharacter{\tau}\left(-1\right)^{\ell \left(w\right)} h_w,$$ where $\centralCharacter{\tau}$ is the central character of $\tau$ and $\ell \left(w\right)$ is the length of the permutation $w$.

For every $1 \le i \le c - 1$, let $s_i = \left(i, i+1\right) \in \SymmetricGroup_c$ be a simple transposition. Then we have the following Hecke algebra relations \cite[Theorem 5.1]{Howe1985}:
\begin{enumerate}
	\item For every $1 \le i, j \le c$ with $\abs{i - j} > 1$, we have $h^0_{s_i} \circ h^0_{s_j} = h^0_{s_j} \circ h^0_{s_i}$.
	\item For every $0 \le i \le c - 2$, we have $h^0_{s_{i}} \circ h^0_{s_{i+1}} \circ h^0_{s_{i}} = h^0_{s_{i+1}} \circ h^0_{s_{i}} \circ  h^0_{s_{i+1}}$.
	\item $h^0_{s_i} \circ h^0_{s_i} = q^k \idmap_{\tau^{\circ c}} + (q^k - 1)h^0_{s_i}$.
\end{enumerate}

By \cite[Lemma 3.1]{guizzi2010cyclic}, we may decompose $\ProjectionOperator_{\SpehRepresentation{\tau}{c}}$ in the following way:
\begin{proposition}\label{prop:projection-operator-as-product-of-simple-reflections}
	$$\ProjectionOperator_{\SpehRepresentation{\tau}{c}} = \frac{1}{\PoincarePolynomial{k}{c}\left(q^k\right)} \prod_{\substack{\left(i,j\right)\\
	1 \le i < j \le c}} \left(h^0_{s_{j-i}} + \frac{q^k - 1}{q^{\left(j-i\right)k}-1} \cdot \idmap_{\tau^{\circ c}} \right),$$
where the product is ordered by the lexicographic ordering, taken either right to left or left to right.
\end{proposition}

\begin{remark}\label{rem:steinberg-representation-projection-operator}
	By \cite[Proposition 4.2]{SilbergerZink00}, the projection operator $\ProjectionOperator_{\SteinbergRepresentation{\tau}{c}} \colon \tau^{\circ c} \to \SteinbergRepresentation{\tau}{c}$ is given by $$\ProjectionOperator_{\SteinbergRepresentation{\tau}{c}} = \frac{1}{\PoincarePolynomial{k}{c}\left(q^{-k}\right)} \sum_{w \in \SymmetricGroup_c} \left(-q^{-k}\right)^{\ell\left(w\right)} h^0_w.$$
	Consider the automorphism of $\EndomorphismRing\left(\tau^{\circ c}\right)$ that sends $h^0_{w}$ to $\left(-q^k\right)^{\ell\left(w\right)} \left(h^0_{w^{-1}}\right)^{-1}$ for every $w \in \SymmetricGroup_c$. By \cite[Lemma 2.5]{Murphy1995}, this automorphism sends $\PoincarePolynomial{k}{c}\left(q^k\right) \ProjectionOperator_{\SpehRepresentation{\tau}{c}}$ to $q^{k \binom{c}{2}} \PoincarePolynomial{k}{c}\left(q^{-k}\right) \ProjectionOperator_{\SteinbergRepresentation{\tau}{c}}$, and therefore sends $\ProjectionOperator_{\SpehRepresentation{\tau}{c}}$ to $\ProjectionOperator_{\SteinbergRepresentation{\tau}{c}}$. Applying this automorphism to both sides of \Cref{prop:projection-operator-as-product-of-simple-reflections} and using the relation  $$-q^k \left(h^0_{s_i}\right)^{-1} =  -h^0_{s_i} + \left(q^k - 1\right) \cdot \idmap_{\tau^{\circ c}},$$ we get
	$$ \ProjectionOperator_{\SteinbergRepresentation{\tau}{c}} = \frac{\left(-1\right)^{\binom{c}{2}}}{\PoincarePolynomial{k}{c}\left(q^{k}\right)} \prod_{\substack{\left(i,j\right)\\ 1 \le i < j \le c
	}} \left(h_{s_{j-i}}^0 + \frac{q^k - 1}{q^{\left(i-j\right)k} - 1}\cdot \idmap_{\tau^{\circ c}}\right).$$	
\end{remark}

\subsection{Relation to Speh representations of level zero representations}\label{subsec:speh-representations-of-depth-zero-representations}

Let $\localField$ be a non-archimedean local field with ring of integers $\ringOfIntegers$, maximal ideal $\maximalIdeal$ and residue field $\finiteField$. Let $\quotientMap \colon \ringOfIntegers  \to \finiteField$ be the quotient map. Let $\uniformizer \in \maximalIdeal$ be a uniformizer.

Let $\depthZeroRepresentation$ be an irreducible supercuspidal representation of $\GL_{k}\left(\localField\right)$. Given an integer $c \ge 1$, one may construct an irreducible representation of $\GL_{kc}\left(\localField\right)$, called the Speh representation and denoted $\SpehRepresentation{\depthZeroRepresentation}{c}$. In this section, we explain the relation between Speh representations of $\GL_{kc}\left(\localField\right)$ and of $\GL_{kc}\left(\finiteField\right)$ using level zero supercuspidal representations.

\subsubsection{Speh representations over local fields}
We first recall the definition of the Speh representation $\SpehRepresentation{\depthZeroRepresentation}{c}$. These were introduced by Jacquet in \cite[Section 2.2]{Jacquet1984}. We follow \cite[Section 2.2]{CaiFriedbergGourevitchKaplan2023}. Given $z_1, \dots, z_c \in \cComplex$, we consider the (normalized) parabolically induced representation
\begin{equation}\label{eq:parabolic-induction-space-for-Speh}
	\holomorphicRepresentation = \Ind{\ParabolicForSpeh{k}{c}}{\GL_{kc}\left(\localField\right)}{\abs{\det}^{z_1}\depthZeroRepresentation \otimes \dots \otimes \abs{\det}^{z_c} \depthZeroRepresentation}.
\end{equation}

A \emph{section} $\parabolicSection$ for the family $\holomorphicRepresentation$ is an assignment $\Omega \times \GL_{kc}\left(\localField\right) \to \depthZeroRepresentation^{\otimes c},$ denoted by $\left(z_1,\dots, z_c, g\right) \mapsto \parabolicSection\left(g\right)$, such that for any $\left(z_1,\dots,z_c\right)$ in an open set $\Omega$ the assignment $g \mapsto \parabolicSection\left(g\right)$ is an element of $\holomorphicRepresentation$. A section $\parabolicSection$ is called \emph{a flat section} if $\Omega = \cComplex^c$ and if for any $k_0 \in \GL_{kc}\left(\ringOfIntegers\right)$, the assignment $\left(z_1,\dots,z_c\right) \mapsto \parabolicSection\left(k_0\right)$ does not depend on $z_1, \dots, z_c$. A section $\parabolicSection$ is called a \emph{holomorphic section} if it can be written as a (finite) $\cComplex\left[q^{\pm z_1}, \dots, q^{\pm z_c}\right]$-linear combination of flat sections. A section $\parabolicSection$ is called \emph{a meromorphic section} if it can be written as a (finite) $\cComplex \left(q^{- z_1}, \dots, q^{- z_c}\right)$-linear combination of flat sections.

Given a flat section $\parabolicSection$ and $g \in \GL_{kc}\left(\localField\right)$, we consider the integral \begin{equation}\label{eq:intertwining-operator-local-field-definition}
\intertwiningOperator_{\weylElement{1^c}} \parabolicSection\left(g\right) = \weylElement{1^c} \int_{\UnipotentForSpeh{k}{c}} \parabolicSection \left( \weylElement{k^c} u g \right) \differential u,
\end{equation}
where $\weylElement{1^c}$ is the longest Weyl element in $\SymmetricGroup_c$ and $\weylElement{k^c} \in \GL_{kc}\left(\localField\right)$ is the matrix representing $\weylElement{1^c}$, where we replaced every zero entry by $0_k$ and every $1$ entry by $\IdentityMatrix{k}$. Here, $\weylElement{1^c}$ acts on $\depthZeroRepresentation^{\otimes c}$ by permuting components of pure tensors.

It is well-known that the integral \eqref{eq:intertwining-operator-local-field-definition} converges absolutely for $\left(z_1,\dots,z_c\right)$ in a cone and has a meromorphic continuation to $\cComplex^c$ \cite[Theorem IV.1.1]{waldspurger2003formule}, by which we mean that it results in a (finite) $\cComplex\left(q^{-z_1},\dots,q^{-z_c}\right)$-linear combination of elements in $\depthZeroRepresentation^{\otimes c}$. We keep denoting the meromorphic continuation of \eqref{eq:intertwining-operator-local-field-definition} by the same symbol. Then in its domain of definition, $\intertwiningOperator_{\weylElement{1^c}}$ defines an intertwining operator $$\holomorphicRepresentation \to \depthZeroRepresentation^{\left(z_c,\dots,z_1\right)}.$$
It is also well-known that $\intertwiningOperator_{\weylElement{1^c}}$ is holomorphic at the point $$\left(z_1,\dots,z_c\right) = \left(\frac{c-1}{2}, \frac{c-3}{2}, \dots, - \frac{c-1}{2}\right).$$ The Speh representation $\SpehRepresentation{\depthZeroRepresentation}{c}$ is defined as the image of $M^{\left(\frac{c-1}{2},\frac{c-3}{2},\dots,-\frac{c-1}{2}\right)}_{\weylElement{1^c}}$.

It turns out that if $c = c_1 + c_2$, then $\SpehRepresentation{\depthZeroRepresentation}{c}$ is a subrepresentation of the parabolically induced representation $\abs{\det}^{-\frac{c_2}{2}} \SpehRepresentation{\depthZeroRepresentation}{c_1} \times \abs{\det}^{\frac{c_1}{2}} \SpehRepresentation{\depthZeroRepresentation}{c_2}$.

\begin{remark}\label{rem:steinberg-representation-image-of-intertwining}
	Similarly, it is also well-known that $\intertwiningOperator_{\weylElement{1^c}}$ is holomorphic at the point $$\left(z_1,\dots,z_c\right) = \left(-\frac{c-1}{2}, -\frac{c-3}{2}, \dots,  \frac{c-1}{2}\right).$$
	The Steinberg representation $\SteinbergRepresentation{\depthZeroRepresentation}{c}$ is defined as the image of $M_{\weylElement{1^c}}^{\left(-\frac{c-1}{2},-\frac{c-3}{2},\dots, \frac{c-1}{2}\right)}$.
\end{remark}
We remark that the definition of the Speh representation $\SpehRepresentation{\depthZeroRepresentation}{c}$ given above makes sense for any essentially unitary irreducible representation $\depthZeroRepresentation$. We also remark that in \cite{CaiFriedbergGourevitchKaplan2023} and \cite{LapidMao2020} a generalized Speh representation $\SpehRepresentation{\depthZeroRepresentation}{c}$ was defined for any irreducible generic representation $\depthZeroRepresentation$. However, for our purposes we only need to consider supercuspidal representations.

\subsubsection{Level zero representations}

Let $\tau$ be an irreducible cuspidal representation of $\GL_k\left(\finiteField\right)$. Let $\chi \colon \multiplicativegroup{\localField} \to \multiplicativegroup{\cComplex}$ be a character, such that $\chi \restriction_{\multiplicativegroup{\ringOfIntegers}} = \centralCharacter{\tau} \circ \quotientMap \restriction_{\multiplicativegroup{\ringOfIntegers}}$. Consider the following representation of $\multiplicativegroup{\localField} \cdot \GL_k\left(\ringOfIntegers\right)$:
$$ \left(\chi \otimes \tau \right)\left(z k_0\right) = \chi\left(z\right) \tau\left(\quotientMap\left(k_0\right)\right),$$
for any $z \in \multiplicativegroup{\localField}$ and $k_0 \in \GL_k\left(\ringOfIntegers\right)$. Then the compactly induced representation $$\depthZeroRepresentation = \ind{\multiplicativegroup{\localField} \cdot \GL_k\left(\ringOfIntegers\right)}{\GL_k\left(\localField\right)}{\chi \otimes \tau}$$ is an irreducible supercuspidal representation of $\GL_k\left(\localField\right)$ \cite[Theorem 6.2]{DipendraRaghuram08}. We say that $\depthZeroRepresentation$ is an \emph{irreducible level zero supercuspidal representation of $\GL_k\left(\localField\right)$ constructed from $\tau$ with central character $\chi$}. Henceforth we will always write $\centralCharacter{\depthZeroRepresentation}$ instead of $\chi$.

Given an element $v_{\tau} \in \tau$, we have an element $\Lift{v_{\tau}} \in \depthZeroRepresentation$ given by the function supported on $\multiplicativegroup{\localField} \cdot \GL_k\left(\ringOfIntegers\right)$, satisfying $$\left(\Lift{v_{\tau}}\right)\left(zk_0\right) = \centralCharacter{\depthZeroRepresentation}\left(z\right) \tau\left(\nu\left(k_0\right)\right) v_{\tau},$$
for any $z \in \multiplicativegroup{\localField}$ and $k_0 \in \GL_k\left(\ringOfIntegers\right)$.

Given vectors $v_{\tau,1}, \dots, v_{\tau,c} \in \tau$, we define $$\Lift\left({v_{\tau, 1} \otimes \dots \otimes v_{\tau,c}} \right) = \Lift{v_{\tau,1}} \otimes \dots \otimes  \Lift{v_{\tau,c}} \in \depthZeroRepresentation \otimes \dots \otimes \depthZeroRepresentation.$$
We extend the definition of $\Lift{v_{\tau^{\otimes c}}}$ for every $v_{\tau^{\otimes c}} \in \tau^{\otimes c}$ using linearity.

Let $f \in \tau^{\circ c}$. We define a flat section $$\Lift^{\left(z_1,\dots,z_c\right)} \left(f\right)\in \holomorphicRepresentation$$ as follows. First notice that the assignment $\GL_{kc}\left(\ringOfIntegers\right) \to \depthZeroRepresentation^{\otimes c}$ given by $k_0 \mapsto \Lift f\left(\quotientMap\left(k_0\right)\right)$ defines an element of $\Ind{\ParabolicForSpeh{k}{c} \cap \GL_{kc}\left(\ringOfIntegers\right)}{\GL_{kc}\left(\ringOfIntegers\right)}{\depthZeroRepresentation^{\otimes c}}$. Hence, by the Iwasawa decomposition $$\GL_{kc}\left(\localField\right) = \ParabolicForSpeh{k}{c} \cdot \GL_{kc}\left(\ringOfIntegers\right),$$ we may lift this element uniquely to a flat section $\Lift^{\left(z_1,\dots,z_c\right)} \left(f\right) \in \holomorphicRepresentation$ that satisfies $$\Lift^{\left(z_1,\dots,z_c\right)}\left(f\right)\left(k_0\right) = \Lift\left({f\left(\nu\left(k_0\right)\right)}\right),$$ for any $k_0 \in \GL_{kc}\left(\ringOfIntegers\right)$. 

\subsubsection{Simple reflection intertwining operators}

Let $\parabolicSection \in \holomorphicRepresentation$ be a flat section. For $1 \le i \le c - 1$, let $s_i = \left(i, i+1\right) \in \SymmetricGroup_c$ be a transposition. For an element $g \in \GL_{2k}\left(\localField\right)$ denote $$d_i\left(g\right) = \diag\left(\IdentityMatrix{\left(i-1\right)k}, g, \IdentityMatrix{\left(c-i-1\right)k}\right).$$ Consider the following intertwining operator 
\begin{equation}\label{eq:simple-reflection-intertwining-operator-local-field-definition}
	\begin{split}
			&\left(\intertwiningOperator_{s_i} \parabolicSection\right)\left(g\right)	= 		s_i \int_{\squareMatrix_k\left(\localField\right)} \parabolicSection \left( d_i\left(\begin{pmatrix}
			& \IdentityMatrix{k}\\
			\IdentityMatrix{k}
		\end{pmatrix} \begin{pmatrix}
		\IdentityMatrix{k} & X\\
		& \IdentityMatrix{k}
	\end{pmatrix} \right) g\right) \differential X.
	\end{split}
\end{equation}
Once again, it is known that the right hand side of \eqref{eq:simple-reflection-intertwining-operator-local-field-definition} converges absolutely for $\left(z_1,\dots,z_c\right)$ in a cone and has a meromorphic continuation to $\cComplex^c$ \cite[Theorem IV.1.1]{waldspurger2003formule}. We keep denoting the meromorphic continuation by the same symbol. It defines an intertwining operator $$\intertwiningOperator_{s_i} \colon \holomorphicRepresentation \to \depthZeroRepresentation^{s_i \left(z_1, \dots, z_c\right)},$$
where $s_i$ acts on $\left(z_1, \dots, z_c\right)$ by swapping the $i$-th and the $i+1$-th components.

Recall the following decompositions of the long Weyl element in $\SymmetricGroup_c$ \cite[Equation (2.3) and below]{guizzi2010cyclic}: $$\left(s_1\right) \left(s_2 s_1\right) \dots \left(s_{c-1} s_{c-2} \dots s_{1}\right) = \left(s_1 s_2 \dots s_{c-1} \right) \dots \left(s_1 s_2\right) s_1 .$$
Using either of these decompositions, we may decompose $\intertwiningOperator_{\weylElement{1^c}}$ as a composition of operators $M_{s_j}$ corresponding to the decomposition we choose, where the complex parameters are inferred.

Let us normalize (in this subsection only) the Haar measure so that $\squareMatrix_k\left( \maximalIdeal \right)$ has volume $q^{-\binom{k}{2}}$. This makes the constants similar to the ones in \Cref{subsec:speh-projection-operator}.

\begin{lemma}\label{lem:simple-reflection-intertwining-operator}
	For any $f \in \tau^{\circ c}$ we have 
	$$ \intertwiningOperator_{s_i} \Lift^{\left(z_1,\dots,z_c\right)}\left(f\right) = q^{-\binom{k}{2}} \Lift^{s_i \left(z_1,\dots,z_c\right)}\left({h_{s_i} f}\right) + \centralCharacter{\tau}\left(-1\right)  \frac{q^k - 1}{q^{k\left(z_i - z_{i+1}\right)} - 1} \Lift^{s_i \left(z_1,\dots,z_c\right)}\left(f\right).$$
	Here, the integral defining $\intertwiningOperator_{s_i}$ converges absolutely for $\RealPart z_i > \RealPart z_{i+1}$ and the equality is understood elsewhere as an equality of its meromorphic continuation.
\end{lemma}
\begin{remark}
	Using transitivity of induction, it suffices to prove this for $c=2$. However, we chose to stick to the notation $d_i$ defined above instead of reducing to $c=2$.
\end{remark}
\begin{proof}
	Throughout the proof, we will use the notation $\tau_i\left(h_i\right) \otimes \tau_{i+1}\left(h_{i+1}\right)$, where $h_i, h_{i+1} \in \GL_k\left(\finiteField\right)$ to mean the linear map $$\idmap_{\tau}^{\otimes (i-1)} \otimes \tau\left(h_i\right) \otimes \tau\left(h_{i+1}\right) \otimes {\idmap_{\tau}}^{\otimes (k - i - 1)} \colon \tau^{\otimes c} \to \tau^{\otimes c},$$ and similarly for the notation $\depthZeroRepresentation_i\left(h_i\right) \otimes \depthZeroRepresentation_{i+1}\left(h_{i+1}\right)$, where $h_i, h_{i+1} \in \GL_k\left(\localField\right)$.
	
	By the Iwasawa decomposition $\GL_{kc}\left(\localField\right) = \ParabolicForSpeh{k}{c} \cdot \GL_{kc}\left(\ringOfIntegers\right)$, it suffices to prove the equality for elements in $\GL_{kc}\left(\ringOfIntegers\right)$. Furthermore, since for any $k_0 \in \GL_{kc}\left( \ringOfIntegers \right)$ we have $$\Lift^{\left(z_1,\dots,z_c\right)}\left( \tau^{\circ c}\left(\nu\left(k_0\right)\right) f\right) = \holomorphicRepresentation\left(k_0\right) \Lift^{\left(z_1,\dots,z_c\right)}\left(f\right),$$ it suffices to prove the equality at $\IdentityMatrix{kc}$.
	Let $\parabolicSection = \Lift^{\left(z_1,\dots,z_c\right)}\left(f\right)$.
	Consider the maximum norm on $\squareMatrix_k \left(\localField\right)$, defined by $\Norm{\left(x_{ij}\right)} = \max \abs{x_{ij}}_{i,j}$. We split the integral in \eqref{eq:simple-reflection-intertwining-operator-local-field-definition} into two parts:
	The first integral is \begin{equation}\label{eq:first-integral-depth-zero-simple-intertwining}
		s_i \int_{\substack{\squareMatrix_k\left(\localField\right)\\
				\Norm{X} \le 1}} \parabolicSection \left(d_i\left(\begin{pmatrix}
			& \IdentityMatrix{k}\\
			\IdentityMatrix{k}
		\end{pmatrix} \begin{pmatrix}
			\IdentityMatrix{k} & X\\
			& \IdentityMatrix{k}
		\end{pmatrix} \right)\right) \differential X
	\end{equation} and the second integral is 
\begin{equation}
	\label{eq:second-integral-depth-zero-simple-intertwining}
	s_i \int_{\substack{\squareMatrix_k\left(\localField\right)\\
			\Norm{X} > 1}} \parabolicSection \left(d_i\left(\begin{pmatrix}
		& \IdentityMatrix{k}\\
		\IdentityMatrix{k}
	\end{pmatrix} \begin{pmatrix}
		\IdentityMatrix{k} & X\\
		& \IdentityMatrix{k}
	\end{pmatrix} \right)\right) \differential X.
\end{equation}

Regarding the first integral, since $\Norm{X} \le 1$, the matrices appearing in the argument of $\parabolicSection$ all lie in $\GL_{kc}\left(\ringOfIntegers\right)$, and therefore by definition we get the value $$s_i \int_{\substack{\squareMatrix_k\left(\localField\right)\\
		\Norm{X} \le 1}} \Lift \left(f  \left(d_i\left(\begin{pmatrix}
	& \IdentityMatrix{k}\\
\IdentityMatrix{k}
\end{pmatrix} \nu \begin{pmatrix}
\IdentityMatrix{k} & X\\
& \IdentityMatrix{k}
\end{pmatrix} \right)\right)\right) \differential X.$$
Since the integrand is invariant under $\squareMatrix_c\left(\maximalIdeal\right)$ translations, we get by the isomorphism $\quotientMap \colon \squareMatrix_k\left(\ringOfIntegers\right) \slash \squareMatrix_k\left( \maximalIdeal \right) \cong \squareMatrix_k\left(\finiteField\right) $ that \eqref{eq:first-integral-depth-zero-simple-intertwining} contributes \begin{equation*}
	q^{-\binom{k}{2}} s_i \sum_{X \in \squareMatrix_k\left(\finiteField\right)} \Lift \left(f  \left(d_i\left(\begin{pmatrix}
		& \IdentityMatrix{k}\\
		\IdentityMatrix{k}
	\end{pmatrix} \begin{pmatrix}
		\IdentityMatrix{k} & X\\
		& \IdentityMatrix{k}
	\end{pmatrix}\right) \right)\right) = q^{-\binom{k}{2}} \Lift\left( h_{s_i} f \right)\left(\IdentityMatrix{kc}\right).
\end{equation*}

Regarding the second part, since $\GL_k\left(\localField\right)$ is open dense in $\squareMatrix_{k}\left(\localField\right)$, we may integrate over $\GL_k\left(\localField\right)$. We replace the additive Haar measure with the multiplicative Haar measure by using the relation $\differential X = \abs{\det X}^k \mdifferential X$, where we normalize correctly, so that $1 + \squareMatrix_k\left(\maximalIdeal\right)$ has volume $q^{-\binom{k}{2}}$. Then \eqref{eq:second-integral-depth-zero-simple-intertwining} becomes
\begin{equation}\label{eq:second-integral}
	s_i \int_{\substack{{\GL_k\left(\localField\right)}\\
	\Norm{X} > 1}} \parabolicSection \left(d_i\left(\begin{pmatrix}
& \IdentityMatrix{k}\\
\IdentityMatrix{k}
\end{pmatrix} \begin{pmatrix}
\IdentityMatrix{k} & X\\
& \IdentityMatrix{k}
\end{pmatrix} \right)\right) \abs{\det X}^k \mdifferential X.
\end{equation}
For $X \in \GL_k\left(\localField\right)$ let $X = k_{1,X} a_X k_{2,X}$ be a Cartan decomposition, where $k_1 = k_{1,X}, k_2 = k_{2,X} \in \GL_k\left(\ringOfIntegers\right)$ and $a_X = \diag\left(\uniformizer^{m_1}, \dots, \uniformizer^{m_k} \right)$, where $m_1 \ge \dots \ge m_k$. Let $0 \le r \le k$ be such that the diagonal entries of $a_{X}^{+} = \diag\left(\uniformizer^{m_1},\dots,\uniformizer^{m_r}\right)$ all have $m_j \ge 0$ and such that the diagonal entries of $a_{X}^{-} = \diag\left(\uniformizer^{m_{r+1}},\dots,\uniformizer^{m_k}\right)$ all have $m_j < 0$. Then we have the following decomposition, where the first two matrices on the second line lie in $\ParabolicForSpeh{k}{c}$ and the last two matrices on that line lie in $\GL_{kc}\left(\ringOfIntegers\right)$: \begin{align*}
	&\begin{pmatrix}
		& \IdentityMatrix{k}\\
		\IdentityMatrix{k}
	\end{pmatrix} \begin{pmatrix}
		\IdentityMatrix{k} & X\\
		& \IdentityMatrix{k}
	\end{pmatrix} \\
&= \begin{pmatrix}
		k_2^{-1}\\
		& k_1
	\end{pmatrix}\begin{pmatrix}
	\IdentityMatrix{r}\\
	&	-\left(a_X^-\right)^{-1} & & \IdentityMatrix{k-r} \\
		& & \IdentityMatrix{r}\\
		& & & a_X^{-}
	\end{pmatrix}
\begin{pmatrix}
	& & \IdentityMatrix{r}\\
	& \IdentityMatrix{k-r}\\
	\IdentityMatrix{r} & & a_X^{+}\\
	& \left(a_X^{-}\right)^{-1} & & \IdentityMatrix{k-r}
\end{pmatrix} \begin{pmatrix}
	k_1^{-1 }\\
	& k_2
\end{pmatrix}.
\end{align*}

Hence, we get
\begin{equation} \label{eq:holomorphic-section-evaluation-after-cartan-decomposition}
\begin{split}
		& \parabolicSection \left(d_i\left(\begin{pmatrix}
		& \IdentityMatrix{k}\\
		\IdentityMatrix{k}
	\end{pmatrix} \begin{pmatrix}
		\IdentityMatrix{k} & X\\
		& \IdentityMatrix{k}
	\end{pmatrix} \right)\right) \\
	= & \abs{\det a_X^{-}}^{-k+z_{i+1}-z_i} \depthZeroRepresentation_i\left(k_2^{-1} \begin{pmatrix}
		\IdentityMatrix{r} &\\
		& -\left(a_X^{-}\right)^{-1}
	\end{pmatrix} \right) \otimes \depthZeroRepresentation_{i+1}\left(k_1 \begin{pmatrix}
		\IdentityMatrix{r} \\
		& a_X^{-}
	\end{pmatrix} \right) \\
	& \times \Lift\left( \left(f \circ \quotientMap \circ d_i\right) \left(\begin{pmatrix}
		& & \IdentityMatrix{r}\\
		& \IdentityMatrix{k-r}\\
		\IdentityMatrix{r} & & a_X^{+}\\
		&  & & \IdentityMatrix{k-r}
	\end{pmatrix} \begin{pmatrix}
		k_1^{-1 }\\
		& k_2
	\end{pmatrix}\right)\right).
\end{split}
\end{equation}

We will first evaluate \eqref{eq:holomorphic-section-evaluation-after-cartan-decomposition} at $\left(\IdentityMatrix{k},\dots, \IdentityMatrix{k}\right)$. This amounts to evaluating a function $\Lift \left( v_{\tau^{\otimes c}} \right)$ for some $v_{\tau^{\otimes c}} \in \tau^{\otimes c}$ at the point $$\left(\IdentityMatrix{k}, \dots, \IdentityMatrix{k}, k_2^{-1} \begin{pmatrix}
	\IdentityMatrix{r} \\
	& -\left(a_X^-\right)^{-1}
\end{pmatrix}, k_1
\begin{pmatrix}
\IdentityMatrix{r}\\
& a_X^-
\end{pmatrix}, \IdentityMatrix{k}, \dots, \IdentityMatrix{k} \right).$$
Since $\Lift \left( v_{\tau^{\otimes c}} \right)$ is only supported on $\prod_{j=1}^c \left(\multiplicativegroup{\localField} \cdot \GL_{k}\left(\localField\right)\right)$, we must have that $\left(\begin{smallmatrix}
	\IdentityMatrix{r} &\\
	& a_X^{-}
\end{smallmatrix}\right)$ is a scalar matrix. Therefore, if $\Norm{X} > 1$ then we must have $r = 0$ and $a_X = \uniformizer^{-j} \IdentityMatrix{kc}$ for some $j \ge 1$, and we may assume $k_2 = \IdentityMatrix{kc}$. Then the integral \eqref{eq:second-integral} evaluated at $\left(\IdentityMatrix{k},\dots,\IdentityMatrix{k}\right)$ becomes
\begin{align*}
	&s_i \sum_{j = 1}^{\infty} q^{kj \left(z_{i+1} - z_i\right)} \int_{{\GL_k\left(\ringOfIntegers\right)}} \tau_i \left( -\quotientMap\left(k_1\right)^{-1} \right) \otimes \tau_{i+1}\left(\quotientMap\left(k_1\right)\right) f \left( \IdentityMatrix{kc} \right) \mdifferential k_1 \\
	&= \frac{1}{q^{k\left(z_{i} - z_{i+1}\right)} - 1} \cdot s_i \int_{{\GL_k\left(\ringOfIntegers\right)}} \tau_i \left( -\quotientMap\left(k_1\right)^{-1} \right) \otimes \tau_{i+1}\left(\quotientMap\left(k_1\right)\right) f \left( \IdentityMatrix{kc} \right) \mdifferential k_1, 
\end{align*}
where the convergence is for $\RealPart\left( z_{i} - z_{i+1} \right) > 0$.

We are left to deal with the inner integral $$s_i \int_{{\GL_k\left(\ringOfIntegers\right)}} \tau_i \left( -\quotientMap\left(k_1\right)^{-1} \right) \otimes \tau_{i+1}\left(\quotientMap\left(k_1\right)\right) f \left( \IdentityMatrix{kc} \right) \mdifferential k_1.$$
Since the argument is invariant under $1 + \squareMatrix_k\left(\maximalIdeal\right)$, this evaluates to
$$ q^{-\binom{k}{2}} \centralCharacter{\tau}\left(-1\right) s_i \sum_{h \in \GL_k\left(\finiteField\right)} \tau_i\left(h^{-1}\right) \otimes \tau_{i+1}\left(h\right) f\left(I_{kc}\right).$$

By \Cref{cor:swap-map-for-cuspidal-representations}, we have the following equality of operators $\tau^{\otimes c} \to \tau^{\otimes c}$, $$s_i \sum_{h \in \GL_k\left(\finiteField\right)} \tau_i\left(h^{-1}\right) \otimes \tau_{i+1}\left(h\right) = q^{\binom{k}{2}} (q^k - 1) \cdot \idmap_{\tau^{\otimes c}}.$$
Thus, the integral \eqref{eq:second-integral-depth-zero-simple-intertwining} evaluated at $\left(\IdentityMatrix{k},\dots,\IdentityMatrix{k}\right)$ contributes
$$ \centralCharacter{\tau}\left(-1\right) \frac{q^k - 1}{q^{k\left(z_i - z_{i+1}\right)} - 1} f\left(I_{kc}\right).$$

To finish, we claim that the function defined by \eqref{eq:second-integral} is supported on $\prod_{j=1}^c \left(\multiplicativegroup{\localField} \cdot \GL_k\left(\ringOfIntegers\right)\right)$. It is clear that if $\left(h_1,\dots,h_c\right)$ is in the support of the function defined by \eqref{eq:holomorphic-section-evaluation-after-cartan-decomposition}, then for any $j \ne i, i+1$ we must have $h_j \in \multiplicativegroup{\localField} \cdot \GL_k\left(\ringOfIntegers\right)$. By replacing $f$ with a suitable right translation of it, we may assume henceforth, without loss of generality, that $h_j = \IdentityMatrix{k}$ for every $j \ne i, i+1$.

Suppose that $h_i$ or $h_{i+1}$ are not in $\multiplicativegroup{\localField} \cdot \GL_k\left(\ringOfIntegers\right)$. We want to show that \eqref{eq:second-integral} vanishes at $\left(h_1,\dots,h_c\right)$. Since the center acts by the central character, we may assume that $h_i$ and $h_{i+1}$ have Cartan decomposition with diagonal part $\left(\begin{smallmatrix}
	\IdentityMatrix{r_1}\\
	& b_1
\end{smallmatrix}\right)$ and $\left(\begin{smallmatrix}
b_2\\
& \IdentityMatrix{r_2}
\end{smallmatrix}\right)$, respectively, where $0 < r_1, r_2 \le k$ and $b_1 = \diag\left(\uniformizer^{m'_{r_1+1}}, \dots, \uniformizer^{m'_{k}}\right)$ and $b_2 = \diag\left(\uniformizer^{m''_{k}}, \dots, \uniformizer^{m''_{r_2 + 1}}\right)$, such that $\min\left(r_1, r_2\right) < k$ and $0 > m'_{r_1 + 1} \ge \dots \ge m'_{k}$ and $m''_{k} \ge \dots \ge m''_{r_2 + 1} > 0$. If $\left(h_1,\dots, h_c\right)$ is in the support of \eqref{eq:holomorphic-section-evaluation-after-cartan-decomposition} for some $X$, then we must have that \begin{equation}\label{eq:second-integral-cartan-decomposition-equality}
h_i k_2^{-1} \begin{pmatrix}
	\IdentityMatrix{r} &\\
	& -\left(a_X^{-}\right)^{-1}
\end{pmatrix} = t_1 k'_1 \text{ and } h_{i+1} k_1 \begin{pmatrix}
	\IdentityMatrix{r} &\\
	& a_X^{-}
\end{pmatrix} = t_2 k'_2
\end{equation} for some $t_1, t_2 \in \multiplicativegroup{\localField}$ and $k'_1, k'_2 \in \GL_{k}\left(\ringOfIntegers\right)$. Writing \begin{equation}\label{eq:second-integral-cartan-decomposition-equality-for-t-in-o-star}
h_i = k'_1 t_1 \begin{pmatrix}
	\IdentityMatrix{r} &\\
	& -a_X^{-}
\end{pmatrix} k_2 \text{ and } h_{i+1}  = k'_2 t_2 \begin{pmatrix}
\IdentityMatrix{r} &\\
& \left(a_X^{-}\right)^{-1}
\end{pmatrix} k_1^{-1}
\end{equation} and comparing the Cartan decomposition of both sides of \eqref{eq:second-integral-cartan-decomposition-equality-for-t-in-o-star}, we see that $t_1,t_2 \in \multiplicativegroup{\ringOfIntegers}$, so without loss of generality $t_1 = t_2 = 1$. In addition, \eqref{eq:second-integral-cartan-decomposition-equality-for-t-in-o-star} implies that the diagonal parts of the Cartan decomposition of $h_i$ and $h_{i+1}$ must be related, and that they determine $a_X^{-}$. More precisely, \begin{equation}\label{eq:second-integral-cartan-decomposition-diagonal-equality}
\begin{pmatrix}
	\IdentityMatrix{r} &\\
	& a_X^{-}
\end{pmatrix} = \begin{pmatrix}
	\IdentityMatrix{r_1}\\
	& b_1
\end{pmatrix} \text{ and } \begin{pmatrix}
	b_2 &\\
	& I_{r_2}
\end{pmatrix} = \begin{pmatrix}
	\weylElement{1^{k-r}} \left(a_X^{-}\right)^{-1} \weylElement{1^{k-r}}\\
	& \IdentityMatrix{r}
\end{pmatrix}
\end{equation} which implies that $1 \le r_1 = r_2 = r < k$ and $m_{j} = m'_{j} = -m''_{j}$ for every $r \le j \le k$, and hence $a_{X}^{-} = b_1$. Denote $b = b_1$. As we just explained above, the value $b$ is determined by the Cartan decomposition of $h_i$. By \eqref{eq:holomorphic-section-evaluation-after-cartan-decomposition} and \eqref{eq:second-integral-cartan-decomposition-diagonal-equality}, the evaluation of \eqref{eq:second-integral} at $\left(h_1,\dots,h_c\right)$ is given by \begin{equation}\label{eq:intertwining-operator-lemma-torus-integration}
\begin{split}
& \abs{\det b}^{z_{i+1}-z_i} \sum_{\substack{a^{+} = \diag\left(\uniformizer^{m_1},\dots, \uniformizer^{m_r}\right)\\
m_1 \ge m_2 \ge \dots \ge m_r \ge 0}} \frac{\VolumeOf \left(\GL_k\left(\ringOfIntegers\right) \begin{pmatrix}
	a^{+} &\\
	& b
\end{pmatrix} \GL_k\left(\ringOfIntegers\right)\right)}{\VolumeOf\left(\GL_k\left(\ringOfIntegers\right)\right)^2} \\
& \times \int_{\GL_k\left(\ringOfIntegers\right)} \int_{\GL_k\left(\ringOfIntegers\right)} \Lift\left( \left(f \circ \quotientMap \circ d_i\right) \left(\begin{pmatrix}
	& & \IdentityMatrix{r}\\
	& \IdentityMatrix{k-r}\\
	\IdentityMatrix{r} & & a^{+}\\
	&  & & \IdentityMatrix{k-r}
\end{pmatrix} \begin{pmatrix}
	k_1^{-1 }\\
	& k_2
\end{pmatrix}\right)\right) \\
& \left(\IdentityMatrix{k},\dots,\IdentityMatrix{k}, h_i k_2^{-1} \begin{pmatrix}
	\IdentityMatrix{r} &\\
	& -b^{-1}
\end{pmatrix}, h_{i+1} k_1 \begin{pmatrix}
\IdentityMatrix{r} \\
& b
\end{pmatrix}, \IdentityMatrix{k}, \dots, \IdentityMatrix{k} \right) \mdifferential k_1 \mdifferential k_2.
\end{split}
\end{equation}

Decomposing the integration over $k_1$ in \eqref{eq:intertwining-operator-lemma-torus-integration} through the subgroup $$N_b = \left\{\begin{pmatrix}
	\IdentityMatrix{r} & Y b^{-1}\\
	& \IdentityMatrix{k-r}
\end{pmatrix} \mid Y \in \Mat{r}{\left(k-r\right)}\left(\ringOfIntegers\right) \right\} \subset \IdentityMatrix{k} + \squareMatrix_k\left(\maximalIdeal\right)$$ yields an inner integral
\begin{equation}
	\begin{split}\label{eq:intertwining-operator-lemma-inner-integral}
		& \int_{Y \in \Mat{r}{\left(k-r\right)}\left(\ringOfIntegers\right)} \Lift\left( \left(f \circ \quotientMap \circ d_i\right) \left(\begin{pmatrix}
			& & \IdentityMatrix{r}\\
			& \IdentityMatrix{k-r}\\
			\IdentityMatrix{r} & & a^{+}\\
			&  & & \IdentityMatrix{k-r}
		\end{pmatrix} \begin{pmatrix}
			k_1^{-1 }\\
			& k_2
		\end{pmatrix}\right)\right) \\
		& \left(\IdentityMatrix{k},\dots,\IdentityMatrix{k}, h_i k_2^{-1} \begin{pmatrix}
			\IdentityMatrix{r} &\\
			& -b^{-1}
		\end{pmatrix}, h_{i+1} k_1 \begin{pmatrix}
			\IdentityMatrix{r} \\
			& b
		\end{pmatrix} \begin{pmatrix}
		\IdentityMatrix{r} & Y\\
		& \IdentityMatrix{k-r}
	\end{pmatrix}, \IdentityMatrix{k}, \dots, \IdentityMatrix{k} \right) \differential Y.
	\end{split}
\end{equation}
Since $\left(\begin{smallmatrix}
	\IdentityMatrix{r} & Y\\
	& \IdentityMatrix{k-r}
\end{smallmatrix}\right) \in \GL_k\left(\ringOfIntegers\right)$, we see that the argument of the last expression is in the support of $\Lift\left(v_{\tau^{\otimes c}}\right)$ for some $v_{\tau^{\otimes c}} \in \tau^{\otimes c}$ if and only if $k_1$ and $k_2$ are such that $$h_i k_2^{-1} \begin{pmatrix}
\IdentityMatrix{r} &\\
& -b^{-1}
\end{pmatrix}, h_{i+1} k_1 \begin{pmatrix}
\IdentityMatrix{r} &\\
& b
\end{pmatrix} \in \GL_k\left(\ringOfIntegers\right),$$
and in this case \eqref{eq:intertwining-operator-lemma-inner-integral} evaluates to a scalar multiple of
\begin{equation}
	\begin{split}
		& \tau_i \left(\quotientMap\left(h_i k_2^{-1} \begin{pmatrix}
			\IdentityMatrix{r} &\\
			& -b^{-1}
		\end{pmatrix}\right)\right) \otimes \tau_{i+1}\left(\quotientMap\left(h_{i+1} k_1 \begin{pmatrix}
		\IdentityMatrix{r} \\
		& b
	\end{pmatrix}\right)\right) \\
	& \times \sum_{Y \in \Mat{r}{\left(k-r\right)}\left(\finiteField\right)} \tau_{i+1}\begin{pmatrix}
		\IdentityMatrix{r} & Y\\
		& \IdentityMatrix{k-r}
	\end{pmatrix} \left(f \circ \quotientMap \circ d_i\right) \left(\begin{pmatrix}
			& & \IdentityMatrix{r}\\
			& \IdentityMatrix{k-r}\\
			\IdentityMatrix{r} & & a^{+}\\
			&  & & \IdentityMatrix{k-r}
		\end{pmatrix} \begin{pmatrix}
			k_1^{-1 }\\
			& k_2
		\end{pmatrix}\right).
	\end{split}
\end{equation}

But for $1 \le r \le k-1$ we have that $$\sum_{Y \in \Mat{r}{\left(k-r\right)}\left(\finiteField\right)} \tau\begin{pmatrix}
	\IdentityMatrix{r} & Y\\
	& \IdentityMatrix{k-r}
\end{pmatrix} = 0$$ because $\tau$ is cuspidal. Therefore, we proved that the evaluation of \eqref{eq:second-integral} is zero if $h_i$ or $h_{i+1}$ are not in $\multiplicativegroup{\localField} \GL_k\left(\ringOfIntegers\right)$.

\end{proof}

\begin{remark}
	For future purposes, it will be useful to evaluate the following more general intertwining operator. Let $\depthZeroRepresentation_1, \dots, \depthZeroRepresentation_c$ be irreducible level zero supercuspidal representations of $\GL_k\left(\localField\right)$ constructed from irreducible cuspidal representations $\tau_1, \dots, \tau_c$ of $\GL_k\left(\finiteField\right)$, respectively. Let $$\depthZeroRepresentation^{\left(z_1,\dots,z_c\right)}_{\left(\depthZeroRepresentation_1,\dots,\depthZeroRepresentation_c\right)} = \Ind{\ParabolicForSpeh{k}{c}}{\GL_{kc}\left(\localField\right)}{\abs{\det}^{z_1} \depthZeroRepresentation_1 \otimes \dots \otimes \abs{\det}^{z_c} \depthZeroRepresentation_c}.$$
	We can define the notion of lifts and also define intertwining operators $$h_{s_i} \colon \tau_1 \circ \dots \circ \tau_c \to \tau_1 \circ \dots \circ \tau_{i+1} \circ \tau_i \circ \dots \circ \tau_c.$$ Consider the intertwining operator $$\intertwiningOperator_{s_i} \colon \depthZeroRepresentation^{\left(z_1,\dots,z_c\right)}_{\left(\depthZeroRepresentation_1,\dots,\depthZeroRepresentation_c\right)} \to \depthZeroRepresentation^{s_i \left(z_1,\dots,z_c\right)}_{s_i \left(\depthZeroRepresentation_1,\dots,\depthZeroRepresentation_c\right)}$$
	defined by the same formula as \eqref{eq:simple-reflection-intertwining-operator-local-field-definition}. Then, similarly to the proof of \Cref{lem:simple-reflection-intertwining-operator}, we can show that for any $f \in \tau_1 \circ \dots \circ \tau_c$ we have 
	\begin{enumerate}
		\item If $\tau_i$ is not isomorphic to $\tau_{i+1}$, then
		$$ \intertwiningOperator_{s_i} \Lift^{\left(z_1,\dots,z_c\right)}\left(f\right) = q^{-\binom{k}{2}} \Lift^{s_i \left(z_1,\dots,z_c\right)}\left({h_{s_i} f}\right).$$
		\item If $\tau_i = \tau_{i+1}$, then
		\begin{align*}
			\intertwiningOperator_{s_i} \Lift^{\left(z_1,\dots,z_c\right)}\left(f\right) =& q^{-\binom{k}{2}} \Lift^{s_i \left(z_1,\dots,z_c\right)}\left({h_{s_i} f}\right) \\
			& + \centralCharacter{\tau_i}\left(-1\right)  \frac{q^k - 1}{\centralCharacter{\depthZeroRepresentation_i}\left(\uniformizer\right)^{-1} \centralCharacter{\depthZeroRepresentation_{i+1}}\left(\uniformizer\right) q^{k\left(z_i - z_{i+1}\right)} - 1} \Lift^{s_i \left(z_1,\dots,z_c\right)}\left(f\right),
		\end{align*}
		where for $1 \le j \le c$, $\centralCharacter{\depthZeroRepresentation_j}$ is the central character of $\depthZeroRepresentation_j$.
	\end{enumerate}
\end{remark}

By applying \Cref{lem:simple-reflection-intertwining-operator} repeatedly to the decomposition $$M_{\left(1^c\right)} =  \left(M_{s_1}\right) \circ \left(M_{s_2} \circ M_{s_1} \right) \circ \dots \circ  \left(M_{s_{c-1}} \circ \dots \circ M_{s_1}\right),$$ or to the decomposition $$M_{\left(1^c\right)} =   \left(M_{s_{1}} \circ \dots \circ M_{s_{c-1}}\right) \circ  \dots \circ \left(M_{s_1} \circ M_{s_2} \right) \circ \left(M_{s_1}\right),$$
where the complex parameters are inferred, we get the following theorem.

\begin{theorem}\label{thm:intertwining-operator-of-depth-zero-representation}
	For any $f \in \tau^{\circ c}$, we have
	$$\intertwiningOperator_{\weylElement{1^c}} \Lift^{\left(z_1,\dots,z_c\right)} f = \centralCharacter{\tau}\left(-1\right)^{\binom{c}{2}} \Lift^{\left(z_c,\dots,z_1\right)} \left(\left(\prod_{\substack{\left(i,j\right)\\
	1 \le i < j \le c}} \left(h_{s_{j - i}}^0 + \frac{q^k - 1}{q^{k\left(z_i - z_{j}\right)} - 1} \cdot \idmap_{\tau^{\circ c}} \right)\right) f\right),$$
	where the product is ordered by the lexicographic ordering, taken either right to left or left to right.
\end{theorem}

Applying \Cref{thm:intertwining-operator-of-depth-zero-representation} to $$\left(z_1,\dots,z_c\right) = \left(\frac{c-1}{2}, \frac{c-3}{2}, \dots, -\frac{c-1}{2}\right),$$ and combining with \Cref{prop:projection-operator-as-product-of-simple-reflections}, we get the following corollary.

\begin{corollary}\label{cor:intertwining-operator-of-depth-zero-representation-of-speh}
	The following diagram commutes
	$$\xymatrix{
		\tau^{\circ c} \ar[d]_{\Lift^{\left(\frac{c-1}{2},\dots,-\frac{c-1}{2}\right)}} \ar[rr]^{\ProjectionOperator_{\SpehRepresentation{\tau}{c}}}  & & \SpehRepresentation{\tau}{c} \ar[d]^{\Lift^{\left(-\frac{c-1}{2},\dots,\frac{c-1}{2}\right)}}\\
		\depthZeroRepresentation^{\left(\frac{c-1}{2},\dots,-\frac{c-1}{2}\right)} \ar[rr]^{M^{\ast}_{\weylElement{1^c}}} & & \depthZeroRepresentation^{\left(-\frac{c-1}{2},\dots,\frac{c-1}{2}\right)},}$$
	and hence we have a commutative diagram
	$$\xymatrix{
		\tau^{\circ c} \ar[d]_{\Lift^{\left(\frac{c-1}{2},\dots,-\frac{c-1}{2}\right)}} \ar[rr]^{\ProjectionOperator_{\SpehRepresentation{\tau}{c}}}  & & \SpehRepresentation{\tau}{c} \ar[d]^{\Lift^{\left(-\frac{c-1}{2},\dots,\frac{c-1}{2}\right)}}\\
		\depthZeroRepresentation^{\left(\frac{c-1}{2},\dots,-\frac{c-1}{2}\right)} \ar[rr]^{M^{\ast}_{\weylElement{1^c}}} & & \SpehRepresentation{\depthZeroRepresentation}{c}.}$$
	where $M^{\ast}_{\weylElement{1^c}} = \PoincarePolynomial{k}{c}\left(q^k\right)^{-1} \centralCharacter{\tau}\left(-1\right)^{-\binom{c}{2}} M^{\left(\frac{c-1}{2},\dots, -\frac{c-1}{2}\right)}_{\weylElement{1^c}}.$
\end{corollary}
From now and on, we denote for a function $f \in \SpehRepresentation{\tau}{c}$ its lift by $\Lift f = \Lift^{\left(-\frac{c-1}{2},\dots,\frac{c-1}{2}\right)} f \in \SpehRepresentation{\depthZeroRepresentation}{c}$.

\begin{remark}\label{rem:commutative-diagram-for-steinberg-representations}
	Applying the theorem at the point $$\left(z_1,\dots,z_c\right) = \left(-\frac{c-1}{2}, -\frac{c-3}{2}, \dots, \frac{c-3}{2}, \frac{c-1}{2}\right),$$ and using \Cref{rem:steinberg-representation-projection-operator}, we see that the Steinberg representation (see \Cref{rem:steinberg-representation-image-of-intertwining}) satisfies the following similar relation. We have the following commutative diagram:
		$$\xymatrix{
		\tau^{\circ c} \ar[d]_{\Lift^{\left(-\frac{c-1}{2},\dots,\frac{c-1}{2}\right)}} \ar[rr]^{\ProjectionOperator_{\SteinbergRepresentation{\tau}{c}}}  & & \SteinbergRepresentation{\tau}{c} \ar[d]^{\Lift^{\left(\frac{c-1}{2},\dots,-\frac{c-1}{2}\right)}}\\
		\depthZeroRepresentation^{\left(-\frac{c-1}{2},\dots,\frac{c-1}{2}\right)} \ar[rr]^{M^{\ast}_{\weylElement{1^c}}} & & \SteinbergRepresentation{\depthZeroRepresentation}{c},}$$
	where $M^{\ast}_{\weylElement{1^c}} = \left(-1\right)^{\binom{c}{2}} \PoincarePolynomial{k}{c}\left(q^k\right)^{-1} \centralCharacter{\tau}\left(-1\right)^{-\binom{c}{2}} M^{\left(-\frac{c-1}{2},\dots, \frac{c-1}{2}\right)}_{\weylElement{1^c}}.$
\end{remark}

\section{$\left(k,c\right)$ Whittaker models of Speh representations}\label{sec:wss-models}

This section is devoted to explaining the notion of $(k,c)$ $\fieldCharacter$-Whittaker models of Speh representations in the context of finite fields and of non-archimedean local fields, and to establishing a relation between these two models.

\subsection{$\left(k,c\right)$ $\fieldCharacter$-Whittaker models of Speh representations over finite fields}\label{sec:whittaker-models-of-speh-representations-over-finite-fields}

We define a character $\fieldCharacterkc{k}{c} \colon \UnipotentRadicalForWss{k}{c}\left(\finiteField\right) \to \multiplicativegroup{\cComplex}$ by the formula
\begin{equation}\label{eq:definition-of-k-c-character}
	\fieldCharacterkc{k}{c} \begin{pmatrix}
		\IdentityMatrix{c} & X_1 & \ast & \ast  & \ast \\
		& \IdentityMatrix{c} & X_2 & \ast & \ast \\
		& & \ddots & \ddots & \ast  \\
		& & & \IdentityMatrix{c} &  X_{k-1} \\
		& & & & \IdentityMatrix{c}
	\end{pmatrix} = \fieldCharacter\left( \sum_{j=1}^{k-1} \trace X_j \right).
\end{equation}
If $\Sigma$ is a finite-dimensional representation of $\GL_{kc}\left(\finiteField\right)$, we say that $0 \ne v \in \Sigma$ is a \emph{$(k,c)$ $\fieldCharacter$-Whittaker vector} if for any $u \in \UnipotentRadicalForWss{k}{c}\left(\finiteField\right)$, we have
$\Sigma(u) v = \fieldCharacterkc{k}{c}\left(u\right) v$. An element $0 \ne \ell \in \Hom_{\UnipotentRadicalForWss{k}{c}(\finiteField)}\left(\Sigma, \fieldCharacterkc{k}{c}\right)$ is called a \emph{$(k,c)$ $\fieldCharacter$-Whittaker functional}. The representation $\Sigma$ admits a $(k,c)$ $\fieldCharacter$-Whittaker vector if and only if it admits a $(k,c)$ $\fieldCharacter$-Whittaker functional.

When $c=1$ we omit the $(k,c)$ from the notation and simply say that $v$ is a $\fieldCharacter$-Whittaker vector and that $\ell$ is a $\fieldCharacter$-Whittaker functional. We also denote $\UnipotentSubgroup_k = \UnipotentRadicalForWss{k}{1}$ and $\fieldCharacter = \fieldCharacterkc{k}{c}$. A representation $\Sigma$ of $\GL_k\left(\finiteField\right)$ is called \emph{generic} if it admits a $\fieldCharacter$-Whittaker vector, and is called \emph{of Whittaker type} if it admits a unique (up to scalar multiplication) $\fieldCharacter$-Whittaker vector. These definitions do not depend on the choice of the non-trivial character $\fieldCharacter$. It is well-known that irreducible cuspidal representation of $\GL_k\left(\finiteField\right)$ are generic. A classical result of Gelfand--Graev is that irreducible generic representations are of Whittaker type. By Frobenius reciprocity, this implies that if $\Sigma$ is an irreducible generic representation of $\GL_k\left(\finiteField\right)$, then there exists a unique subspace of $\Ind{\UnipotentSubgroup_k\left(\finiteField\right)}{\GL_k\left(\finiteField\right)}{\fieldCharacter}$ that is isomorphic to $\Sigma$. We denote this subspace by $\Whittaker\left(\Sigma, \fieldCharacter\right)$ and call it the \emph{$\fieldCharacter$-Whittaker model of $\Sigma$}.

In his master's thesis \cite[Theorem 6.4]{Carmon2023}, Carmon showed that if $\tau$ is an irreducible cuspidal representation of $\GL_k\left(\finiteField\right)$, then for any $c$ the Speh representation $\SpehRepresentation{\tau}{c}$ admits a unique (up to scalar multiplication) $(k,c)$ $\fieldCharacter$-Whittaker vector, and hence also a unique (up to scalar multiplication) $(k,c)$ $\fieldCharacter$-Whittaker functional\footnote{His statement is more general and concerns general representations of $(k,c)$ type, but we do not need the general result for this work.}. These results are analogous to the results of \cite{CaiFriedbergGourevitchKaplan2023}, which we will discuss in the next section. It follows from Frobenius reciprocity that there exists a unique subspace of $\Ind{\UnipotentRadicalForWss{k}{c}\left(\finiteField\right)}{\GL_{kc}\left(\finiteField\right)}{\fieldCharacterkc{k}{c}}$ that is isomorphic to $\SpehRepresentation{\tau}{c}$. We denote this subspace by $\Whittaker\left(\SpehRepresentation{\tau}{c}, \fieldCharacterkc{k}{c}\right)$, and call it \emph{the $(k,c)$ $\fieldCharacter$-Whittaker model of $\SpehRepresentation{\tau}{c}$}.

This space can be realized explicitly as follows. Choose an inner product $\innerproduct{\cdot}{\cdot}_{\SpehRepresentation{\tau}{c}}$ on $\SpehRepresentation{\tau}{c}$ that is invariant under the action of $\GL_{kc}\left(\finiteField\right)$. Let $0 \ne v_{\SpehRepresentation{\tau}{c}, \fieldCharacterkc{k}{c}} \in \SpehRepresentation{\tau}{c}$ be a $(k,c)$ $\fieldCharacter$-Whittaker vector. Then
$$ \Whittaker\left(\SpehRepresentation{\tau}{c}, \fieldCharacterkc{k}{c}\right) = \left\{ \left. g \mapsto \innerproduct{\SpehRepresentation{\tau}{c}\left(g\right) v'_{\SpehRepresentation{\tau}{c}}}{v_{\SpehRepresentation{\tau}{c}, \fieldCharacterkc{k}{c}}}_{\SpehRepresentation{\tau}{c}}  \right| v'_{\SpehRepresentation{\tau}{c}} \in \SpehRepresentation{\tau}{c} \right\}.$$
A distinguished element of the $(k,c)$ $\fieldCharacter$-Whittaker model of $\SpehRepresentation{\tau}{c}$ is given by choosing a function corresponding to a $(k,c)$ $\fieldCharacter$-Whittaker vector, normalized at the identity to have the value $1$. This is the \emph{Bessel--Speh function corresponding to $\SpehRepresentation{\tau}{c}$ with respect to $\fieldCharacter$}, defined as the following matrix coefficient:
$$ \besselSpehFunction{\tau}{c}\left(g\right) = \frac{\innerproduct{\SpehRepresentation{\tau}{c}\left(g\right) v_{\SpehRepresentation{\tau}{c}, \fieldCharacterkc{k}{c}}}{v_{\SpehRepresentation{\tau}{c}, \fieldCharacterkc{k}{c}}}_{\SpehRepresentation{\tau}{c}}}{\innerproduct{v_{\SpehRepresentation{\tau}{c}, \fieldCharacterkc{k}{c}}}{v_{\SpehRepresentation{\tau}{c}, \fieldCharacterkc{k}{c}}}_{\SpehRepresentation{\tau}{c}}},$$
where $g \in \GL_{kc}\left(\finiteField\right)$. It follows from the uniqueness property of the $(k,c)$ $\fieldCharacterkc{k}{c}$-Whittaker vector of $\SpehRepresentation{\tau}{c}$ that this function is the unique function $W \in \Whittaker\left(\SpehRepresentation{\tau}{c}, \fieldCharacterkc{k}{c}\right)$ satisfying $W\left(\IdentityMatrix{kc}\right) = 1$ and $W\left(g u\right) = \fieldCharacterkc{k}{c}\left(u\right) W\left(g\right)$ for any $g \in \GL_{kc}\left(\finiteField\right)$ and any $u \in \UnipotentRadicalForWss{k}{c}\left(\finiteField\right)$.

By \cite[Proposition 6.21]{Carmon2023}, for any $\left(k,c\right)$ $\fieldCharacter$-Whittaker vector $v_{\SpehRepresentation{\tau}{c}, \fieldCharacterkc{k}{c}} \in \SpehRepresentation{\tau}{c}$ and any $h \in \GL_c\left(\finiteField\right)$, we have that $\SpehRepresentation{\tau}{c}\left(\diag^k\left(h\right)\right) v_{\SpehRepresentation{\tau}{c}, \fieldCharacterkc{k}{c}} = \centralCharacter{\tau}\left(\det h\right) v_{\SpehRepresentation{\tau}{c}, \fieldCharacterkc{k}{c}}$, where $\centralCharacter{\tau}$ is the central character of $\tau$ and $$\diag^k\left(h\right) = \diag\left(h,\dots,h\right) \in \GL_{kc}\left(\finiteField\right).$$
This result is analogous to \cite[Lemma 12]{CaiFriedbergGourevitchKaplan2023}. It follows that for any $W \in \Whittaker\left(\SpehRepresentation{\tau}{c}, \fieldCharacterkc{k}{c}\right)$, any $g \in \GL_{kc}\left(\finiteField\right)$ and any $h \in \GL_c\left(\finiteField\right)$,
$$W\left( \diag^k\left(h\right) g \right) = \centralCharacter{\tau}\left(\det h\right) W\left(g\right),$$ and that $$\besselSpehFunction{\tau}{c}\left(g \diag^k\left(h\right)\right) = \centralCharacter{\tau}\left(\det h\right)\besselSpehFunction{\tau}{c}\left(g\right).$$
In particular it follows that the function $\besselSpehFunction{\tau}{c}$ is invariant under conjugation by elements of $\diag^k\left(\GL_c\left(\finiteField\right)\right)$.

\subsection{$\left(k,c\right)$ $\fieldCharacter$-Whittaker models of Speh representations over local fields}
Similarly to \Cref{sec:whittaker-models-of-speh-representations-over-finite-fields}, if $\localField$ is a non-archimedean local field, $\fieldCharacter \colon \localField \to \multiplicativegroup{\cComplex}$ is a non-trivial character, then we may define a character $\fieldCharacterkc{k}{c} \colon \UnipotentRadicalForWss{k}{c} \to \multiplicativegroup{\cComplex}$ by the formula \eqref{eq:definition-of-k-c-character}.

As before, when $c=1$, we omit the $(k,c)$ from the notation and write $\fieldCharacter = \fieldCharacterkc{k}{1}$, $\UnipotentSubgroup_k\left(\localField\right) = \UnipotentRadicalForWss{k}{1}\left(\localField\right)$.

Let $\depthZeroRepresentation$ be an irreducible supercuspidal representation of $\GL_k\left(\localField\right)$. It is well-known that $\depthZeroRepresentation$ is generic, that is, there exists a non-zero \emph{$\fieldCharacter$-Whittaker functional\footnote{\label{footnote:whittaker-vectors-do-not-exist-in-local-fields}The notion of a $(k,c)$ $\fieldCharacter$-Whittaker vector does not exist for smooth representations over local fields.}} for $\depthZeroRepresentation$, i.e., an element $0 \ne \WhittakerFunctional{\depthZeroRepresentation} \in \Hom_{\UnipotentSubgroup_k\left(\localField\right)}\left(\depthZeroRepresentation, \fieldCharacter\right)$. Moreover, it is well-known that such element is unique, up to scalar multiplication. By Frobenius reciprocity, there exists a unique subspace of $\Ind{\UnipotentSubgroup_k\left(\localField\right)}{\GL_k\left(\localField\right)}{\fieldCharacter}$ that is isomorphic to $\depthZeroRepresentation$. This is \emph{the $\fieldCharacter$-Whittaker model of $\depthZeroRepresentation$}, which we denote by $\Whittaker\left(\depthZeroRepresentation, \fieldCharacter\right)$.

In \cite{CaiFriedbergGourevitchKaplan2023} it was shown that the Speh representation $\SpehRepresentation{\depthZeroRepresentation}{c}$ is a representation of type $\left(k,c\right)$. In particular, this means that $$ \dim \Hom_{\UnipotentRadicalForWss{k}{c}\left(\localField\right)}\left(\SpehRepresentation{\depthZeroRepresentation}{c}, \fieldCharacterkc{k}{c}\right) = 1.$$
We call a non-zero element in $\Hom_{\UnipotentRadicalForWss{k}{c}\left(\localField\right)}\left(\SpehRepresentation{\depthZeroRepresentation}{c}, \fieldCharacterkc{k}{c}\right)$ a \emph{ $\left(k,c\right)$ $\fieldCharacter$-Whittaker functional\cref{footnote:whittaker-vectors-do-not-exist-in-local-fields}}.

As before, by Frobenius reciprocity there exists a unique subspace of $\Ind{\UnipotentRadicalForWss{k}{c}}{\GL_{kc}\left(\localField\right)}{\fieldCharacterkc{k}{c}}$ that is isomorphic to $\SpehRepresentation{\depthZeroRepresentation}{c}$. We  denote this subspace by $\Whittaker\left(\SpehRepresentation{\depthZeroRepresentation}{c}, \fieldCharacterkc{k}{c}\right)$, and call it the \emph{$(k,c)$ $\fieldCharacter$-Whittaker model of $\SpehRepresentation{\depthZeroRepresentation}{c}$}.

By \cite[Lemma 12]{CaiFriedbergGourevitchKaplan2023}, we have the following equivariance property.
\begin{proposition}\label{prop:whittaker-left-diag-equivariance}
	 For any $W \in \Whittaker\left(\SpehRepresentation{\depthZeroRepresentation}{c}, \fieldCharacterkc{k}{c}\right)$, any $h \in \GL_c\left(\localField\right)$ and any $g \in \GL_{kc}\left(\localField\right)$, $$W\left(\diag^k\left(h\right)g\right) = \centralCharacter{\depthZeroRepresentation}\left(\det h\right) \cdot W\left(g\right),$$
	where $\centralCharacter{\depthZeroRepresentation}$ is the central character of $\depthZeroRepresentation$ and where $$\diag^k\left(h\right) = \diag\left(h,\dots,h\right) \in \GL_{kc}\left(\localField\right).$$
\end{proposition}

\subsection{$(k,c)$ $\fieldCharacter$-Whittaker models of Speh representations of level zero representations}\label{subsec:lifts-of-whittaker-functions}

Let $\tau$ and $\depthZeroRepresentation$ be as in \Cref{subsec:speh-representations-of-depth-zero-representations}. Let $\fieldCharacter \colon \localField \to \multiplicativegroup{\cComplex}$ be an additive character with conductor $\maximalIdeal$, i.e., $\fieldCharacter$ is trivial on $\maximalIdeal$ but not on $\ringOfIntegers$. We also realize $\fieldCharacter$ as a character of $\finiteField = \ringOfIntegers \slash \maximalIdeal$. Let $\WhittakerFunctional{\tau} \colon \tau \rightarrow \cComplex$ be a $\fieldCharacter$-Whittaker functional. A $\fieldCharacter$-Whittaker functional $\WhittakerFunctional{\depthZeroRepresentation} \colon \depthZeroRepresentation \to \cComplex$ is given by
$$ \standardForm{f}{\WhittakerFunctional{\depthZeroRepresentation}} = \int_{\UnipotentSubgroup_k \left(\ringOfIntegers\right) \backslash \UnipotentSubgroup_k \left(\localField\right)} \standardForm{f\left( u \right)}{\WhittakerFunctional{\tau}} \fieldCharacter^{-1}\left(u\right) \differential u,$$
where the Haar measures are normalized such that $\UnipotentSubgroup_k \left(\ringOfIntegers\right)$ has volume $1$. See \cite[Proposition 3.7]{YeZeligher18}.

This allows us to lift Whittaker functions $W \in \Whittaker\left(\tau, \fieldCharacter\right)$ to $\Lift\left( W\right) \in \Whittaker\left(\depthZeroRepresentation, \fieldCharacter\right)$ as follows. Any $W \in \Whittaker\left(\tau, \fieldCharacter\right)$ can be written as $W\left(g\right) = W_{v_\tau}\left(g\right) = \standardForm{\tau\left(g\right) v_{\tau}}{\WhittakerFunctional{\tau}}$, where $g \in \GL_k\left(\finiteField\right)$ and $v_{\tau} \in \tau$. We define for $g \in \GL_k\left(\localField\right)$, $$\Lift\left(W\right)\left(g\right) = \Lift\left(W_{v_\tau}\right)\left(g\right) = \standardForm{\depthZeroRepresentation \left(g\right) \Lift\left(v_\tau\right)}{\WhittakerFunctional{\depthZeroRepresentation}}.$$
We have the following formula for $\Lift\left(W\right)$ (\cite[Proposition 3.9]{YeZeligher18}).
\begin{proposition}\label{prop:whittaker-functional-lift}
	For any $W \in \Whittaker\left(\tau, \fieldCharacter\right)$, the function $\Lift\left(W\right) \in \Whittaker\left(\depthZeroRepresentation, \fieldCharacter\right)$ is supported on $\multiplicativegroup{\localField} \cdot \UnipotentSubgroup_k\left(\localField\right) \cdot \GL_k\left(\ringOfIntegers\right)$ and satisfies
	$$\Lift \left(W\right)\left(z u k_0\right) = \centralCharacter{\depthZeroRepresentation}\left(z\right)\fieldCharacter\left(u\right) W\left(\quotientMap\left(k_0\right)\right),$$
	for any $z \in \multiplicativegroup{\localField}$, $u \in \UnipotentSubgroup_k$ and $k_0 \in \GL_k\left(\ringOfIntegers\right)$.
\end{proposition}

Using the Whittaker functional $\WhittakerFunctional{\depthZeroRepresentation}$, we may construct a $\left(k,c\right)$ $\fieldCharacter$-Whittaker functional for $\SpehRepresentation{\depthZeroRepresentation}{c}$. This is a recursive process which is explained in \cite[Section 2.4]{CaiFriedbergGourevitchKaplan2023}. We recall this process now.

Let $c = c_1 + c_2$ where $0 < c_1, c_2 < c$ and suppose that we have already constructed $\left(k, c_1\right)$ and $\left(k, c_2 \right)$ $\fieldCharacter$-Whittaker functionals $\gShortSpehWhittakerFunctional{\depthZeroRepresentation}{k}{c_1}$ and $\gShortSpehWhittakerFunctional{\depthZeroRepresentation}{k}{c_2}$ for $\SpehRepresentation{\depthZeroRepresentation}{c_1}$ and $\SpehRepresentation{\depthZeroRepresentation}{c_2}$, respectively.

Let us realize $\SpehRepresentation{\depthZeroRepresentation}{c}$ as a subrepresentation of the parabolic induction $\abs{\det}^{-\frac{c_2}{2}} \SpehRepresentation{\depthZeroRepresentation}{c_1} \times \abs{\det}^{\frac{c_1}{2}} \SpehRepresentation{\depthZeroRepresentation}{c_2}$.

Let $\kappa \in \SymmetricGroup_{kc}$ be the permutation $$\kappa\left(1 + a + b c\right) = \begin{cases}
	1 + a + b c_1 & 0 \le a \le c_1 - 1,\\
	1+ a + c_1 k + b c_2 & c_1 \le a \le c_1 + c_2 - 1,
\end{cases}$$
where $0 \le a \le c-1$ and $0 \le b \le k-1$.
Also denote by $\kappa$ the column permutation matrix corresponding to $\kappa$, that is
\begin{equation}\label{eq:definition-of-kappa}
	\kappa = \begin{pmatrix}
		\transpose{e_{\kappa\left(1\right)}} & \dots & \transpose{e_{\kappa\left(kc\right)}}
	\end{pmatrix} = \begin{pmatrix}
		\IdentityMatrix{c_1} & 0\\
		0 & 0 & \IdentityMatrix{c_1} & 0\\
		\vdots & \vdots & &\ddots & \ddots\\
		0 & 0 & \cdots & 0 & 0 & \IdentityMatrix{c_1} & 0\\
		0 & \IdentityMatrix{c_2}\\
		0 & 0 & 0 & \IdentityMatrix{c_2}\\
		\vdots & \vdots & & \ddots & \ddots\\
		0 & 0 & \cdots & 0 & 0 & 0 & \IdentityMatrix{c_2}
	\end{pmatrix},
\end{equation}
where for every $1 \le j \le kc$, $\transpose{e_j}$ is the $j$th standard column vector. Suppose that $R$ is either a ring or a fractional ideal of a ring. Let \begin{equation}\label{eq:Y-R-subgroup}
	\mathcal{Y}\left(R\right) = \left\{ \begin{pmatrix}
		\IdentityMatrix{k c_1}\\
		Y & \IdentityMatrix{k c_2} \end{pmatrix} \mid Y = \begin{pmatrix}
		0& y_{12} & y_{13} & \dots & y_{1k}\\
		&0 & y_{23} & \dots & y_{2k}\\
		& & \ddots & \ddots & \vdots\\
		& & & 0 & y_{k-1,k}\\
		& & & & 0
	\end{pmatrix} \mid y_{ij} \in \Mat{c_2}{c_1}\left(R\right) \right\}.
\end{equation}

By \cite[Lemma 9]{CaiFriedbergGourevitchKaplan2023} the functional $\ell_{\SpehRepresentation{\depthZeroRepresentation}{c}} \colon \SpehRepresentation{\depthZeroRepresentation}{c} \to \cComplex$ defined by the formula
\begin{equation}\label{eq:recursive-formula-for-k-c-whittaker-function}
	\standardForm{\Phi}{\gShortSpehWhittakerFunctional{\depthZeroRepresentation}{k}{c}} = \int_{\mathcal{Y}\left(\localField\right)} \standardForm{\Phi\left(  y \kappa  \right)}{\gShortSpehWhittakerFunctional{\depthZeroRepresentation}{k}{c_1} \otimes \gShortSpehWhittakerFunctional{\depthZeroRepresentation}{k}{c_2}} \differential y
\end{equation}
is a non-zero $\left(k,c\right)$ $\fieldCharacter$-Whittaker functional of $\SpehRepresentation{\depthZeroRepresentation}{c}$.

By realizing $\SpehRepresentation{\tau}{c}$ as a subrepresentation of the parabolic induction $\SpehRepresentation{\tau}{c_1} \circ \SpehRepresentation{\tau}{c_2}$, we form the exact same process over a finite field to obtain a non-zero $\left(k,c\right)$ $\fieldCharacter$-Whittaker functional $\ell_{\SpehRepresentation{\tau}{c}} \colon \SpehRepresentation{\tau}{c_1} \circ \SpehRepresentation{\tau}{c_2} \to \cComplex$, given by the recursive formula
$$ \standardForm{\Phi}{\gShortSpehWhittakerFunctional{\tau}{k}{c}} = \frac{1}{\sizeof{\mathcal{Y}\left(\finiteField\right)}} \sum_{y \in \mathcal{Y}\left(\finiteField\right)} \standardForm{\Phi\left(  y \kappa \right)}{\gShortSpehWhittakerFunctional{\tau}{k}{c_1} \otimes \gShortSpehWhittakerFunctional{\tau}{k}{c_2}}. $$

For our applications, we will need to reduce the domain of integration in the formula for $\gShortSpehWhittakerFunctional{\depthZeroRepresentation}{k}{c}$. This is achieved by the method of root exchange (see \Cref{appendix:root-exchange}).

\begin{proposition}\label{prop:integrand-for-k-c-functional}
	For any $\Phi \in \SpehRepresentation{\depthZeroRepresentation}{c} \subset \abs{\det}^{-\frac{c_2}{2}} \SpehRepresentation{\depthZeroRepresentation}{c_1} \times \abs{\det}^{\frac{c_1}{2}} \SpehRepresentation{\depthZeroRepresentation}{c_2}$, such that $\Phi$ is right invariant under $\IdentityMatrix{kc} + \squareMatrix_{kc}\left(\maximalIdeal^{r+1}\right)$ for $r \ge 0$, we have $$\standardForm{\Phi}{\gShortSpehWhittakerFunctional{\depthZeroRepresentation}{k}{c}} = \int_{\mathcal{Y}\left(\uniformizer^{-r} \ringOfIntegers \right)} \standardForm{\Phi\left( y \kappa \right)}{\gShortSpehWhittakerFunctional{\depthZeroRepresentation}{k}{c_1} \otimes \gShortSpehWhittakerFunctional{\depthZeroRepresentation}{k}{c_2}} \differential y.$$
\end{proposition}

For any $f \in \SpehRepresentation{\tau}{c}$, let $W_f \colon \GL_{kc}\left(\finiteField\right) \to \cComplex$ be the function $$W_f\left(g\right) = \standardForm{\SpehRepresentation{\tau}{c}\left(g\right)f}{\gShortSpehWhittakerFunctional{\tau}{k}{c}}.$$ Similarly, for any $\Phi \in \SpehRepresentation{\depthZeroRepresentation}{c}$, let $W_\Phi \colon \GL_{kc}\left(\localField\right) \to \cComplex$ be the function $$ W_{\Phi}\left(g\right) = \standardForm{\SpehRepresentation{\depthZeroRepresentation}{c}\left(g\right)\Phi}{\gShortSpehWhittakerFunctional{\depthZeroRepresentation}{k}{c}}.$$ Then we have the following relation.

\begin{proposition}\label{prop:relation-between-k-c-whittaker-models}
	Normalize the Haar measure on $\mathcal{Y}\left(\localField\right)$ so that $\mathcal{Y}\left(\ringOfIntegers\right)$ has volume $1$. Then for any $k_0 \in \GL_{kc}\left(\ringOfIntegers\right)$ and any $f \in \SpehRepresentation{\tau}{c}$,
	$$W_{\Lift\left(f\right)}\left(k_0\right) = W_f\left(\quotientMap\left(k_0\right)\right).$$
\end{proposition}
\begin{proof}
	The proof is by induction on $c$. If $c=1$ then the statement follows from \Cref{prop:whittaker-functional-lift}. Otherwise, let $c = c_1 + c_2$ where $c_1, c_2 < c$. Then by induction we have that for $j = 1,2$ and $f_j \in \SpehRepresentation{\tau}{c_j}$ the equality $W_{\Lift f_j}\left(\IdentityMatrix{k c_j}\right) = W_{f_j}\left(\IdentityMatrix{k c_j}\right)$ holds, which implies that $\standardForm{{f_j}}{\gShortSpehWhittakerFunctional{\tau}{k}{c_j}} = \standardForm{{\Lift f_j}}{\gShortSpehWhittakerFunctional{\depthZeroRepresentation}{k}{c_j}}$. By \Cref{prop:integrand-for-k-c-functional}, we may reduce the domain of integration in \eqref{eq:recursive-formula-for-k-c-whittaker-function} to $\mathcal{Y}\left(\ringOfIntegers\right)$. Since we normalized $\mathcal{Y}\left(\ringOfIntegers\right)$ to have volume $1$, it follows that $\standardForm{\Lift f}{\gShortSpehWhittakerFunctional{\depthZeroRepresentation}{k}{c}} = \standardForm{f}{\gShortSpehWhittakerFunctional{\tau}{k}{c}}$. Finally, the proposition holds for general $k_0 \in \GL_{kc}\left(\ringOfIntegers\right)$ because $\SpehRepresentation{\depthZeroRepresentation}{c}\left(k_0\right) \Lift f = \Lift \left(\SpehRepresentation{\tau}{c}\left(\quotientMap\left(k_0\right)\right) f\right) $. 
\end{proof}

As before, for $W = W_f \in \Whittaker\left(\SpehRepresentation{\tau}{c}, \fieldCharacterkc{k}{c}\right)$, where $f \in \SpehRepresentation{\tau}{c}$, denote $\Lift W = W_{\Lift f} \in \Whittaker \left(\SpehRepresentation{\depthZeroRepresentation}{c}, \fieldCharacterkc{k}{c}\right)$. We will need knowledge of the behavior of $\Lift W$ at elements of the form $\diag\left(h, \IdentityMatrix{\left(k-1\right)c}\right)$ and of the form $\diag\left( \IdentityMatrix{\left(k-1\right)c}, h\right)$, where $h \in \GL_c\left(\localField\right)$. Using Propositions \ref{prop:whittaker-left-diag-equivariance} and \ref{prop:relation-between-k-c-whittaker-models} and the Cartan decomposition, it suffices to know how to compute these values for diagonal $h$.

\begin{proposition}\label{prop:support-of-diagonal-whittaker-elements}
	Let $t = \diag\left(t_1,\dots,t_c\right)$ where $t_1,\dots,t_c \in \multiplicativegroup{\localField}$.
	\begin{enumerate}
		\item \label{item:support-of-W-diag-t-I-kc} If $\Lift W\left(\diag\left(t, \IdentityMatrix{\left(k-1\right)c}\right)\right) \ne 0$, then $\abs{t_1}, \dots, \abs{t_c} \le 1$. Moreover, if $k > c$, then $\abs{t_1} = \dots = \abs{t_c} = 1$.
		\item \label{item:support-of-W-diag-I-kc-t} If $\Lift W\left(\diag\left(\IdentityMatrix{\left(k-1\right)c}, t\right)\right) \ne 0$, then $\abs{t_1},\dots,\abs{t_c} \ge 1$. Moreover, if $k > c$ then $\abs{t_1} = \dots = \abs{t_c} = 1$.
	\end{enumerate}
\end{proposition}
\begin{proof}
	We only prove the first part. The second part is proved similarly.
	
	If $\abs{t_i} > 1$ for some $i$, let $u \in \UnipotentRadicalForWss{k}{c}$ be an element such that it has zeros everywhere outside of its diagonal except for its $(i,i+c)$ entry, which is chosen to be $t_i^{-1} y$, where $y \in \multiplicativegroup{\ringOfIntegers}$ is an element such that $\fieldCharacter(y) \ne 1$. Then since $\mathcal{L}W$ is right invariant under $1 + \squareMatrix_{kc}\left(\maximalIdeal\right)$, we have $$\mathcal{L}W\left(\diag\left(t, \IdentityMatrix{\left(k-1\right)c}\right) u\right) = \mathcal{L}W\left(\diag\left(t, \IdentityMatrix{\left(k-1\right)c}\right) \right).$$ On the other hand, we get from conjugation and the left $\left(\UnipotentRadicalForWss{k}{c}, \fieldCharacterkc{k}{c}\right)$-equivariance property that $$\mathcal{L}W\left(\diag\left(t, \IdentityMatrix{\left(k-1\right)c}\right) u\right) = \fieldCharacter\left(y\right)\mathcal{L}W\left(\diag\left(t, \IdentityMatrix{\left(k-1\right)c}\right) \right).$$
	Thus we get that $\Lift W$ vanishes on $\diag\left(t, \IdentityMatrix{(k-1)c}\right)$.
	
	Suppose that $k > c$. Assume first that $t$ is dominant, i.e., that $\abs{t_1} \le \abs{t_2} \le \dots \le \abs{t_c}$. Consider the functional $$\ell_t \colon \Whittaker\left(\SpehRepresentation{\tau}{c}, \fieldCharacterkc{k}{c}\right) \to \cComplex$$ defined by the formula $$\ell_t\left(W\right) = \Lift W\left(\diag\left(t, \IdentityMatrix{\left(k-1\right)c}\right)\right).$$
	Suppose that $\abs{t_i} < 1$ and $\abs{t_{i+1}} = 1$. If $u \in \UnipotentRadical_{(i, kc - i)}\left(\ringOfIntegers\right)$, we get by conjugation and the left $\left(\UnipotentRadicalForWss{k}{c}, \fieldCharacterkc{k}{c}\right)$-equivariance property and the fact that $\fieldCharacter$ is trivial on $\maximalIdeal$ that $$\Lift W\left(\diag\left(t, \IdentityMatrix{(k-1)c}\right) u\right) = \Lift W\left(\diag\left(t, \IdentityMatrix{(k-1)c}\right)\right).$$
	In particular, this implies that the functional $\ell_t$ is invariant under right translations of $\UnipotentRadical_{(i,kc - i)}\left(\finiteField\right)$. It follows that $$\ell_t \in \Hom_{\UnipotentRadical_{(i, kc - i)}\left(\finiteField\right)}\left(\SpehRepresentation{\tau}{c}, 1\right).$$ The last $\Hom$-space is zero because otherwise there would be an irreducible cuspidal representation of $\GL_{c'}\left(\finiteField\right)$ for $c' \le i \le c < k$ in the cuspidal support of $\SpehRepresentation{\tau}{c}$.
	
	In the general case, let $w \in \GL_c\left(\finiteField\right)$ be a permutation matrix such that $w t w^{-1}$ is dominant. Then by \Cref{prop:whittaker-left-diag-equivariance}, we have $$\Lift{W}\left({\diag\left(t, \IdentityMatrix{(k-1)c}\right)}\right) = \centralCharacter{\tau}^{-1}\left(\det w\right)\Lift{W'}\left({\diag\left(wtw^{-1}, \IdentityMatrix{(k-1)c}\right)}\right),$$
	where $W' \in \Whittaker\left(\SpehRepresentation{\tau}{c}, \fieldCharacterkc{k}{c}\right)$ is given by $W' = \SpehRepresentation{\tau}{c}\left(\diag^{k}\left(w\right)\right) W$. From the previous case we get $\abs{t_1} = \dots = \abs{t_c} = 1$.
\end{proof}

We will also need formulas for the special case $k = c$. In this case, the special values can be expressed in terms of functionals on the space $\SpehRepresentation{\tau}{c}$, realized as a subrepresentation of the parabolic induction $\tau \circ \SpehRepresentation{\tau}{c - 1}$ (or of $\SpehRepresentation{\tau}{c - 1} \circ \tau$). In the notation of \Cref{subsec:lifts-of-whittaker-functions} and \eqref{eq:recursive-formula-for-k-c-whittaker-function}, we take $c_1 = 1$ and $c_2 = c-1$ (respectively, $c_1 = c-1$ and $c_2 = 1$).

\begin{theorem}\label{thm:special-value-of-diag-t-c(c-1)}
	Suppose that $\gShortSpehWhittakerFunctional{\tau}{k}{c}$ is constructed with respect to $\WhittakerFunctional{\tau}$ and $\gShortSpehWhittakerFunctional{\tau}{k}{c-1}$, and that $\gShortSpehWhittakerFunctional{\depthZeroRepresentation}{k}{c}$ is constructed with respect to $\WhittakerFunctional{\depthZeroRepresentation}$ and $\gShortSpehWhittakerFunctional{\depthZeroRepresentation}{k}{c-1}$. Then for any $t = \diag\left(t_1,\dots,t_c\right) = \diag\left(\uniformizer^{i_1},\dots,\uniformizer^{i_c}\right)$ with $\max(i_j)_{j=1}^c \ge 1$ and any $\Phi = \Lift f$ for $f \in \SpehRepresentation{\tau}{c} \subset \tau \circ \SpehRepresentation{\tau}{c-1}$, we have that $W_{\Phi}\left( \diag\left(t, \IdentityMatrix{c(c-1)}\right) \right)$ is zero unless $i_1 = \dots = i_c = m \ge 1$, in which case $$W_{\Phi}\left(\diag\left(t, \IdentityMatrix{c(c-1)}\right)\right) = \centralCharacter{\depthZeroRepresentation}\left(\uniformizer\right)^m q^{-(m+1)(c-1)\binom{c}{2}} \standardForm{f\left(\IdentityMatrix{c^2}\right)}{\firstSpecialFunctional},$$
		where $\firstSpecialFunctional$ is the functional on $\tau \otimes \SpehRepresentation{\tau}{c-1}$, given by
		\begin{align*}
			\standardForm{v_{\tau} \otimes f_{\SpehRepresentation{\tau}{c-1}}}{\firstSpecialFunctional} =& \sum_{g \in \GL_{c-1}\left(\finiteField\right)} \sum_{y \in \mathcal{Y}_{1,c-1}\left(\finiteField\right)} \standardForm{\tau\begin{pmatrix}
					1\\
					& -g^{-1}
				\end{pmatrix} v_{\tau}}{\WhittakerFunctional{\tau}}\\
			& \times \standardForm{\SpehRepresentation{\tau}{c-1}\left(\begin{pmatrix}
					g\\
					 & \IdentityMatrix{(c-1)^2}
				\end{pmatrix} y \kappa_{1,c-1} \right) f_{\SpehRepresentation{\tau}{c-1}}}{\gShortSpehWhittakerFunctional{\tau}{k}{c-1}},
		\end{align*}
		where $\mathcal{Y}_{1,c-1}$ is the group in \eqref{eq:Y-R-subgroup} for $c_1 = 1$ and $c_2 = c-1$ and $k = c-1$, and where $\kappa_{1,c-1}$ is the element from \eqref{eq:definition-of-kappa} corresponding to same parameters.
\end{theorem}
We have the following similar result for elements of the form $\diag\left(\IdentityMatrix{(k-1)c}, t\right)$.
\begin{theorem}\label{thm:special-value-of-diag-c(c-1)-t}
	Suppose that $\gShortSpehWhittakerFunctional{\tau}{k}{c}$ is constructed with respect to $\gShortSpehWhittakerFunctional{\tau}{k}{c-1}$ and $\WhittakerFunctional{\tau}$, and that $\gShortSpehWhittakerFunctional{\depthZeroRepresentation}{k}{c}$ is constructed with respect to $\gShortSpehWhittakerFunctional{\depthZeroRepresentation}{k}{c-1}$ and $\WhittakerFunctional{\depthZeroRepresentation}$. Then for any $t = \diag\left(t_1,\dots,t_c\right) = \diag\left(\uniformizer^{i_1},\dots,\uniformizer^{i_c}\right)$ with $\min(i_j)_{j=1}^c \le -1$ and any $\Phi = \Lift f$ for $f \in \SpehRepresentation{\tau}{c}$, we have that $W_{\Phi}\left( \diag\left(\IdentityMatrix{c(c-1)}, t\right) \right)$ is zero unless $i_1 = \dots = i_c = -m \le -1$, in which case $$W_{\Phi}\left(\diag\left(\IdentityMatrix{c(c-1)}, t\right)\right) = \centralCharacter{\depthZeroRepresentation}\left(\uniformizer\right)^{-m} q^{-(m+1)(c-1)\binom{c}{2}} \standardForm{f\left(\IdentityMatrix{c^2}\right)}{\firstDualSpecialFunctional},$$
	where $\firstDualSpecialFunctional$ is the functional on $\SpehRepresentation{\tau}{c-1} \otimes \tau$, given by
	\begin{align*}
		\MoveEqLeft[3] \standardForm{f_{\SpehRepresentation{\tau}{c-1}} \otimes v_{\tau}}{\firstDualSpecialFunctional} \\
		=& \sum_{g \in \GL_{c-1}\left(\finiteField\right)} \sum_{y \in \mathcal{Y}_{c-1,1}\left(\finiteField\right)} \standardForm{\SpehRepresentation{\tau}{c-1}\left(\begin{pmatrix}
				\IdentityMatrix{(c-1)^2}\\
				& -g^{-1}
			\end{pmatrix} y \kappa_{c-1,1} \right) f_{\SpehRepresentation{\tau}{c-1}}}{\gShortSpehWhittakerFunctional{\tau}{k}{c-1}} \\ 
		& \times \standardForm{\tau\begin{pmatrix}
				g\\
				& 1
			\end{pmatrix} v_{\tau}}{\WhittakerFunctional{\tau}},
	\end{align*}
	where $\mathcal{Y}_{c-1,1}$ is the group in \eqref{eq:Y-R-subgroup} for $c_1 = c-1$ and $c_2 = 1$ and $k = c-1$, and where $\kappa_{c-1,1}$ is the element from \eqref{eq:definition-of-kappa} corresponding to same parameters.	
\end{theorem}

\subsubsection{Proof of the formula for $k=c$}
The goal of this section is to prove \Cref{thm:special-value-of-diag-t-c(c-1)}. The proof of \Cref{thm:special-value-of-diag-c(c-1)-t} is very similar and we omit it.

Suppose that $k=c$. Let $f \in \SpehRepresentation{\tau}{c}$ and let $t = \diag\left(t_1,\dots,t_c\right) = \diag\left(\uniformizer^{i_1},\dots,\uniformizer^{i_c}\right)$, where $i_1,\dots,i_c \ge 0$ and $\max\left(i_1,\dots,i_c\right) > 0$. As explained in the proof of \Cref{prop:support-of-diagonal-whittaker-elements}, we may assume that $i_1 \ge \dots \ge i_c \ge 0$ and therefore $i_1 > 0$. Set $\Phi = \Lift f$ and $\Phi_t = \SpehRepresentation{\tau}{c}\left(\diag\left(t, \IdentityMatrix{\left(c-1\right)c}\right)\right) \Phi$. We are interested in computing $\standardForm{\Phi_t}{\gShortSpehWhittakerFunctional{\depthZeroRepresentation}{k}{c}}$. We will realize  $\SpehRepresentation{\depthZeroRepresentation}{c}$ as a subspace of the parabolic induction $\abs{\det}^{-\frac{c-1}{2}} \depthZeroRepresentation \times \abs{\det}^{\frac{1}{2}} \SpehRepresentation{\depthZeroRepresentation}{c-1}$ and will use the recursive formula \eqref{eq:recursive-formula-for-k-c-whittaker-function} with $c_1=1$ and $c_2=c-1$.

This proof is quite technical and requires work. We describe the process and postpone a technical lemma to later. Throughout the proof, all Haar measures are normalized so that the compact subgroups $\GL_k\left(\ringOfIntegers\right) \subset \GL_k\left(\localField\right)$ and $\squareMatrix_k\left(\ringOfIntegers\right) \subset \squareMatrix_k\left(\localField\right)$ have volume $1$, for any $k$.

Notice that $\Phi_t$ is invariant under right translations of $$\diag\left(t, \IdentityMatrix{c\left(c-1\right)} \right) \left(\IdentityMatrix{c^2} + \squareMatrix_{c^2}\left(\maximalIdeal\right) \right) \diag\left(t, \IdentityMatrix{c\left(c-1\right)} \right)^{-1},$$
and in particular is invariant under right translations of $\IdentityMatrix{c^2} + t_1 \squareMatrix_{c^2}\left(\maximalIdeal\right)$.
By \Cref{prop:integrand-for-k-c-functional}, we have that
$$\standardForm{\Phi_t}{\gShortSpehWhittakerFunctional{\depthZeroRepresentation}{k}{c}} =  \int_{\mathcal{Y}\left(t_1^{-1} \ringOfIntegers \right)} \standardForm{\Phi_t\left( y \kappa \right)}{\WhittakerFunctional{\depthZeroRepresentation} \otimes \gShortSpehWhittakerFunctional{\depthZeroRepresentation}{c}{c-1}} \differential y.$$
We have that $\Phi_t \left( y \kappa \right) = \Phi \left( y \kappa \diag\left(t, \IdentityMatrix{\left(c-1\right)c} \right) \right)$ and that $$ y \kappa \diag\left(t, \IdentityMatrix{\left(c-1\right)c}\right) = \diag\left(t_1 \IdentityMatrix{c}, \IdentityMatrix{c(c-1)} \right) y' \kappa,$$
where $y'$ is a matrix of the form
\begin{equation}\label{eq:k-equals-c-y-prime-definition}
	y' = \begin{pmatrix}
		1\\
		& t_1^{-1} \IdentityMatrix{c-1}\\
		& t_1^{-1} Y_1 & t' \\
		& t_1^{-1} Y' & & \IdentityMatrix{\left(c-1\right)^2}
	\end{pmatrix},
\end{equation}
where $t' = \diag\left(t_2,\dots,t_c\right)$, $Y_1 \in \squareMatrix_{c-1}\left(\ringOfIntegers\right)$ and $$Y' = \begin{pmatrix}
	0 & y_{23} & y_{24} & \dots & y_{2 c}\\
	& 0 & y_{34} & \dots & y_{3c}\\
	& & \ddots & \ddots & \vdots\\
	& & & 0 &  y_{c - 1, c}  \\
	& & & & 0
\end{pmatrix},$$
where $y_{ij} \in \Mat{\left(c-1\right)}{1}\left(\ringOfIntegers\right)$.
We have
\begin{equation}\label{eq:k-equal-c-integral-to-compute}
	\begin{split}
		\standardForm{\Phi_t}{\gShortSpehWhittakerFunctional{\depthZeroRepresentation}{k}{c}}= & \abs{t_1}^{-\binom{c}{2}\left(c-1\right)} \abs{t_1}^{\frac{(c-1)^2c}{2} }  \centralCharacter{\depthZeroRepresentation}\left(t_1\right) \int_{\squareMatrix_{c-1}\left(\ringOfIntegers\right)} \differential Y_1 \int_{\Mat{(c-1)}{1}\left(\ringOfIntegers\right)^{\binom{c-1}{2}}} \differential Y'\\
		& \times  \standardForm{ \Phi\left(  y' \kappa \right)}{\WhittakerFunctional{\depthZeroRepresentation} \otimes \gShortSpehWhittakerFunctional{\depthZeroRepresentation}{c}{c-1}}.
	\end{split}
\end{equation}
We will investigate the integrand of \eqref{eq:k-equal-c-integral-to-compute}:
\begin{equation}\label{eq:integrand-with-y0}
	\standardForm{\Phi\left(  y' \kappa \right)}{\WhittakerFunctional{\depthZeroRepresentation} \otimes \gShortSpehWhittakerFunctional{\depthZeroRepresentation}{c}{c-1}}.
\end{equation}

The following lemma allows us to reduce the integration domain of $Y_1$.
\begin{lemma}\label{lem:k-equals-c-implies-Y_1-is-in-maximal-compact}
	Let $y'$ be as in \eqref{eq:k-equals-c-y-prime-definition}. Suppose that the expression \eqref{eq:integrand-with-y0} is non-zero.
	Then there exists $k_1 \in \GL_{c-1}\left(\ringOfIntegers\right)$ such that $Y_1 = t' \cdot k_1$.
\end{lemma}
We postpone the proof of this lemma to later.

It follows that we may write $Y_1 = t' k_1$ and replace the integration over $Y_1$ with an integration over $k_1 \in \GL_{c-1}\left(\ringOfIntegers\right)$ in \eqref{eq:k-equal-c-integral-to-compute} with $\differential Y_1 = \abs{\det t'}^{c-1} \mdifferential k_1$, where we also multiply by $\frac{\sizeof{\GL_{c-1}\left(\finiteField\right)}}{\sizeof{\squareMatrix_{c-1}\left(\finiteField\right)}}$ in order to correct the volume. We write the following decomposition of $y'$, where the first matrix lies in $\ParabolicSubgroup_{(c,c(c-1))}$ and the other two matrices lie in $\GL_{c^2}\left(\ringOfIntegers\right)$.
$$y' = \begin{pmatrix}
	1\\
	& \IdentityMatrix{c-1} & t_1^{-1} k_1^{-1} \\
	& & t_1^{-1} t' \\
	& & t_1^{-1} Y' k_1^{-1} & \IdentityMatrix{\left(c-1\right)^2}
\end{pmatrix} \begin{pmatrix}
	1 \\
	& \IdentityMatrix{c-1}\\
	& & \IdentityMatrix{c-1}\\
	& Y' & & \IdentityMatrix{(c-1)^2}
\end{pmatrix} \begin{pmatrix}
	1\\
	& & -k_1^{-1}\\
	& k_1 & t_1 \IdentityMatrix{c-1}\\
	& & & \IdentityMatrix{\left(c-1\right)^2}
\end{pmatrix}.$$
Hence, we have \begin{equation}\label{eq:Phi-after-y_0-prime-iwasawa-decomposition}
	\begin{split}
		\Phi\left(y' \kappa \right) =& \abs{t_1^{-(c-1)} \det t'}^{-\frac{c-1}{2}} \idmap_{\depthZeroRepresentation} \otimes \SpehRepresentation{\depthZeroRepresentation}{c-1} \begin{pmatrix}
			t_1^{-1} t'  \\
			t_1^{-1} Y' k_1^{-1} & \IdentityMatrix{\left(c-1\right)^2}
		\end{pmatrix}  \\
		& \times \Phi\left(\begin{pmatrix}
			1\\
			& \IdentityMatrix{c-1}\\
			& & \IdentityMatrix{c-1}\\
			& Y' & & \IdentityMatrix{\left(c-1\right)^2}
		\end{pmatrix} \diag\left(1, \begin{pmatrix}
			& -k_1^{-1}\\
			k_1
		\end{pmatrix}, \IdentityMatrix{\left(c-1\right)^2}\right) \kappa \right).
	\end{split}
\end{equation}
Using \eqref{eq:Phi-after-y_0-prime-iwasawa-decomposition} and the method of root exchange, we are able to reduce the integration over $Y'$ in \eqref{eq:k-equal-c-integral-to-compute}. This is described in the following lemma. Its proof is very similar to \Cref{prop:integrand-for-k-c-functional}, and we omit it.
\begin{lemma}
	The integral $$ \int_{\GL_{c-1}\left(\ringOfIntegers\right)} \int_{\Mat{(c-1)}{1}\left(\ringOfIntegers\right)^{\binom{c-1}{2}}} \standardForm{\Phi\left(  y' \kappa \right)}{\WhittakerFunctional{\depthZeroRepresentation} \otimes \gShortSpehWhittakerFunctional{\depthZeroRepresentation}{c}{c-1}} \differential Y' \mdifferential Y_1$$ equals the integral
	$$ \int_{\GL_{c-1}\left(\ringOfIntegers\right)} \int_{t_1 \cdot \Mat{(c-1)}{1}\left(\ringOfIntegers\right)^{\binom{c-1}{2}}} \standardForm{\Phi\left(  y' \kappa \right)}{\WhittakerFunctional{\depthZeroRepresentation} \otimes \gShortSpehWhittakerFunctional{\depthZeroRepresentation}{c}{c-1}} \differential Y' \mdifferential Y_1.$$
\end{lemma} 
It follows that we may assume $Y' = t_1 Y_1'$, where $Y'_1 \in \Mat{(c-1)}{1}\left(\ringOfIntegers\right)^{\binom{c-1}{2}}$ with $\differential Y' = \abs{t_1}^{(c-1) \binom{c-1}{2}} \differential Y'_1$. Let us denote $$k_0 = \begin{pmatrix}
	1\\
	& \IdentityMatrix{c-1}\\
	& & \IdentityMatrix{c-1}\\
	& & Y_1' k_1^{-1} & \IdentityMatrix{\left(c-1\right)^2}
\end{pmatrix} \begin{pmatrix}
	1 \\
	& & -k_1^{-1}\\
	& k_1\\		
	& & & \IdentityMatrix{\left(c-1\right)^2}
\end{pmatrix} \kappa \in \GL_{c^2}\left(\ringOfIntegers\right).$$ Notice that \eqref{eq:integrand-with-y0} is the same as \begin{equation}\label{eq:k-equals-c-to-reduction-to-c-minu-1-speh}
	\abs{t_1^{-(c-1)} \det t'}^{-\frac{c-1}{2}} \standardForm{\idmap_{\depthZeroRepresentation} \otimes \SpehRepresentation{\depthZeroRepresentation}{c-1} \begin{pmatrix}
			t_1^{-1} t'  \\
			& \IdentityMatrix{\left(c-1\right)^2}
		\end{pmatrix} \Phi \left(k_0\right)}{\WhittakerFunctional{\depthZeroRepresentation} \otimes \gShortSpehWhittakerFunctional{\depthZeroRepresentation}{c}{c-1}}.
\end{equation}
By \Cref{prop:support-of-diagonal-whittaker-elements} part \ref{item:support-of-W-diag-t-I-kc}, we have that \eqref{eq:k-equals-c-to-reduction-to-c-minu-1-speh} is zero unless $\abs{t_1} = \abs{t_j}$ for every $j$. Hence, we must have $t' = t_1 \IdentityMatrix{c-1}$.

Since $Y_1' \in \Mat{\left(c-1\right)^2}{\left(c-1\right)}\left(\ringOfIntegers\right)$, using the identity
$$ \diag\left(1, \begin{pmatrix}
	& \IdentityMatrix{c-1}\\
	\IdentityMatrix{c-1}
\end{pmatrix}, \IdentityMatrix{(c-1)^2}\right) \cdot \kappa = \diag\left(\IdentityMatrix{c}, \kappa_{1,c-1}\right),$$
the integral \eqref{eq:k-equal-c-integral-to-compute} becomes
\begin{align*}
	& \frac{\sizeof{\GL_{c-1}\left(\finiteField\right)}}{\sizeof{\squareMatrix_{(c-1)}\left(\finiteField\right)}} \centralCharacter{\depthZeroRepresentation}\left(t_1\right) \abs{t_1}^{(c-1)^2}\abs{t_1}^{(c-1) \binom{c-1}{2}} \int_{\GL_{c-1}\left(\ringOfIntegers\right)} \mdifferential k_1 \cdot  \int_{\Mat{(c-1)}{1}\left(\ringOfIntegers\right)^{\binom{c-1}{2}}} \differential Y_1' \\
	&\times \standardForm{\depthZeroRepresentation\begin{pmatrix}
			1 &\\
			& -k_1^{-1}
		\end{pmatrix} \otimes \SpehRepresentation{\depthZeroRepresentation}{c-1} \left(\begin{pmatrix}
			k_1 &\\
			Y_1' & \IdentityMatrix{\left(c-1\right)^2}
		\end{pmatrix} \kappa_{1,c-1} \right) \Phi\left( \IdentityMatrix{c^2} \right)} {\WhittakerFunctional{\depthZeroRepresentation} \otimes \gShortSpehWhittakerFunctional{\depthZeroRepresentation}{c}{c-1}}.
\end{align*}
Since $\Phi = \Lift f$, the last integral becomes the finite sum
\begin{align*}
	& \centralCharacter{\depthZeroRepresentation}\left(t_1\right) \abs{t_1}^{(c-1)\binom{c}{2}} \frac{1}{\sizeof{\squareMatrix_{c-1}\left(\finiteField\right)}} \frac{1}{\sizeof{\Mat{(c-1)}{1}\left(\finiteField\right)}^{\binom{c-1}{2}}} \sum_{g \in \GL_{c-1}\left(\finiteField\right)} \sum_{Y \in \Mat{(c-1)}{1}\left(\finiteField\right)^{\binom{c-1}{2}}}\\
	\times &\standardForm{\tau\begin{pmatrix}
			1 &\\
			& -g^{-1}
		\end{pmatrix} \otimes \SpehRepresentation{\tau}{c-1} \left(\begin{pmatrix}
			g &\\
			Y & \IdentityMatrix{\left(c-1\right)^2}
		\end{pmatrix} \kappa_{1,c-1} \right) f \left( \IdentityMatrix{c^2} \right)} {\WhittakerFunctional{\tau} \otimes \gShortSpehWhittakerFunctional{\tau}{c}{c-1}},
\end{align*}
as required. \qed

We move to prove \Cref{lem:k-equals-c-implies-Y_1-is-in-maximal-compact}.
\begin{proof}
	We start with finding a decomposition of the matrix $y'$. Let $\left(t'\right)^{-1} Y_1 = k_1 a k_2$, be a Cartan decomposition, that is, $k_1, k_2 \in \GL_{c-1}\left(\ringOfIntegers\right)$ and $a = \diag\left(\uniformizer^{m_1},\dots, \uniformizer^{m_{c-1}}\right)$ with $m_1 \ge m_2 \ge \dots \ge m_{c-1}$. We write \begin{equation}\label{eq:first-decomposition-of-k1}
		y' = \diag\left(1, k_2^{-1}, t' k_1, \IdentityMatrix{\left(c-1\right)^2} \right) \begin{pmatrix}
			1\\
			& t_1^{-1} \IdentityMatrix{c-1}\\
			& t_1^{-1} a & \IdentityMatrix{c-1} \\
			& t_1^{-1} Y'' & & \IdentityMatrix{\left(c-1\right)^2}
		\end{pmatrix} \diag\left(1, k_2, k_1^{-1}, \IdentityMatrix{\left(c-1\right)^2} \right),
	\end{equation}
	where $Y'' = Y' k_2^{-1}$.
	Hence, \begin{equation}\label{eq:expression-for-Phi-y-prime-c-c}
		\begin{split}
			\Phi\left(y'\right) =& \depthZeroRepresentation\left(\diag\left(1, k_2^{-1}\right)\right) \otimes 
			\SpehRepresentation{\depthZeroRepresentation}{c-1}\left(\diag\left(t' k_1, \IdentityMatrix{\left(c-1\right)^2}\right)\right) \\
			&\times \Phi\left(\begin{pmatrix}
				1\\
				& t_1^{-1} \IdentityMatrix{c-1}\\
				& t_1^{-1} a & \IdentityMatrix{c-1}\\
				& t_1^{-1} Y'' & & \IdentityMatrix{\left(c-1\right)^2}
			\end{pmatrix} \diag\left(1, k_2, k_1^{-1}, \IdentityMatrix{\left(c-1\right)^2} \right)\right).
		\end{split}
	\end{equation}
	
	Assume first that $\left(t_1^{-1} a\right)^{-1} \in \squareMatrix_{(c-1)}\left(\ringOfIntegers\right)$, that is, $\abs{t_1} \le q^{-m_j}$ for all $j$. Then we have that the first matrix on the second line of \eqref{eq:expression-for-Phi-y-prime-c-c} can be decomposed as \begin{equation}\label{eq:first-decomposition-of-simple-case}
		\begin{pmatrix}
			1\\
			& -a^{-1} & t_1^{-1} \IdentityMatrix{c-1}\\
			& & t_1^{-1} a\\
			& & t_1^{-1} Y'' & \IdentityMatrix{(c-1)^2}
		\end{pmatrix} \begin{pmatrix}
		1\\
		& \IdentityMatrix{c-1}\\
		& & \IdentityMatrix{c-1}\\
		& -Y'' a^{-1} & & \IdentityMatrix{(c-1)^2}
		\end{pmatrix} \begin{pmatrix}
			1\\
			& & \IdentityMatrix{c-1}\\
			& \IdentityMatrix{c-1} & t_1 a^{-1}\\
			& & & \IdentityMatrix{(c-1)^2}
		\end{pmatrix}.
	\end{equation}
	The first matrix in \eqref{eq:first-decomposition-of-simple-case} lies in $P_{(c,(c-1)c)}$. By our assumption, the third matrix in \eqref{eq:first-decomposition-of-simple-case} lies in $\GL_{c^2}\left(\ringOfIntegers\right)$. A root exchange algorithm (see \Cref{appendix:root-exchange}) shows that if $-Y'' a^{-1}$ is not in $\Mat{(c-1)^2}{c-1}\left(\ringOfIntegers\right)$, then for any for any $\Phi' \in \SpehRepresentation{\depthZeroRepresentation}{c} \subset \abs{\det}^{-\frac{c-1}{2}} \depthZeroRepresentation \times \abs{\det}^{\frac{1}{2}} \SpehRepresentation{\depthZeroRepresentation}{c-1}$ that is right invariant under $\IdentityMatrix{kc} + \squareMatrix_{kc}\left(\maximalIdeal\right)$, the following expression is zero
	\begin{equation}
		\standardForm{\Phi'\left(  	\begin{pmatrix}
				1\\
				& -k_2^{-1} a^{-1} & t_1^{-1} k_2^{-1} \\
				& & t_1^{-1} t' k_1 a\\
				& & t_1^{-1} Y'' & \IdentityMatrix{(c-1)^2}
			\end{pmatrix} \begin{pmatrix}
				1\\
				& \IdentityMatrix{c-1}\\
				& & \IdentityMatrix{c-1}\\
				& -Y'' a^{-1} & & \IdentityMatrix{(c-1)^2}
			\end{pmatrix} \right)}{\WhittakerFunctional{\depthZeroRepresentation} \otimes \gShortSpehWhittakerFunctional{\depthZeroRepresentation}{c}{c-1}}.
	\end{equation}
	This is proved by writing for $2 \le  i \le c-1$, $$\Phi'\left(g\right) = \frac{1}{\VolumeOf\left(\squareMatrix_{c-1}\left(\maximalIdeal\right)\right)}\int_{\squareMatrix_{c-1}\left(\maximalIdeal\right)} \Phi'\left(g r_i^{\ast}\left(x\right)\right) \differential x,$$
	where for $x \in \squareMatrix_{c-1}\left(\localField\right)$, $r^{\ast}_i\left(x\right) = \diag\left(\IdentityMatrix{c}, \IdentityMatrix{\left(i - 1\right)\left(c-1\right)}, \begin{pmatrix}
		\IdentityMatrix{c-1} & x\\
		& \IdentityMatrix{c-1}
	\end{pmatrix}, \IdentityMatrix{(c-i-1)(c-1)} \right)$ and by using the fact that $\gShortSpehWhittakerFunctional{\depthZeroRepresentation}{c}{c-1}$ is a $(c, c-1)$ $\fieldCharacter$-Whittaker functional. We start with $i = c-1$ and end with $i = 2$ in this process (notice that the last $(c-1)$ rows of $-Y'' a^{-1}$ are zero). Hence it follows that \eqref{eq:expression-for-Phi-y-prime-c-c} amounts to an expression of the form $$ \standardForm{\depthZeroRepresentation\begin{pmatrix}
		1\\
		& -k_2^{-1} a^{-1}
	\end{pmatrix} \otimes \SpehRepresentation{\depthZeroRepresentation}{c-1}\begin{pmatrix}
	t_1^{-1} t' k_1 a\\
	t_1^{-1} Y'' & \IdentityMatrix{(c-1)^2}
	\end{pmatrix} \Lift v_{\tau^{\otimes c}}}{\WhittakerFunctional{\depthZeroRepresentation} \otimes \gShortSpehWhittakerFunctional{\depthZeroRepresentation}{c}{c-1}},$$
	where $v_{\tau^{\otimes c}} \in \tau^{\otimes c}$. 
	By \Cref{prop:whittaker-functional-lift}, we have that this expression is zero unless $\begin{pmatrix}
		1 &\\
		& -k_2^{-1} a^{-1}
	\end{pmatrix} \in \multiplicativegroup{\localField} \cdot \UnipotentSubgroup_c \cdot \GL_k\left(\ringOfIntegers\right)$, which by the Iwasawa decomposition implies that $a = \IdentityMatrix{c-1}$, which implies that $Y_1 = t' k_1 k_2$, as required.

	We now move to the general case. Let $a^{+} = \diag\left(\uniformizer^{m_1},\dots,\uniformizer^{m_r}\right)$ and $a^{-} = \diag\left(\uniformizer^{m_{r+1}},\dots,\uniformizer^{m_{c-1}}\right)$ where the $m_j$'s in $a^{+}$ satisfy $q^{-m_j} < \abs{t_1}$, while the $m_j$'s in $a^{-}$ satisfy $q^{-m_j} \ge \abs{t_1}$. Write $Y'' = \begin{pmatrix}
		Y''_+ & Y''_{-}
	\end{pmatrix}$, where $Y''_+ \in \Mat{\left(c-1\right)^2}{r}\left(\ringOfIntegers\right)$ and $Y''_- \in \Mat{\left(c-1\right)^2}{(c-r-1)}\left(\ringOfIntegers\right)$. We may decompose the first matrix on the second line of \eqref{eq:expression-for-Phi-y-prime-c-c} as \begin{equation}\label{eq:cartan-decomposition-for-whittaker-c-c-y'}
		\begin{split}
			&\begin{pmatrix}
				1\\
				& t_1^{-1} \IdentityMatrix{r}\\
				& & -\left(a^{-}\right)^{-1} & & t_1^{-1} \IdentityMatrix{c-r-1}\\
				& & & \IdentityMatrix{r} \\
				& & & & t_1^{-1} a^{-}\\
				& t_1^{-1} Y_+'' & -Y_{-}''  \left(a^-\right)^{-1}  & & t_1^{-1} Y''_{-} & \IdentityMatrix{\left(c-1\right)^2}
			\end{pmatrix} \\
			\times & \begin{pmatrix}
				1\\		
				& \IdentityMatrix{r}\\
				& & \IdentityMatrix{c-r-1}\\
				& t_1^{-1} a^{+} & & \IdentityMatrix{r}\\
				& & & & \IdentityMatrix{c-r-1}\\
				& & & & & \IdentityMatrix{\left(c-1\right)^2}
			\end{pmatrix} \\
			\times & \diag\left(\IdentityMatrix{r+1}, \begin{pmatrix}
				& & \IdentityMatrix{c-r-1}\\
				& \IdentityMatrix{r}\\
				\IdentityMatrix{c-r-1} & & t_1 \left(a^-\right)^{-1}
			\end{pmatrix}, \IdentityMatrix{\left(c-1\right)^2}\right).
		\end{split}
	\end{equation}
	Denote the matrix on the first row of \eqref{eq:cartan-decomposition-for-whittaker-c-c-y'} by $y'_{-}$.
	Notice that the two other matrices in \eqref{eq:cartan-decomposition-for-whittaker-c-c-y'} lie in $\GL_{c^2}\left(\ringOfIntegers\right)$. Hence, to determine whether \eqref{eq:integrand-with-y0} vanishes, it suffices to check whether the similar expression, in which we replace $y' \kappa$ with $\diag\left(1,k_2^{-1}, t' k_1, \IdentityMatrix{(c-1)^2}\right) y'_{-}$, vanishes.
	We proceed by decomposing $y'_{-}$. Let $Y''_{+} = k'_1 a_0' k'_2,$ where $k'_1 \in \GL_{\left(c-1\right)^2}\left(\ringOfIntegers\right)$, $k'_2 \in \GL_{r}\left(\ringOfIntegers\right)$, and $$a'_0 = \begin{pmatrix}
		a'\\
		0_{\left(\left(c-1\right)^2-r\right) \times r}
	\end{pmatrix},$$ with $a' = \diag\left(\uniformizer^{m'_1}, \dots, \uniformizer^{m'_r}\right)$, where $\infty \ge m'_1 \ge \dots \ge m'_r \ge 0$. As before, we write $y'_{-}$ as \begin{equation}\label{eq:decomposition-of-y-minus-prime}
		\begin{split}
			& \diag\left(1,\left(k'_2\right)^{-1}, \IdentityMatrix{2c-r-2}, t_1^{-1} k'_1\right) \cdot \begin{pmatrix}
				1\\
				& t_1^{-1} \IdentityMatrix{r}\\
				& & -\left(a^{-}\right)^{-1} & & t_1^{-1} \IdentityMatrix{c-r-1}\\
				& & & \IdentityMatrix{r} \\
				& & & & t_1^{-1} a^{-}\\
				& a'_0 & -(k'_1)^{-1} Y''_{-} \left(t_1^{-1} a^{-}\right)^{-1} & & (k'_1)^{-1} Y''_{-} & t_1 \IdentityMatrix{\left(c-1\right)^2} 
			\end{pmatrix} \\
			& \times \diag\left(1, k'_2, \IdentityMatrix{2c-r-2}, \left(k'_1\right)^{-1}\right).
		\end{split}
	\end{equation}
	As before, let $a'_{+} = \diag\left(\uniformizer^{m'_1},\dots,\uniformizer^{m'_{r'}}\right)$ and $a'_{-} = \diag\left(\uniformizer^{m'_{r'+1}}, \dots, \uniformizer^{m'_{r}} \right)$, where all the $m'_j$s in $a'_{+}$ satisfy $q^{-m'_j} < \abs{t_1}$, while all the $m'_j$s in $a'_{-}$ satisfy $q^{-m'_j} \ge \abs{t_1}$. Notice that if $r \ge 1$ then $r' \ge 1$. This is because the rank of $Y_{-}''$ is at most $c-2$, which implies that $m'_1 = \infty$. Write $$\left(k'_1\right)^{-1} Y_{-}'' = \begin{pmatrix}
		Z_1\\
		Z_2\\
		Z_3
	\end{pmatrix}$$ where $Z_1 \in \Mat{r'}{(c-r-1)}\left(\localField\right)$, $Z_2 \in \Mat{(r-r')}{(c-r-1)}\left(\localField\right)$ and $Z_3 \in \Mat{(\left(c-1\right)^2 - r)}{(c-r-1)}\left(\localField\right)$. We may write the second matrix in \eqref{eq:decomposition-of-y-minus-prime} in the following form:
	\begin{equation}\label{eq:decomposition-of-second-matrix-in-decomposition-of-y-minus-prime}
		\begin{split}
			& \begin{pmatrix}
				1\\
				& t_1^{-1} \IdentityMatrix{r'}\\
				& & -\left(a_{-}'\right)^{-1} & & & & & t_1^{-1} \IdentityMatrix{r - r'}\\
				& & & -(a^{-})^{-1} & & t_1^{-1} \IdentityMatrix{c-r-1} & &  \\
				& &  & & \IdentityMatrix{r} & & &  \\
				& & & & & t_1^{-1} a^{-} & &  \\
				& & & -Z_1 \left(t_1^{-1} a^{-}\right)^{-1} & & Z_1 & t_1 \IdentityMatrix{r'}\\
				& & & -Z_2 \left(t_1^{-1} a^{-}\right)^{-1} & & Z_2 & & a'_{-}\\
				& & & -Z_3 \left(t_1^{-1} a^{-}\right)^{-1} & & Z_3 & & & t_1 \IdentityMatrix{\left(c-1\right)^2 - r}
			\end{pmatrix}\\
			\times & \begin{pmatrix}1\\
				& \IdentityMatrix{r'}\\
				& & & & & \IdentityMatrix{r - r'}\\
				& & & \IdentityMatrix{2c - r - 2}\\
				& t_1^{-1} a'_{+} & & & \IdentityMatrix{r'} &\\	
				& & \IdentityMatrix{r - r'} & & & t_1 \left(a'_{-}\right)^{-1}\\
				& & &  & & & \IdentityMatrix{\left(c-1\right)^2 - r}
			\end{pmatrix}
		\end{split}
	\end{equation}
	
	Notice that the matrix appearing on the second row of \eqref{eq:decomposition-of-second-matrix-in-decomposition-of-y-minus-prime} lies in $\GL_{c^2}\left(\ringOfIntegers\right)$. Let $$\begin{pmatrix}
		-(a^{-})^{-1} & & t_1^{-1} \IdentityMatrix{c-r-1} & &  \\
		& \IdentityMatrix{r} & & &  \\
		& & t_1^{-1} a^{-} & &  \\
		-Z_1 \left(t_1^{-1} a^{-}\right)^{-1} & & Z_1 & t_1 \IdentityMatrix{r'}\\
		-Z_2 \left(t_1^{-1} a^{-}\right)^{-1} & & Z_2 & & a'_{-}\\
		-Z_3 \left(t_1^{-1} a^{-}\right)^{-1} & & Z_3 & & & t_1 \IdentityMatrix{\left(c-1\right)^2 - r}
	\end{pmatrix} = \begin{pmatrix}
	b'_{c-1-r} & X\\
	& b'_{c\left(c-1\right)}
	\end{pmatrix} k_0'$$ be an Iwasawa decomposition, where $b'_{c-1-r} \in \GL_{c-1-r}\left(\localField\right)$ and $b'_{c\left(c-1\right)} \in \GL_{c\left(c-1\right)}\left(\localField\right)$ are upper triangular matrices, $X \in \Mat{(c-1-r)}{c(c-1)}\left(\localField\right)$, and $k_0' \in \GL_{c^2 - r - 1}\left(\ringOfIntegers\right)$. We have that $$\Phi\left(y' \kappa \right) = \abs{\det b_1}^{\frac{(c-1)^2}{2}} \abs{\det b_{c-1}}^{-\frac{c-1}{2}} \depthZeroRepresentation\left(b_1\right) \otimes \SpehRepresentation{\depthZeroRepresentation}{c-1}\left(\diag\left(t', \IdentityMatrix{(c-1)^2}\right) b_{c-1}\right) \Phi\left(k_0\right),$$
	where \begin{align*}
		b_1 &= \diag\left(1, k_2^{-1}\right) \cdot \diag\left(1, \left(k'_2\right)^{-1}, \IdentityMatrix{c-r-1}\right) \cdot \diag\left(1, t_1^{-1}\IdentityMatrix{r'}, -\left(a_{-}'\right)^{-1}, b'_{c-1-r}\right),\\
		b_{c-1} &= \diag\left(t' k_1, \IdentityMatrix{\left(c-1\right)^2}\right) \cdot \diag\left(\IdentityMatrix{c-1}, t_1^{-1} k'_1\right) \cdot b'_{c(c-1)},
	\end{align*}
	and $k_0 \in \GL_{c^2}\left(\ringOfIntegers\right)$. Since $k_0 \in \GL_{c^2}\left(\ringOfIntegers\right)$, $\Phi\left(k_0\right) = \mathcal{L}\left(v_{\tau \otimes \SpehRepresentation{\tau}{c-1}} \right),$ for some $v_{\tau \otimes \SpehRepresentation{\tau}{c-1}} \in \tau \otimes \SpehRepresentation{\tau}{c-1}$. By \Cref{prop:whittaker-functional-lift}, in order for $\standardForm{\Phi\left(y' \kappa \right)}{\WhittakerFunctional{\depthZeroRepresentation} \otimes \gShortSpehWhittakerFunctional{\depthZeroRepresentation}{c-1}{c}}$ not to vanish, we must have $b_1 \in \multiplicativegroup{\localField} \cdot \UnipotentSubgroup_c \cdot \GL_c\left(\ringOfIntegers\right)$, which by the Iwasawa decomposition implies that $$\diag\left(1,t_1^{-1}\IdentityMatrix{r'}, -\left(a'_{-}\right)^{-1}, d'_{c-1-r}\right) \in \GL_c\left(\ringOfIntegers\right),$$
	where $d'_{c-1-r}$ is the diagonal part of $b'_{c-1-r}$. Since $\abs{t_1} < 1$, we must have $r' = 0$, which implies that $r = 0$. Therefore, $a = a^{-}$ and we are in the first case.
\end{proof}

\section{Gamma factors}

This section is devoted to explaining the notion of the Ginzburg--Kaplan gamma factor in the context of finite fields and of non-archimedean local fields, and to establishing a relation between these two definitions of gamma factors.

\subsection{Ginzburg--Kaplan gamma factors for finite general linear groups}\label{sec:ginzburg-kaplan-finite-gln}

We recall the definition and properties of the Ginzburg--Kaplan gamma factor defined by Carmon and the author in \cite{CarmonZelingher2024}. This is a finite field analog of a construction given by Kaplan in \cite[Appendix A]{kaplan2018} and \cite{Kaplan2023}, whose global counterpart was given by Ginzburg \cite{ginzburg2019tensor}.

Let $\tau$ be an irreducible cuspidal\footnote{In \cite{CarmonZelingher2024}, the authors define this notion for any irreducible generic $\tau$, but for our purposes, we only need to consider cuspidal $\tau$.} representation of $\GL_k\left(\finiteField\right)$. For any positive integer $c$ and any $h \in \GL_c\left(\finiteField\right)$, let $$\specialBesselSpeh{\tau}\left(h\right) = \besselSpehFunction{\tau}{c}\begin{pmatrix}
	& \IdentityMatrix{(k-1)c}\\
	h
\end{pmatrix}.$$
Notice that the notation $\specialBesselSpeh{\tau}\left(h\right)$ does not include the parameters $k$ and $c$. These parameters should be inferred from $\tau$ and $h$, respectively. By the discussion at the end of \Cref{sec:whittaker-models-of-speh-representations-over-finite-fields}, it follows that the assignment $\GL_c\left(\finiteField\right) \to \cComplex$ given by $h \mapsto \specialBesselSpeh{\tau}\left(h\right)$ is a class function of $\GL_c\left(\finiteField\right)$, that is, it is invariant under $\GL_c\left(\finiteField\right)$-conjugation. Let $\pi$ be an irreducible representation of $\GL_c\left(\finiteField\right)$, and consider the operator $$\GKGaussSum{\pi}{\tau}{\fieldCharacter} = q^{\frac{\left(k-2\right)c^2}{2}} \sum_{h \in \GL_c\left(\finiteField\right)} \specialBesselSpeh{\tau}\left(h\right) \pi\left(h\right).$$
Since $h \mapsto \specialBesselSpeh{\tau}\left(h\right)$ is a class function, we have that $\GKGaussSum{\pi}{\tau}{\fieldCharacter}$ is an element of $\Hom_{\GL_c\left(\finiteField\right)}\left(\pi, \pi\right)$. Since $\pi$ is irreducible, by Schur's lemma there exists a scalar $\GKPreGammaFactor{\pi}{\tau}{\fieldCharacter} \in \cComplex$ such that $\GKGaussSum{\pi}{\tau}{\fieldCharacter} = \GKPreGammaFactor{\pi}{\tau}{\fieldCharacter} \cdot \idmap_{\pi}$. We call $\GKPreGammaFactor{\pi}{\tau}{\fieldCharacter}$  \emph{the Ginzburg--Kaplan gamma factor associated to $\pi$ and $\tau$, with respect to $\fieldCharacter$}. We denote $$\GKGammaFactor{\pi}{\tau}{\fieldCharacter} = \centralCharacter{\pi}\left(-1\right)^{k-1} \cdot \GKPreGammaFactor{\pi}{\tau}{\fieldCharacter},$$
where $\centralCharacter{\pi}$ is the central character of $\pi$.

In \cite{CarmonZelingher2024}, Carmon and the author studied the Ginzburg--Kaplan gamma factor. They showed that this gamma factor is multiplicative and that it fits into a functional equation. They used these results to show that $\specialBesselSpeh{\tau}$ is expressible in terms of exotic matrix Kloosterman sums. The proof of the relation to exotic matrix Kloosterman sums relies on an identity (\Cref{cor:gamma-factor-equality-with-epsilon-factor}) that follows from a local argument that we will establish in this paper.

\subsubsection{Kaplan's functional equation}\label{sec:kaplan-functional-equation-finite-field}
Let $\pi$ and $\tau$ be as above. Define for any $W \in \Whittaker\left(\SpehRepresentation{\tau}{c}, \fieldCharacterkc{k}{c}\right)$ the following \emph{zeta operators} on the space $\pi$: $$\zetaOperator\left(W, \pi \times \tau\right) = \frac{1}{\sizeof{\GL_c\left(\finiteField\right)}}\sum_{h \in \GL_c\left(\finiteField\right)} W\begin{pmatrix}
	h\\
	& \IdentityMatrix{(k-1)c}
\end{pmatrix} \pi\left(h\right)$$
and
$$\dualZetaOperator\left(W, \pi \times \tau\right) = q^{-\frac{\left(k-2\right)c^2}{2}} \frac{1}{\sizeof{\GL_c\left(\finiteField\right)}} \sum_{h \in \GL_c\left(\finiteField\right)} \sum_{X \in \Mat{c}{(k-2)c}\left(\finiteField\right)} W \begin{pmatrix}
	& \IdentityMatrix{c}\\
	& & \IdentityMatrix{(k-2)c}\\
	h & & X
\end{pmatrix}.$$
These are the operators $\frac{1}{\sizeof{\GL_c\left(\finiteField\right)}} \zetaOperator_{k-2}\left(W, \pi \times \tau\right)$ and $\frac{1}{\sizeof{\GL_c\left(\finiteField\right)}} \dualZetaOperator_{k-2}\left(W, \pi \times \tau\right)$ from \cite[Section 4.3]{CarmonZelingher2024}, respectively. They are finite field analogs of operators studied by Kaplan in \cite[Appendix A]{kaplan2018} and \cite{Kaplan2023}, see also \Cref{subsec:kaplans-local-zeta-integrals}.

In \cite[Theorem 4.14 and Corollary 4.21]{CarmonZelingher2024}, the following functional equation was established.
\begin{theorem}
	Suppose that $\Contragradient{\tau}$ does not appear in the cuspidal support of $\pi$. Then for any $W \in \Whittaker\left(\SpehRepresentation{\tau}{c}, \fieldCharacterkc{k}{c}\right)$, we have the equality of operators $$\dualZetaOperator\left(W, \pi \times \tau\right) = \GKPreGammaFactor{\pi}{\tau}{\fieldCharacter} \cdot \zetaOperator\left(W, \pi \times \tau\right).$$
	Moreover, in this case $$\abs{\GKPreGammaFactor{\pi}{\tau}{\fieldCharacter}} = 1.$$
\end{theorem}

\subsection{Relation to level zero representations}
In this section, we relate the zeta operators from \Cref{sec:kaplan-functional-equation-finite-field} and their counterparts defined over a local field using level zero representations.

\subsubsection{Kaplan's local zeta integrals}\label{subsec:kaplans-local-zeta-integrals}
In this section, we recall the definition of Kaplan's local zeta integrals and the gamma factor that they define.

Let $\Pi$ and $\depthZeroRepresentation$ be irreducible representations of $\GL_c\left( \localField \right)$ and $\GL_k\left( \localField \right)$, respectively. Assume that $\depthZeroRepresentation$ is generic. In \cite[Appendix A]{kaplan2018} and \cite{Kaplan2023}, Kaplan introduced a local integral construction that represents the $L$-function associated to the tensor product representation of the pair $\left(\Pi, \depthZeroRepresentation\right)$. Kaplan's integrals arise as local integrals of a global construction of Ginzburg \cite[Section 5.3]{ginzburg2019tensor}. Kaplan's zeta integral is given by the formula
$$\zetaOperator\left(s, v, v^{\vee}, W; \fieldCharacter \right) = \int_{\GL_c\left(\localField\right)} W\left(\diag\left(h, \IdentityMatrix{\left(k-1\right)c}\right) \right) \standardForm{\Pi\left(h\right)v}{v^{\vee}} \abs{\det h}^{s - \frac{\left(k-2\right)c + 1}{2}} \mdifferential h.$$
Here, $s \in \cComplex$, $v \in \Pi$, $v^{\vee} \in \Contragradient{\Pi}$ and $W \in \Whittaker\left(\SpehRepresentation{\depthZeroRepresentation}{c}, \fieldCharacterkc{k}{c}\right)$ (as explained before, the Speh representation $\SpehRepresentation{\depthZeroRepresentation}{c}$ is defined for $\depthZeroRepresentation$ generic and not necessarily supercuspidal \cite{CaiFriedbergGourevitchKaplan2023, LapidMao2020}, but we will not need this).
The integral $\zetaOperator\left(s, v, v^{\vee}, W; \fieldCharacter \right)$ converges absolutely for $\RealPart s \gg 0$, and in this domain it converges to an element of $\cComplex\left(q^{-s}\right)$. We keep denoting this element of $\cComplex\left(q^{-s}\right)$ by $\zetaOperator\left(s, v, v^{\vee}, W; \fieldCharacter \right)$. It can be shown that for some data $\left(v,v^{\vee},W\right)$, the integral $\zetaOperator\left(s, v, v^{\vee}, W; \fieldCharacter \right)$ is the constant function $1$.

Kaplan's dual zeta integral is given by the formula
\begin{align*}
	\dualZetaOperator\left(s,v,v^{\vee},W;\fieldCharacter\right) = & \int_{\GL_c\left(\localField\right)} \int_{\squareMatrix_{c \times 
			\left(k-2\right)c}\left(\localField\right)} \standardForm{\Pi\left(h\right)v}{v^{\vee}} \cdot \abs{\det h}^{s - 1 + \frac{\left(k-2\right)c + 1}{2}}\\
	& \times W\left( \diag \left(
		\IdentityMatrix{\left(k-1\right)c}, h
	\right) \begin{pmatrix}
		\IdentityMatrix{c}\\
		& \IdentityMatrix{\left(k-2\right)c}\\
		& X & \IdentityMatrix{c}
	\end{pmatrix} \begin{pmatrix}
	& \IdentityMatrix{\left(k-1\right)c}\\
	\IdentityMatrix{c}
\end{pmatrix} \right) \differential X \mdifferential h.
\end{align*}
This integral converges absolutely for $\RealPart s \ll 0$ and in this domain it converges to an element of $\cComplex\left(q^{-s}\right)$. We keep denoting this element by $	\dualZetaOperator\left(s,v,v^{\vee},W;\fieldCharacter\right)$. Similarly to before, it can be shown that for some data, the integral $\dualZetaOperator\left(s, v, v^{\vee}, W; \fieldCharacter \right)$ is the constant function $1$.

These zeta integrals satisfy a functional equation.

\begin{theorem}
	There exists an element $\LocalGKGammaFactor{s}{\Pi}{\depthZeroRepresentation}{\fieldCharacter} \in \cComplex\left(q^{-s}\right)$, such that for every $W \in \Whittaker\left(\SpehRepresentation{\depthZeroRepresentation}{c}, \fieldCharacterkc{k}{c}\right)$, $v \in \Pi$ and $v^{\vee} \in \Contragradient{\Pi}$, the following functional equation holds:
	$$ \dualZetaOperator\left(s,v,v^{\vee},W;\fieldCharacter\right) = \centralCharacter{\pi}\left(-1\right)^{k-1} \LocalGKGammaFactor{s}{\Pi}{\depthZeroRepresentation}{\fieldCharacter} \zetaOperator\left(s,v,v^{\vee},W;\fieldCharacter\right).$$ 
\end{theorem}
	We call the gamma factor arising from this functional equation the \emph{Ginzburg--Kaplan $\gamma$-factor}.
	
	Although not stated in \cite[Appendix A]{kaplan2018}, it follows from a globalization argument that for supercuspidal representations $\Pi$ and $\depthZeroRepresentation$, the Ginzburg--Kaplan gamma factor $\LocalGKGammaFactor{s}{\Pi}{\depthZeroRepresentation}{\fieldCharacter}$ agrees with the Jacquet--Piatetski-Shapiro--Shalika gamma factor \cite{Jacquet1983rankin}. The reason is because by \cite[Theorem A.2]{kaplan2018}, the Ginzburg--Kaplan $\gamma$-factor is multiplicative\footnote{It is only shown to be multiplicative in the first argument, but when $c=1$ the integral coincides with the Rankin--Selberg integral of \cite{Jacquet1983rankin} for $\GL_{k} \times \GL_1$, and these gamma factors are multiplicative. These results combined are sufficient to take care of the unramified and archimedean places.} and by \cite[Section 5.3]{ginzburg2019tensor}, the zeta integrals arise from a global integral representing the $L$-function of the tensor product. See \cite[Section 6.7.1]{CaiFriedbergKaplan2022}.
\subsubsection{Computation for level zero representations}\label{subsubsec:computation-for-level-zero-representations}
Let $\tau$ and $\depthZeroRepresentation$ be as in \Cref{subsec:speh-representations-of-depth-zero-representations}. Let $\pi$ be an irreducible cuspidal representation of $\GL_c\left(\finiteField\right)$, and let $\Pi$ be an irreducible level zero supercuspidal representation constructed from $\pi$ with central character $\centralCharacter{\Pi}$. The goal of this section is to express Kaplan's local zeta integrals in terms of those in \Cref{sec:kaplan-functional-equation-finite-field}.

We first prove a simple lemma regarding matrix coefficients that arise by lifting elements from $\pi$ and $\Contragradient{\pi}$. Let $\Pi'$ be the irreducible level zero supercuspidal representation constructed from $\Contragradient{\pi}$ with central character $\centralCharacter{\Pi}^{-1}$. We have that $\Contragradient{\Pi}$ is isomorphic to $\Pi'$. Henceforth we will identify $\Contragradient{\Pi}$ with the representation $\Pi'$.

\begin{lemma} \label{lem:support-of-matrix-coefficient-of-depth-zero}
	Let $v \in \pi$ and $v^{\vee} \in \Contragradient{\pi}$. 
	\begin{enumerate}
		\item Suppose that $h \in \GL_c\left(\localField\right)$ is such that $$\standardForm{\Pi\left(h\right)\Lift v}{\Lift v^{\vee}} \ne 0.$$
		Then $h \in \multiplicativegroup{\localField} \cdot \GL_c\left(\ringOfIntegers\right)$.
		\item For any $h \in \GL_c\left(\ringOfIntegers\right)$, $$\standardForm{\Pi\left(h\right) \Lift v}{\Lift v^{\vee}} = \standardForm{\pi\left(\quotientMap\left(h\right)\right)v}{v^{\vee}}.$$
	\end{enumerate}
	
\end{lemma}
\begin{proof}
	Given $f \in \Pi$ and $f^{\vee} \in \Pi' \cong \Contragradient{\Pi}$, the pairing $\standardForm{f}{f^{\vee}}$ is given by
	
	$$\standardForm{f}{f^{\vee}} = \int_{\multiplicativegroup{\localField} \GL_c\left(\ringOfIntegers\right) \backslash \GL_c\left(\localField\right)} \standardForm{f\left(x\right)}{f^{\vee}\left(x\right)} \mdifferential x.$$
	Suppose that $v \in \pi$ and $v \in \Contragradient{\pi}$. Then we have that $$\standardForm{ \Pi\left(h\right) \Lift v}{\Lift v^{\vee}} = \int_{\multiplicativegroup{\localField} \GL_c\left(\ringOfIntegers\right) \backslash \GL_c \left(\localField\right)} \standardForm{\Lift v \left(xh\right)}{\Lift \Contragradient{v}\left(x\right)} \mdifferential x,$$
	Since $\Lift v^{\vee}$ is supported on $\multiplicativegroup{\localField} \GL_c\left(\ringOfIntegers\right)$, in order for the integrand not to vanish, we must have $x \in \multiplicativegroup{\localField} \GL_c\left(\ringOfIntegers\right)$. Since $\Lift v$ is also supported on $\multiplicativegroup{\localField} \GL_c\left(\ringOfIntegers\right)$, this implies if $\standardForm{\Pi\left(h\right)\Lift v}{\Lift v^{\vee}} \ne 0$ then $h \in \multiplicativegroup{\localField} \GL_c\left(\ringOfIntegers\right)$. This proves the first part. The second part is immediate.
\end{proof}

Next, we prove a lemma regarding partial evaluation of Kaplan's zeta integrals.

\begin{lemma}\label{lem:integral-vanishes-depth-zero-non-isomorphic}
	Suppose that $\pi$ is not isomorphic to $\Contragradient{\tau}$.
	\begin{enumerate}
		\item \label{item:vanishing-lemma-for-first-integral}If $\abs{t} < 1$, then
		$$\int_{\GL_c\left( \ringOfIntegers \right)} \Lift W\left(\diag\left(th, \IdentityMatrix{\left(k-1\right)c}\right) \right) \standardForm{\Pi\left(h\right) \Lift v}{\Lift v^{\vee}} \mdifferential h = 0.$$
		\item \label{item:vanishing-lemma-for-second-integral} If $\abs{t} > 1$, then \begin{align*}
			& \int_{\GL_c\left(\ringOfIntegers\right)}   \Lift W\left( \diag\left(\IdentityMatrix{\left(k-1\right)c}, th\right) \right) \standardForm{\Pi\left(h\right) \Lift v}{\Lift v^{\vee}}  \mdifferential h = 0.
		\end{align*}
	\end{enumerate}
\end{lemma}
\begin{proof}
	Let $t \in \multiplicativegroup{\localField}$ and suppose that $\abs{t} < 1$. Define a map $Z_t \colon \Whittaker \left(\SpehRepresentation{\tau}{c}, \fieldCharacter\right) \times \pi \times \Contragradient{\pi} \to \cComplex$ by
	$$ Z_t\left(W,v,v^{\vee}\right) = \int_{\GL_c\left( \ringOfIntegers \right)} \Lift W\left(\diag\left(th, \IdentityMatrix{\left(k-1\right)c}\right) \right) \standardForm{\Pi\left(h\right) \Lift v}{\Lift v^{\vee}} \mdifferential h.$$
	Let $n \in \UnipotentRadical_{\left(c,\left(k-1\right)c\right)}\left(\ringOfIntegers\right)$, then \begin{equation}\label{eq:conjugation-of-unipotent-by-diagonal}
		\diag\left(th, \IdentityMatrix{\left(k-1\right)c}\right) n \diag\left(th, \IdentityMatrix{\left(k-1\right)c}\right)^{-1} = \IdentityMatrix{kc} \left(\bmod \maximalIdeal\right).
	\end{equation} Therefore, since the left hand side of \eqref{eq:conjugation-of-unipotent-by-diagonal} lies in $\UnipotentRadicalForWss{k}{c}$, we have that $$\Lift W\left(\diag\left(th, \IdentityMatrix{\left(k-1\right)c}\right) n \right) = \Lift W\left(\diag\left(th, \IdentityMatrix{\left(k-1\right)c}\right) \right).$$
	This implies that for any $n \in \UnipotentRadical_{\left(c,\left(k-1\right)c\right)}\left(\finiteField\right)$ and any $h \in \GL_c\left(\finiteField\right)$, we have $$Z_t\left( \SpehRepresentation{\tau}{c}\left(n \diag\left(h, \IdentityMatrix{\left(k-1\right)c}\right) \right) W, \pi\left(h\right)v , v^{\vee} \right) = Z_t\left(W, v , v^{\vee} \right).$$
	If $Z_t$ is not identically zero, it follows that $\SpehRepresentation{\tau}{c}$ embeds into the space $\pi^{\vee} \circ \Sigma$ for some irreducible representation $\Sigma$ of $\GL_{\left(k-1\right)c}\left(\finiteField\right)$. Since the cuspidal support of $\SpehRepresentation{\tau}{c}$ is $\left\{ \tau,\dots,\tau \right\}$, we have that $Z_t$ is identically zero unless $k = c$ and $\pi^{\vee} = \tau$. This proves part \ref{item:vanishing-lemma-for-first-integral}. Part \ref{item:vanishing-lemma-for-second-integral} is proved similarly.		
\end{proof}

Consider the following functionals $\secondSpecialFunctional \colon \Contragradient{\tau} \otimes \tau \otimes \tau \otimes \SpehRepresentation{\tau}{c-1} \to \cComplex$ and $\secondDualSpecialFunctional \colon \Contragradient{\tau} \otimes \tau \otimes \SpehRepresentation{\tau}{c-1}\otimes \tau \to \cComplex$.
\begin{align*}
	\MoveEqLeft[3] \standardForm{v_{\tau^{\vee}} \otimes v_{\tau} \otimes v'_{\tau} \otimes f_{\SpehRepresentation{\tau}{c-1}}}{\secondSpecialFunctional}
	&\\
	= & \standardForm{v_{\Contragradient{\tau}}}{v'_{\tau}} \cdot \frac{1}{\sizeof{\GL_{c-1}\left(\finiteField\right)}}\frac{1}{\sizeof{\mathcal{Y}_{1,c-1}\left(\finiteField\right)}}\sum_{g \in \GL_{c-1}\left(\finiteField\right)} \sum_{y \in \mathcal{Y}_{1,c-1}\left(\finiteField\right)} W_{v_{\tau}}\begin{pmatrix}
		1\\
		& -g^{-1}
	\end{pmatrix} \\
	& \times W_{f_{\SpehRepresentation{\tau}{c-1}}}\left(\begin{pmatrix}
			g\\
			& \IdentityMatrix{(c-1)^2}
		\end{pmatrix} y \kappa_{1,c-1}\right).
	\end{align*}
and
\begin{align*}
	\MoveEqLeft[3] \standardForm{v_{\tau^{\vee}} \otimes v_{\tau} \otimes f_{\SpehRepresentation{\tau}{c-1}} \otimes v'_{\tau}}{\secondDualSpecialFunctional} 
	&\\
	= & \standardForm{v_{\Contragradient{\tau}}}{v'_{\tau}} \cdot \frac{1}{\sizeof{\GL_{c-1}\left(\finiteField\right)}}\frac{1}{\sizeof{\mathcal{Y}_{c-1,1}\left(\finiteField\right)}} \sum_{g \in \GL_{c-1}\left(\finiteField\right)} \sum_{y \in \mathcal{Y}_{c-1,1}\left(\finiteField\right)} W_{v_{\tau}}\begin{pmatrix}
		g\\
		& 1
	\end{pmatrix}\\
		&\times W_{f_{\SpehRepresentation{\tau}{c-1}}}\left(\begin{pmatrix}
			\IdentityMatrix{(c-1)^2}\\
			& -g^{-1}
		\end{pmatrix} y \kappa_{c-1,1} \right).
\end{align*}
These functionals are similar to the ones appearing in Theorems \ref{thm:special-value-of-diag-t-c(c-1)} and \ref{thm:special-value-of-diag-c(c-1)-t}.

We are now ready to state our results regarding the relation between Kaplan's zeta integrals and the ones in \Cref{sec:kaplan-functional-equation-finite-field}. 
Normalize the Haar measures on $\multiplicativegroup{\localField}$ and $\GL_{c}\left(\localField\right)$ so that $\multiplicativegroup{\ringOfIntegers}$ and $\GL_c \left(\ringOfIntegers\right)$ have volume $1$, respectively.

\begin{theorem}\label{prop:equality-of-gk-zeta-integrals}
	 Let $v \in \pi$, $v^{\vee} \in \Contragradient{\pi}$ and $W \in \Whittaker\left(\SpehRepresentation{\tau}{c}, \fieldCharacterkc{k}{c}\right)$.
	 \begin{enumerate}
	 	\item Suppose that $\pi \ncong \Contragradient{\tau}$. Then we have $$\zetaOperator\left(s,\Lift v,\Lift v^{\vee}, \Lift W;\fieldCharacter\right) =  \standardForm{\zetaOperator\left(W, \pi \times \tau \right) v}{v^{\vee}}.$$
	 	\item \label{item:gk-zeta-integral-exceptional-case} Suppose that $\pi = \Contragradient{\tau}$ (and therefore $k=c$). Then for any $v \in \Contragradient{\tau}$, $v^{\vee} \in \tau$ and $W = W_{f}$ for $f \in \SpehRepresentation{\tau}{c} \subset \tau \circ \SpehRepresentation{\tau}{c-1}$ we have \begin{align*}
	 		\MoveEqLeft[3] \zetaOperator\left(s,\Lift v,\Lift v^{\vee}, \Lift W;\fieldCharacter\right) = \standardForm{\zetaOperator\left(W, \pi \times \tau \right) v}{v^{\vee}}\\
	 		& + q^{-\frac{c^2}{2}} \cdot \frac{\centralCharacter{\Pi}\left(\uniformizer\right) \centralCharacter{\depthZeroRepresentation}\left(\uniformizer\right)q^{-c\left(s - \frac{1}{2}\right)}}{1 - \centralCharacter{\Pi}\left(\uniformizer\right) \centralCharacter{\depthZeroRepresentation}\left(\uniformizer\right) q^{-cs}} \standardForm{v \otimes v^{\vee} \otimes f\left(\IdentityMatrix{c^2}\right)}{\secondSpecialFunctional}.
	 	\end{align*}
	 \end{enumerate}
\end{theorem}
\begin{proof}
	By \Cref{lem:support-of-matrix-coefficient-of-depth-zero} we have that \begin{align*}
		\zetaOperator\left(s, \Lift v, \Lift v^{\vee}, \Lift W; \fieldCharacter \right) =& \int_{\multiplicativegroup{\localField}} \int_{\GL_c\left( \ringOfIntegers \right)} \Lift W\left(\diag\left(t h, \IdentityMatrix{\left(k-1\right)c}\right) \right) \standardForm{\Pi\left(th\right)\Lift v}{\Lift v^{\vee}} \\
		& \times \abs{t}^{cs - c\frac{\left(k-2\right)c + 1}{2}}  \mdifferential h \mdifferential t.
	\end{align*}
	By \Cref{prop:support-of-diagonal-whittaker-elements}, we have that if the integrand does not vanish for $t \in \multiplicativegroup{\localField}$, then $\abs{t} \le 1$. Thus \begin{equation}\label{eq:zeta-integral-as-infinite-series}
		\begin{split}
			\MoveEqLeft[3]\zetaOperator\left(s, \Lift v, \Lift v^{\vee}, \Lift W; \fieldCharacter \right)
			= \sum_{j = 0}^{\infty} \centralCharacter{\Pi}\left(\uniformizer\right)^j q^{-jcs+jc \frac{\left(k-2\right)c + 1}{2}}\\
			& \times \int_{\multiplicativegroup{\ringOfIntegers}} \int_{\GL_c\left( \ringOfIntegers \right)} \Lift W\left(\diag\left( \uniformizer^j t h, \IdentityMatrix{\left(k-1\right)c}\right) \right) \standardForm{\Pi\left(th\right) \Lift v}{\Lift v^{\vee}} \mdifferential h \mdifferential t.
		\end{split}
	\end{equation}
	By changing variable, we omit the integration over $t$, and are reduced to a series of integrals as in \Cref{lem:integral-vanishes-depth-zero-non-isomorphic} part \ref{item:vanishing-lemma-for-first-integral}. Notice that for $j = 0$, we have from \Cref{prop:relation-between-k-c-whittaker-models} and from the second part of \Cref{lem:support-of-matrix-coefficient-of-depth-zero} that $$\int_{\GL_c\left( \ringOfIntegers \right)} \Lift W\left(\diag\left(h, \IdentityMatrix{\left(k-1\right)c}\right) \right) \standardForm{\Pi\left(h\right) \Lift v}{\Lift v^{\vee}} \mdifferential h = \standardForm{\zetaOperator\left(W; \pi \times \tau\right) v}{v^{\vee}}.$$

	Suppose that $\pi \ncong \Contragradient{\tau}$. By \Cref{lem:integral-vanishes-depth-zero-non-isomorphic} part \ref{item:vanishing-lemma-for-first-integral}, the summand of \eqref{eq:zeta-integral-as-infinite-series} vanishes for $j \ge 1$. Hence the result follows.
	
	Suppose that $\pi = \Contragradient{\tau}$. In this case, we get from \Cref{thm:special-value-of-diag-t-c(c-1)} applied to the function $\SpehRepresentation{\tau}{c}\left(\diag\left(\quotientMap\left(h\right), \IdentityMatrix{(k-1)c}\right)\right)f$ and from the second part of \Cref{lem:support-of-matrix-coefficient-of-depth-zero}, that for $j \ge 1$,
	\begin{equation}\label{eq:expression-of-k-c-representations-are-related-gl-c-o-integral}
		\begin{split}
			\MoveEqLeft[3] \int_{\GL_c\left( \ringOfIntegers \right)} \Lift W\left(\diag\left( \uniformizer^j h, \IdentityMatrix{\left(k-1\right)c}\right) \right) \standardForm{\Pi\left(h\right) \Lift v}{\Lift v^{\vee}} \mdifferential h \\
			=& \frac{q^{-\left(j+1\right)\left(c-1\right)\binom{c}{2}} \centralCharacter{\depthZeroRepresentation}\left(\uniformizer\right)^j}{\sizeof{\GL_c\left(\finiteField\right)}}\sum_{h \in \GL_c\left(\finiteField\right)}  \standardForm{\tau\left(h\right) \otimes \idmap_{\SpehRepresentation{\tau}{c-1}} f\left(\IdentityMatrix{c^2}\right)}{\firstSpecialFunctional}\standardForm{v}{\tau\left(h^{-1}\right)v^{\vee}}.
		\end{split}
	\end{equation}
	By \Cref{cor:swap-map-for-cuspidal-representations}, we have that \eqref{eq:expression-of-k-c-representations-are-related-gl-c-o-integral} is given by $$\frac{\sizeof{\GL_{c-1}\left(\finiteField\right)} \sizeof{\mathcal{Y}_{1,c-1}\left(\finiteField\right)}}{\sizeof{\GL_c\left(\finiteField\right)}} \cdot q^{\binom{c}{2}}\left(q^c-1\right) \cdot q^{-\left(j+1\right)\left(c-1\right)\binom{c}{2}}  \centralCharacter{\depthZeroRepresentation}\left(\uniformizer\right)^j \standardForm{v \otimes v^{\vee} \otimes f\left(\IdentityMatrix{c^2}\right)}{\secondSpecialFunctional}.$$
	Substituting this expression in \eqref{eq:zeta-integral-as-infinite-series} and using the identity $$\frac{q^{\binom{c}{2}} \left(q^c -1\right)}{q^{\left(c-1\right) \binom{c}{2}} \sizeof{\GL_c\left(\finiteField\right)}} = \frac{q^{-\binom{c}{2}}}{\sizeof{\mathcal{Y}_{1,c-1}\left(\finiteField\right)} \sizeof{\GL_{c-1}\left(\finiteField\right)}}$$ yields the desired result.
\end{proof}

We move to state our result regarding the dual zeta integral. Normalize the Haar measures on $\multiplicativegroup{\localField}$ and on $\GL_c\left(\localField\right)$ as described before \Cref{prop:equality-of-gk-zeta-integrals}. Additionally, normalize the Haar measure on $\Mat{c}{\left(k-2\right)c}\left(\localField\right)$ so that $\Mat{c}{\left(k-2\right)c}\left(\maximalIdeal\right)$ has volume $\sizeof{\squareMatrix_{c\times \left(k-2\right)c}\left(\finiteField\right)}^{-\frac{1}{2}}$. This normalization makes the Fourier transform of $\Mat{c}{(k-2)c}\left(\localField\right)$ self-dual.

\begin{theorem}\label{prop:dual-zeta-gk-evaluation-for-depth-zero}
	Let $v \in \pi$, $v^{\vee} \in \Contragradient{\pi}$ and $W \in \Whittaker\left(\SpehRepresentation{\tau}{c}, \fieldCharacterkc{k}{c}\right)$.
	\begin{enumerate}
		\item Suppose that $\pi \ncong \Contragradient{\tau}$. Then $$\dualZetaOperator\left(s,\Lift v,\Lift v^{\vee}, \Lift W;\fieldCharacter\right) =  \standardForm{\dualZetaOperator\left(W, \pi \times \tau \right) v}{v^{\vee}}.$$
		\item \label{item:dual-zeta-integral-exceptional-case}  Suppose that $\pi = \Contragradient{\tau}$ (and therefore $k=c$). Then for any $v \in \Contragradient{\tau}$, $v^{\vee} \in \tau$ and $W = W_{\mathcal{L}f}$ for $f \in \SpehRepresentation{\tau}{c} \subset \tau \circ \SpehRepresentation{\tau}{c-1}$ we have \begin{align*}
			\MoveEqLeft[3] \dualZetaOperator\left(s,\Lift v,\Lift v^{\vee}, \Lift W;\fieldCharacter\right) = \standardForm{\dualZetaOperator\left(W, \pi \times \tau \right) v}{v^{\vee}}\\
			& + q^{-\binom{c}{2}} \cdot \frac{\centralCharacter{\Pi}^{-1}\left(\uniformizer\right) \centralCharacter{\depthZeroRepresentation}^{-1}\left(\uniformizer\right)q^{-c\left(1-s\right)}}{1 - \centralCharacter{\Pi}^{-1}\left(\uniformizer\right) \centralCharacter{\depthZeroRepresentation}^{-1}\left(\uniformizer\right) q^{-c(1- s)}} \\
			& \times q^{\frac{-c^2\left(c-2\right)}{2}} \sum_{X \in \Mat{c}{(c-2)c}\left(\finiteField\right)}  \standardForm{v \otimes v^{\vee} \otimes f\left( \begin{pmatrix}
					& \IdentityMatrix{(c-1)c}\\
					\IdentityMatrix{c}
				\end{pmatrix} \begin{pmatrix}
				\IdentityMatrix{c} & & X\\
				& \IdentityMatrix{c}\\
				& & \IdentityMatrix{(c-2)c}
				\end{pmatrix} \right)}{\secondDualSpecialFunctional}.
		\end{align*}
	\end{enumerate} 
\end{theorem}
\begin{proof}
	By \Cref{lem:support-of-matrix-coefficient-of-depth-zero}, we have that $\dualZetaOperator\left(s,\Lift v, \Lift v^{\vee},\Lift W;\fieldCharacter\right)$ is given by \begin{align*}
		 & \int_{\GL_c\left(\ringOfIntegers\right)} \int_{\multiplicativegroup{\localField}} \int_{\squareMatrix_{c \times 
				\left(k-2\right)c}\left(\localField\right)}  \standardForm{\Pi\left(th\right) \Lift v}{\Lift v^{\vee}} \cdot \abs{t}^{c\left(s - 1 + \frac{\left(k-2\right)c + 1}{2}\right)} \\
		& \times \Lift W\left( \diag\left(
			\IdentityMatrix{\left(k-1\right)c},
			t h\right)
		\begin{pmatrix}
			\IdentityMatrix{c}\\
			& \IdentityMatrix{\left(k-2\right)c}\\
			& X & \IdentityMatrix{c}
		\end{pmatrix} \begin{pmatrix}
			& \IdentityMatrix{\left(k-1\right)c}\\
			\IdentityMatrix{c}
		\end{pmatrix} \right) \differential X \mdifferential t \mdifferential h.
	\end{align*}
	Applying \Cref{lem:reduce-integration-over-unipotent-to-ring-of-integers}  to the function $g \mapsto \Lift W \left(g  \left(\begin{smallmatrix}
		& \IdentityMatrix{\left(k-1\right)c}\\
		\IdentityMatrix{c} &
	\end{smallmatrix}\right)\right)$, we have that $\dualZetaOperator\left(s,\Lift v, \Lift v^{\vee},\Lift W;\fieldCharacter\right)$ is given by 
	\begin{align*}
		& \int_{\GL_c\left(\ringOfIntegers\right)} \int_{\multiplicativegroup{\localField}} \int_{\squareMatrix_{c \times 
				\left(k-2\right)c}\left(\ringOfIntegers\right)}  \standardForm{\Pi\left(th\right) \Lift v}{\Lift v^{\vee}} \cdot \abs{t}^{c\left(s - 1 + \frac{\left(k-2\right)c + 1}{2}\right)} \\
		& \times \Lift W\left( \diag\left(
		\IdentityMatrix{\left(k-1\right)c},
		t h\right)
		\begin{pmatrix}
			\IdentityMatrix{c}\\
			& \IdentityMatrix{\left(k-2\right)c}\\
			& X & \IdentityMatrix{c}
		\end{pmatrix} \begin{pmatrix}
			& \IdentityMatrix{\left(k-1\right)c}\\
			\IdentityMatrix{c}
		\end{pmatrix} \right) \differential X \mdifferential t \mdifferential h.
	\end{align*}
	
	By \Cref{prop:support-of-diagonal-whittaker-elements} applied to the function $g \mapsto \Lift W\left(g \left(\begin{smallmatrix}
		& \IdentityMatrix{\left(k-1\right)c}\\
		\IdentityMatrix{c}
	\end{smallmatrix}\right)\right)$, we have that the domain of integration of $t$ is $\abs{t} \ge 1$. Hence $\dualZetaOperator\left(s,\Lift v, \Lift v^{\vee},\Lift W;\fieldCharacter\right)$ is given by
	\begin{equation}\label{eq:series-for-dual-kaplan-zeta-integral}
		 \begin{split}
		 	&\sum_{j = 0}^{\infty} \centralCharacter{\Pi}\left(\uniformizer\right)^{-j} q^{jc\left(s-1 + \frac{\left(k-2\right)c + 1}{2}\right)} \int_{\GL_c\left(\ringOfIntegers\right)} \int_{\multiplicativegroup{\ringOfIntegers}} \int_{\squareMatrix_{c \times 
		 			\left(k-2\right)c}\left(\ringOfIntegers\right)} \standardForm{\Pi\left(th\right) \Lift v}{\Lift v^{\vee}} \\
		 	& \times \Lift W\left( \diag\left(\IdentityMatrix{\left(k-1\right)c}, \uniformizer^{-j} th\right) \begin{pmatrix}
		 		\IdentityMatrix{c}\\
		 		& \IdentityMatrix{\left(k-2\right)c}\\
		 		& X & \IdentityMatrix{c}
		 	\end{pmatrix} \begin{pmatrix}
		 		& \IdentityMatrix{\left(k-1\right)c}\\
		 		\IdentityMatrix{c}
		 	\end{pmatrix} \right) \differential X \mdifferential t \mdifferential h.
		 \end{split}
	\end{equation}
	We may omit the integration over $t$ by changing variables. 

	Notice that for $j=0$, we have from \Cref{prop:relation-between-k-c-whittaker-models} and from the second part of \Cref{lem:support-of-matrix-coefficient-of-depth-zero} that the summand \begin{align*}
		&\int_{\GL_c\left(\ringOfIntegers\right)} \int_{\squareMatrix_{c \times 
				\left(k-2\right)c}\left(\ringOfIntegers\right)} \standardForm{\Pi\left(h\right) \Lift v}{\Lift v^{\vee}} \\
		& \times \Lift W\left( \begin{pmatrix}
			\IdentityMatrix{\left(k-1\right)c}\\
			& h
		\end{pmatrix} \begin{pmatrix}
			\IdentityMatrix{c}\\
			& \IdentityMatrix{\left(k-2\right)c}\\
			& X & \IdentityMatrix{c}
		\end{pmatrix} \begin{pmatrix}
			& \IdentityMatrix{\left(k-1\right)c}\\
			\IdentityMatrix{c}
		\end{pmatrix} \right) \differential X \mdifferential h
	\end{align*}
	evaluates to $\standardForm{\dualZetaOperator\left(W, \pi \times \tau\right) v}{v^{\vee}}$ (recall that the Haar measure is normalized such that $\Mat{c}{(k-2)c}\left(\maximalIdeal\right)$ has volume $q^{-\frac{c^2\left(k-2\right)}{2}}$).

	If $\pi \ncong \Contragradient{\tau}$, then by applying \Cref{lem:integral-vanishes-depth-zero-non-isomorphic} part \ref{item:vanishing-lemma-for-second-integral} to the $(k,c)$ $\fieldCharacter$-Whittaker function $$g \mapsto \int_{\Mat{c}{\left(k-2\right)c}\left(\ringOfIntegers\right)} \Lift W\left( g \begin{pmatrix}
	\IdentityMatrix{c}\\
	& \IdentityMatrix{\left(k-2\right)c}\\
	& X & \IdentityMatrix{c}
\end{pmatrix} \begin{pmatrix}
	& \IdentityMatrix{\left(k-1\right)c}\\
	\IdentityMatrix{c}
\end{pmatrix} \right) \differential X,$$ we get that the summand of \eqref{eq:series-for-dual-kaplan-zeta-integral} vanishes for $j > 0$, and hence the result the follows.

If $\pi = \Contragradient{\tau}$, then the proof proceeds similarly to the proof of \Cref{prop:equality-of-gk-zeta-integrals} part \ref{item:gk-zeta-integral-exceptional-case}.

\end{proof}

As a corollary, we get a relation between the Ginzburg--Kaplan gamma factors over the finite field and over the local field.

\begin{corollary}\label{cor:ginzburg-kaplan-gamma-factor-equality}Suppose that $\pi \ncong \Contragradient{\tau}$. Then
	$$\LocalGKGammaFactor{s}{\Pi}{\depthZeroRepresentation}{\fieldCharacter} = \GKGammaFactor{\pi}{\tau}{\fieldCharacter}.$$
\end{corollary}

\begin{theorem}\label{thm:pi-equals-tau-dual}
	Suppose that $\pi = \Contragradient{\tau}$ (and therefore $k = c$). Then $$\GKGammaFactor{\pi}{\tau}{\fieldCharacter} = -\centralCharacter{\tau}\left(-1\right)^{c-1} q^{-\frac{c}{2}}.$$
\end{theorem}
\begin{proof}
	Take $W = \besselSpehFunction{\tau}{c}$. We have that for any $h \in \GL_c\left(\ringOfIntegers\right)$ and any $j \ge 1$, $$\Lift W \left(\diag\left(\uniformizer^j h, \IdentityMatrix{\left(c-1\right)c}\right)\right) = 0.$$
	Indeed, for any $X \in \squareMatrix_c\left(\ringOfIntegers\right)$ if $u_X = \diag\left(\left(\begin{smallmatrix}
		\IdentityMatrix{c} & X\\
		& \IdentityMatrix{c}\\
	\end{smallmatrix}\right), \IdentityMatrix{(c-2)c}\right)$ then we have $$\SpehRepresentation{\depthZeroRepresentation}{c}\left(u_X\right) \Lift W = \Lift \left(\SpehRepresentation{\tau}{c}\left(\quotientMap\left(u_X\right)\right) W\right) = \fieldCharacter\left(\trace X\right) \Lift W.$$
	Substituting $\diag\left(\uniformizer^j h, \IdentityMatrix{(c-1)c}\right)$ yields
	$$\Lift W\left( \diag\left(\uniformizer^j h, \IdentityMatrix{(c-1)c}\right) u_X \right) = \fieldCharacter\left(\trace X\right) \Lift W\left(\diag\left(\uniformizer^j h, \IdentityMatrix{(c-1)c}\right)\right).$$
	On the other hand, the equality $$\diag\left(\uniformizer^j h, \IdentityMatrix{(c-1)c}\right) u_X = u_{\uniformizer^j h X} \cdot \diag\left(\uniformizer^j h, \IdentityMatrix{(c-1)c}\right)$$ yields that $$\Lift W\left( \diag\left(\uniformizer^j h, \IdentityMatrix{(c-1)c}\right) u_X \right) = \Lift W\left( \diag\left(\uniformizer^j h, \IdentityMatrix{(c-1)c}\right)\right).$$
	Therefore, by choosing $X$ such that $\fieldCharacter\left(\trace X\right) \ne 0$, we get that $\Lift W\left( \diag\left(\uniformizer^j h, \IdentityMatrix{(c-1)c}\right)\right)$ must be zero. A similar argument shows that $\Lift W \left(\diag\left(h, \IdentityMatrix{\left(c-1\right)c}\right)\right) = 0$ for $h \in \GL_c\left(\ringOfIntegers\right)$ such that $h \ne \IdentityMatrix{c}$ (see \cite[Proposition 2.2]{CarmonZelingher2024}).
	
	Let $v_{\tau, \fieldCharacter} \in \tau$ and $v_{\Contragradient{\tau}, \fieldCharacter^{-1}} \in \Contragradient{\tau}$ be a $\fieldCharacter$-Whittaker vector and a $\fieldCharacter^{-1}$-Whittaker vector, respectively, normalized such that $\standardForm{v_{\tau, \fieldCharacter}}{v_{\Contragradient{\tau}, \fieldCharacter^{-1}}} = 1$. Substituting $W = \besselSpehFunction{\tau}{c}$, $v = v_{\Contragradient{\tau}, \fieldCharacter}$ and $v^{\vee} = v_{\tau, \fieldCharacter}$, we get that $$\zetaOperator\left(s, \Lift v, \Lift v^{\vee}, \Lift W; \fieldCharacter\right) = \frac{1}{\sizeof{\GL_c\left(\finiteField\right)}}.$$
	Therefore, we have that $$\LocalGKGammaFactor{s}{\Pi}{\depthZeroRepresentation}{\fieldCharacter} = \centralCharacter{\pi}\left(-1\right)^{c-1} \sizeof{\GL_c\left(\finiteField\right)}\dualZetaOperator\left(s, \Lift v, \Lift v^{\vee}, \Lift W; \fieldCharacter\right).$$
	Since $W$ is invariant under right translations by matrices of the form $\left(\begin{smallmatrix}
		\IdentityMatrix{c} & & X\\
		& \IdentityMatrix{c}\\
		& & \IdentityMatrix{(c-2)c}
	\end{smallmatrix}\right)$, where $X \in \Mat{c}{(c-2)c}\left(\ringOfIntegers\right)$, we get from \Cref{prop:dual-zeta-gk-evaluation-for-depth-zero} part \ref{item:dual-zeta-integral-exceptional-case} that
	\begin{align*}
		\MoveEqLeft[3] \dualZetaOperator\left(s,\Lift v,\Lift v^{\vee}, \Lift W;\fieldCharacter\right) = \standardForm{\dualZetaOperator\left(W, \pi \times \tau \right) v}{v^{\vee}}\\
		& + q^{\left(c-1\right) \binom{c}{2}} \cdot  q^{-\frac{c^2}{2}} \cdot \frac{\centralCharacter{\Pi}^{-1}\left(\uniformizer\right) \centralCharacter{\depthZeroRepresentation}^{-1}\left(\uniformizer\right)q^{-c\left(1-s\right)}}{1 - \centralCharacter{\Pi}^{-1}\left(\uniformizer\right) \centralCharacter{\depthZeroRepresentation}^{-1}\left(\uniformizer\right) q^{-c(1- s)}} \standardForm{v \otimes v^{\vee} \otimes f \begin{pmatrix}
				& \IdentityMatrix{(c-1)c}\\
				\IdentityMatrix{c}
			\end{pmatrix}}{\secondDualSpecialFunctional},
	\end{align*}
	where $f \in \SpehRepresentation{\tau}{c} \subset \SpehRepresentation{\tau}{c-1} \circ \tau$ is chosen so that $$\besselSpehFunction{\tau}{c}\left(g\right) = \standardForm{\SpehRepresentation{\tau}{c}\left(g\right) f}{\gShortSpehWhittakerFunctional{\tau}{c}{c}}$$ for any $g \in \GL_{c^2}\left(\finiteField\right)$.
	On the other hand, since $\LocalGKGammaFactor{s}{\Pi}{\depthZeroRepresentation}{\fieldCharacter}$ equals the Jacquet--Piatetski-Shapiro--Shalika gamma factor defined in \cite{Jacquet1983rankin}, we have by \cite[Remark after Corollary 4.3]{Ye18} that\footnote{A factor of $q^{-\frac{c}{2}}$ is missing in \cite{Ye18}. This factor should appear in the definition of the Fourier transform in order to make it self-dual. Also, a factor of $\centralCharacter{\depthZeroRepresentation}\left(-1\right)^{c-1}$ is missing in the definition of the gamma factor in \cite{Ye18}.} $$\LocalGKGammaFactor{s}{\Pi}{\depthZeroRepresentation}{\fieldCharacter} = \centralCharacter{\tau}\left(-1\right)^{c-1} \frac{q^{c\left(s-\frac{1}{2}\right)}}{\centralCharacter{\Pi}\left(\uniformizer\right) \centralCharacter{\depthZeroRepresentation}\left(\uniformizer\right)} \cdot \frac{1 - \centralCharacter{\Pi}\left(\uniformizer\right) \centralCharacter{\depthZeroRepresentation}\left(\uniformizer\right) q^{-cs}}{1 - \centralCharacter{\Pi}^{-1}\left(\uniformizer\right) \centralCharacter{\depthZeroRepresentation}^{-1}\left(\uniformizer\right) q^{-c\left(1-s\right)}},$$
	which can also be written as $$ \LocalGKGammaFactor{s}{\Pi}{\depthZeroRepresentation}{\fieldCharacter} = \centralCharacter{\pi}\left(-1\right)^{c-1} q^{\frac{c}{2}} \left(-1 + \frac{1-q^{-c}}{1-\centralCharacter{\Pi}^{-1}\left(\uniformizer\right) \centralCharacter{\depthZeroRepresentation}^{-1}\left(\uniformizer\right) q^{-c\left(1-s\right)}}\right).$$
	One can show that $$\GKGammaFactor{\pi}{\tau}{\fieldCharacter} = \centralCharacter{\pi}\left(-1\right)^{c-1} \sizeof{\GL_c\left(\finiteField\right)} \standardForm{\dualZetaOperator\left(W, \pi \times \tau\right) v}{v^{\vee}} ,$$ see also \cite[Section 4.3.3]{CarmonZelingher2024}. Hence, we get by comparing the constant coefficient of the Taylor expansion of $\LocalGKGammaFactor{s}{\Pi}{\depthZeroRepresentation}{\fieldCharacter}$ (in the variable $q^{-c\left(1-s\right)}$) that $$\GKGammaFactor{\pi}{\tau}{\fieldCharacter} = \centralCharacter{\pi}\left(-1\right)^{c-1} q^{\frac{c}{2}}\left(-1 + 1-q^{-c}\right) = -\centralCharacter{\pi}\left(-1\right)^{c-1} q^{-\frac{c}{2}},$$
	and by comparing any other non-zero coefficient that
	$$\sizeof{\GL_c\left(\finiteField\right)} q^{\left(c-1\right) \binom{c}{2}} \cdot q^{-\frac{c^2}{2}} \standardForm{v \otimes v^{\vee} \otimes f \begin{pmatrix}
		& \IdentityMatrix{(c-1)c}\\
		\IdentityMatrix{c}
\end{pmatrix}}{\secondDualSpecialFunctional} = q^{\frac{c}{2}} - q^{-\frac{c}{2}},$$
	which implies that	
	$$q^{\left(c-1\right) \binom{c}{2}} \standardForm{v \otimes v^{\vee} \otimes f \begin{pmatrix}
		& \IdentityMatrix{(c-1)c}\\
		\IdentityMatrix{c}
\end{pmatrix}}{\secondDualSpecialFunctional} = \frac{1}{\dim \tau}.$$		
\end{proof}

As another corollary, we obtain the modified functional equation for the case $\pi = \Contragradient{\tau}$.
\begin{corollary}\label{cor:modified-functional-equation}
	Suppose that $\pi = \Contragradient{\tau}$ (and therefore, $k=c$) and denote $T = \centralCharacter{\Pi}\left(\uniformizer\right) \centralCharacter{\depthZeroRepresentation}\left(\uniformizer\right) q^{-cs}$. Then for any $v \in \Contragradient{\tau}$, $v^{\vee} \in \tau$ and $W \in \Whittaker\left(\SpehRepresentation{\tau}{c}, \fieldCharacterkc{c}{c}\right)$, such that $W = W_f$ for $f \in \SpehRepresentation{\tau}{c}$, the following modified functional equation holds:
	\begin{equation*}
		\begin{split}
			&\left(1 - T\right) \standardForm{\zetaOperator\left(W, \pi \times \tau \right) v}{v^{\vee}} + q^{-\binom{c}{2}} \cdot  \standardForm{v \otimes v^{\vee} \otimes f_{\tau \circ \SpehRepresentation{\tau}{c-1}}\left(\IdentityMatrix{c^2}\right)}{\secondSpecialFunctional} T\\
			=& \left(q^{\frac{c}{2}}T - q^{-\frac{c}{2}}\right)\standardForm{\dualZetaOperator\left(W, \pi \times \tau \right) v}{v^{\vee}}\\
			& + q^{-c \cdot \binom{c}{2}} \sum_{X \in \Mat{c}{(c-2)c}\left(\finiteField\right)}  \standardForm{v \otimes v^{\vee} \otimes f_{\SpehRepresentation{\tau}{c-1} \circ \tau}\left( \begin{pmatrix}
					& \IdentityMatrix{(c-1)c}\\
					\IdentityMatrix{c}
				\end{pmatrix} \begin{pmatrix}
					\IdentityMatrix{c} & & X\\
					& \IdentityMatrix{c}\\
					& & \IdentityMatrix{(c-2)c}
				\end{pmatrix} \right)}{\secondDualSpecialFunctional}.
		\end{split}
	\end{equation*}	
	Therefore, we have the following simple modified functional equation:
		\begin{align*}
			\standardForm{\dualZetaOperator\left(W, \pi \times \tau \right) v}{v^{\vee}} =& \GKPreGammaFactor{\pi}{\tau}{\fieldCharacter} \left(\standardForm{\zetaOperator\left(W, \pi \times \tau \right) v}{v^{\vee}}\right) \\
			&+ q^{-\frac{c^2}{2}} \cdot  \standardForm{v \otimes v^{\vee} \otimes f_{\tau \circ \SpehRepresentation{\tau}{c-1}}\left(\IdentityMatrix{c^2}\right)}{\secondSpecialFunctional}.
		\end{align*}
	(Recall that $\GKPreGammaFactor{\pi}{\tau}{\fieldCharacter} = -q^{-\frac{c}{2}}$ in this case).
	
	Here, $f_{\tau \circ \SpehRepresentation{\tau}{c-1}}$ (respectively, $f_{\SpehRepresentation{\tau}{c-1} \circ \tau}$) is the element representing $f$ in $\tau \circ \SpehRepresentation{\tau}{c-1}$ (respectively, in $\SpehRepresentation{\tau}{c-1} \circ \tau$), given by $f_{\tau \circ \SpehRepresentation{\tau}{c-1}}\left(g\right)\left(g'\right) = f\left(\diag\left(\IdentityMatrix{c}, g'\right)g\right)$ (respectively, by $f_{\SpehRepresentation{\tau}{c-1} \circ \tau}\left(g\right)\left(g'\right) = f\left(\diag\left(g', \IdentityMatrix{c}\right)\right) g$) for $g' \in \GL_{(c-1)c}\left(\finiteField\right)$ and $g \in \GL_{c^2}\left(\finiteField\right)$.
\end{corollary}

We use the results of Theorems \ref{prop:equality-of-gk-zeta-integrals} and \ref{prop:dual-zeta-gk-evaluation-for-depth-zero} to obtain the following identity. This identity is \cite[Theorem 4.24]{CarmonZelingher2024} and is crucial for the proofs of the results of \cite[Section 5]{CarmonZelingher2024}. See \cite{ye2021epsilon} for the definition of $\varepsilon_0\left(\pi \times \tau, \fieldCharacter\right)$.
\begin{theorem}\label{cor:gamma-factor-equality-with-epsilon-factor}
	Let $\pi$ and $\tau$ be irreducible cuspidal representations of $\GL_c\left(\finiteField\right)$ and $\GL_k\left(\finiteField\right)$, respectively. Then $$\GKGammaFactor{\pi}{\tau}{\fieldCharacter} = \varepsilon_0\left(\pi \times \tau, \fieldCharacter\right).$$
\end{theorem}
\begin{proof}
	When $\pi$ is isomorphic to $\Contragradient{\tau}$, by the proof of \cite[Theorem 2.18]{zelingher2022values} we have that $$\varepsilon_0\left(\pi \times \tau, \fieldCharacter\right) = -\centralCharacter{\tau}\left(-1\right)^{c-1} q^{-\frac{c}{2}}.$$ Hence the result follows in this case from \Cref{thm:pi-equals-tau-dual}. When $\pi$ is not isomorphic to $\Contragradient{\tau}$, we have that the Jacquet--Piatetski-Shapiro--Shalika gamma factor \cite{Jacquet1983rankin} equals the Jacquet--Piatetski-Shapiro--Shalika epsilon factor $\varepsilon\left(s, \Pi \times \depthZeroRepresentation, \fieldCharacter\right)$ \cite{Jacquet1983rankin}, and by the proof of \cite[Theorem 4.4]{ye2021epsilon}, this equals $\varepsilon_0\left(\pi \times \tau, \fieldCharacter\right)$. The result now follows from \Cref{cor:ginzburg-kaplan-gamma-factor-equality} and from the equality of the Ginzburg--Kaplan gamma factor and the Jacquet--Piatetski-Shapiro--Shalika gamma factor.
\end{proof}

\appendix

\section{A variant of Schur's lemma}

Throughout the text, the following variant of Schur's lemma is used.

\begin{proposition}
	Let $\tau$ and $\pi$ be irreducible representations of a finite group $G$. Then
	$$\frac{1}{\sizeof{G}}\sum_{g \in G} \tau\left(g\right) \otimes \pi\left(g^{-1}\right) = \begin{dcases}
		0 & \tau \ncong \pi\\
		\frac{1}{\dim \tau} \mathrm{sw} & \tau = \pi
	\end{dcases}, $$
	where $\mathrm{sw} \colon \tau \otimes \tau \to \tau \otimes \tau$ is the swap linear map, acting on pure tensors by $$\mathrm{sw} \left(v \otimes v'\right) = v' \otimes v,$$
	where $v, v' \in \tau$.
\end{proposition}
\begin{proof}
	Denote $$T = \frac{1}{\sizeof{G}}\sum_{g \in G} \tau\left(g\right) \otimes \pi\left(g^{-1}\right).$$
	Then for any $h_1, h_2 \in G$, $$T \circ \left(\tau\left(h_1\right) \otimes \pi\left(h_2\right)\right) = \left(\tau\left(h_2\right) \otimes \pi\left(h_1\right)\right)\circ T =  \mathrm{sw} \circ \left(\pi\left(h_1\right) \otimes \tau\left(h_2\right)\right)\circ \mathrm{sw} \circ T.$$
	Therefore, we have that $\mathrm{sw} \circ T \in \Hom_{G \times G}\left(\tau \otimes \pi, \pi \otimes \tau\right)$. If $\tau \ncong \pi$, then by Schur's lemma the latter $\Hom$-space is zero. If $\tau = \pi$, we have that $\mathrm{sw} \circ T = c \cdot \idmap_{\tau \otimes \tau}$, where $c \in \cComplex$ is a scalar, which is equivalent to $T = c \cdot \mathrm{sw}$. To find the scalar, we take the trace, and using Schur's orthogonality relations, we get that $$c \cdot \trace \mathrm{sw} = \trace \left(T\right) = \frac{1}{\sizeof{G}} \sum_{g \in G} \trace \tau\left(g\right) \trace \tau\left(g^{-1}\right) = 1.$$
	The result now follows because $\trace \mathrm{sw} = \dim \pi$.
\end{proof}
Using the well-known fact that every irreducible cuspidal representation of $\GL_k\left(\finiteField\right)$ is of dimension $\grpIndex{\GL_{k-1}\left(\finiteField\right)}{\UnipotentSubgroup_{k-1}\left(\finiteField\right)}$, we obtain the following result.
\begin{corollary}\label{cor:swap-map-for-cuspidal-representations}
	If $\tau$ is an irreducible cuspidal representation of $\GL_k\left(\finiteField\right)$ then
	$$\frac{1}{\sizeof{\GL_k\left(\finiteField\right)}}  \sum_{g \in \GL_k\left(\finiteField\right)} \tau\left(g\right) \otimes \tau\left(g^{-1}\right) = \frac{q^{\binom{k}{2}} \cdot \left(q^k - 1\right)}{\sizeof{\GL_k\left(\finiteField\right)}} \cdot \mathrm{sw}.$$
\end{corollary}

\section{Root exchange algorithm}\label{appendix:root-exchange}

\newcommand{\rootEmbeddingOne}{\iota'}
\newcommand{\rootEmbeddingTwo}{\iota}

We explain the root exchange algorithm commonly used to reduce domains of integration. Our treatment is informal, and we refer to \cite[Section 6.1]{Cai2019} and \cite[Section 7.1]{GinzburgRallisSoudry2011} for the precise statements.

As always, assume that $\fieldCharacter \colon \localField \to \multiplicativegroup{\cComplex}$ has conductor $\maximalIdeal$. Let $G = \GL_n\left(\localField\right)$ for some $n$. The root exchange algorithm concerns integrals of the form \begin{equation}\label{eq:root-exchange-integral}
	I_l\left(f\right) = \int_{\Mat{a_1}{b_1}\left(\localField\right)} \dots \int_{\Mat{a_l}{b_l}\left(\localField\right)} f\left(r_1\left(y_1\right) \dots r_l\left(y_l\right)\right) \differential y_l \dots \differential y_1,
\end{equation}
where $f \colon G \to \cComplex$ is a right smooth function and for every $1 \le i \le l$, $$r_i \colon \left(\Mat{a_i}{b_i}\left(\localField\right), +\right) \to \left(G, \cdot\right)$$ is a continuous injective homomorphism of groups. We make the additional assumption that $$r_i\left(\Mat{a_i}{b_i}\left( \maximalIdeal^j \right)\right) \subset \IdentityMatrix{n} + \Mat{a_i}{b_i}\left(\maximalIdeal^j\right)$$ for every $j \in \zIntegers$. Integrals over unipotent subgroups can be realized in this form.

The idea of the root exchange algorithm is to use embeddings $r^{\ast}_i \colon \Mat{b_i}{a_i}\left(\localField\right) \to G$ and make use of the right smoothness of $f$ in order to reduce the integration domain in \eqref{eq:root-exchange-integral}. In order for this to be successful, we need $f$ to satisfy the following relation: for any $1 \le i \le l$, any $g_0 \in G$, $y_i \in \Mat{a_i}{b_i}\left(\localField\right)$, and any $x_i \in \Mat{b_i}{a_i}\left(\localField\right)$,
\begin{equation}\label{eq:condition-for-root-exchange}
	\begin{split}
		& \int_{\Mat{a_1}{b_1}\left(\localField\right)} \dots \int_{\Mat{a_{i-1}}{b_{i-1}}\left(\localField\right)} f\left(r_1\left(y_1\right) \dots r_i\left(y_i\right) r_i^{\ast}\left(x_i\right) g_0\right) \differential y_{i-1} \dots \differential y_1 \\
		=& \fieldCharacter\left(\trace\left(y_i x_i\right)\right) \int_{\Mat{a_1}{b_1}\left(\localField\right)} \dots \int_{\Mat{a_{i-1}}{b_{i-1}}\left(\localField\right)} f\left(r_1\left(y_1\right) \dots r_i\left(y_i\right) g_0\right) \differential y_{i-1} \dots \differential y_1.
	\end{split}
\end{equation}
The algorithm goes as follows. Let $f_l = f$.
\begin{enumerate}
	\item Suppose that $f_j$ is defined so that $I_l\left(f\right) = I_{j}\left(f_j\right)$. Suppose that $f_j$ is right smooth and satisfies the condition \eqref{eq:condition-for-root-exchange} for $1 \le i \le j$. Suppose that $f_j$ is right invariant under $\IdentityMatrix{n} + \squareMatrix_n\left(\maximalIdeal^{m_j}\right)$ for $m_j \ge 1$. Then \begin{equation}\label{eq:f-j-is-right-smooth}
		f_j\left(g\right) = \frac{1}{\VolumeOf\left(\Mat{b_{j}}{a_{j}}\left(\maximalIdeal^{m_j}\right)\right)} \int_{\Mat{b_{j}}{a_{j}}\left(\maximalIdeal^{m_j}\right)} f_j\left(g r_{j}\left(x_j\right)\right) \differential x_{j}.
	\end{equation}
	\item \label{item:root-exchange-step-two}Substitute \eqref{eq:f-j-is-right-smooth} in $I_{j}\left(f_j\right)$ to and use \eqref{eq:condition-for-root-exchange} (with $g_0 = \IdentityMatrix{n}$) to get \begin{equation}
		\begin{split}
			I_{j}\left(f_j\right) =& \frac{1}{\VolumeOf\left(\Mat{b_j}{a_j}\left(\maximalIdeal^{m_j}\right)\right)} \int_{\Mat{a_1}{b_1}\left(\localField\right)} \dots \int_{\Mat{a_{j}}{b_{j}}\left(\localField\right)}\int_{\Mat{b_j}{a_j}\left(\maximalIdeal^{m_j}\right)}\\
			& \times \fieldCharacter\left(\trace\left(y_j x_j\right)\right) f_j\left(r_1\left(y_1\right) \dots r_j\left(y_j\right)\right) \differential x_j \differential y_j \dots \differential y_1.
		\end{split}
	\end{equation}
	\item \label{item:root-exchange-step-three}The inner integral $$\frac{1}{\VolumeOf\left(\Mat{b_j}{a_j}\left(\maximalIdeal^{m_j}\right)\right)} \int_{\Mat{b_j}{a_j}\left(\maximalIdeal^{m_j}\right)}\fieldCharacter\left(\trace\left(y_j x_j\right)\right) \differential x_j$$
	is (up to scalar multiplication) the Fourier transform of the characteristic function of $\Mat{b_j}{a_j}\left(\maximalIdeal^{m_j}\right)$ (viewed as a function in the variable $y_j$). It equals the characteristic function of $\maximalIdeal^{-m_j + 1}$. Hence $I_j\left(f_j\right)$ is given by $$\int_{\Mat{a_1}{b_1}\left(\localField\right)} \dots \int_{\Mat{a_{j-1}}{b_{j-1}}\left(\localField\right)} \int_{\Mat{a_j}{b_j}\left(\maximalIdeal^{-m_j+1}\right)} f_j \left(r_1\left(y_1\right) \dots r_j\left(y_j\right)\right) \differential y_j \dots \differential y_l.$$
	Define $$f_{j-1}\left(g\right) = \int_{\Mat{a_j}{b_j}\left(\maximalIdeal^{-m_j+1}\right)} f_{j}\left(g r_j\left(y_j\right)\right) \differential y_j.$$
	Then $I_j\left(f_j\right) = I_{j-1}\left(f_{j-1}\right)$ and $f_{j-1}$ is right smooth and satisfies \eqref{eq:condition-for-root-exchange} for $1 \le i \le j-1$.
	\item Repeatedly run the algorithm steps to eventually get $I_l\left(f\right) = I_0\left(f_0\right) = f_0\left(\IdentityMatrix{n}\right)$, which is the integral
	$$ \int_{\Mat{a_1}{b_1}\left(\maximalIdeal^{-m_1+1}\right)} \dots \int_{\Mat{a_l}{b_l}\left(\maximalIdeal^{-m_l+1}\right)} f\left(r_1\left(y_1\right) \dots r_l\left(y_l\right)\right) \differential y_l \dots \differential y_1.$$
\end{enumerate}

We end with a few useful remarks. The first one is that if the images of $r^{\ast}_1, \dots, r^{\ast}_l$ commute, then we may choose $m_1 = m_2 = \dots = m_l = m$, where $f$ is right invariant under $\IdentityMatrix{n} + \squareMatrix_n\left(\maximalIdeal^m\right)$.

The second remark is there exists a slight generalization of this algorithm, where instead of using the fact that $f_j$ is right smooth, we use the fact that \begin{equation}\label{eq:f-j-is-right-smooth-alternative}
	f_j\left(g\right) = \int_{\Mat{b_j}{a_j}\left(\localField\right)} f'_j\left(g r^{\ast}_j\left(x_j\right)\right) \varphi_j\left(x_j\right) \differential x_j,
\end{equation}
where $f'_j$ satisfies similar properties to $f_j$ and $\varphi_j \colon \Mat{b_j}{a_j}\left(\localField\right) \to \cComplex$ is a compactly supported smooth function. In this generalization, \eqref{eq:f-j-is-right-smooth-alternative} serves an alternative to \eqref{eq:f-j-is-right-smooth}, and after applying steps \ref{item:root-exchange-step-two} and \ref{item:root-exchange-step-three}, we get $$I_j\left(f_j\right) = I_{j-1}\left(f_{j-1}\right),$$ where $$f_{j-1}\left(g\right) = \int_{\Mat{a_j}{b_j}\left(\localField\right)} f'_j\left(g r_j\left(y_j\right) \right) \fourierTransform{\fieldCharacter}{\varphi_j}\left(y_j\right) \differential y_j,$$
and where $$\fourierTransform{\fieldCharacter}{\varphi_j}\left(y_j\right) = \int_{\Mat{b_j}{a_j}\left(\localField\right)} \varphi_j\left(x_j\right) \fieldCharacter\left(\trace\left(x_j y_j\right)\right) \differential x_j.$$
The advantage of this generalized version is that it allows one to apply the algorithm in the context of archimedean fields and finite fields.

\subsection{Reduction of integral domain for $(k,c)$ $\fieldCharacter$-functionals}

In this subsection, we explain the choice of $r_i$ and $r^{\ast}_i$ for the integral \eqref{eq:recursive-formula-for-k-c-whittaker-function} from \Cref{subsec:lifts-of-whittaker-functions}, introduced in \cite[Lemma 9]{CaiFriedbergGourevitchKaplan2023}.

Recall that
\begin{equation*}
	\mathcal{Y}\left(R\right) = \left\{ \begin{pmatrix}
		\IdentityMatrix{k c_1}\\
		Y & \IdentityMatrix{k c_2} \end{pmatrix} \mid Y = \begin{pmatrix}
		0_{c_2 \times c_1}& y_{12} & y_{13} & \dots & y_{1k}\\
		&0_{c_2 \times c_1} & y_{23} & \dots & y_{2k}\\
		& & \ddots & \ddots & \vdots\\
		& & & 0_{c_2 \times c_1} & y_{k-1,k}\\
		& & & & 0_{c_2 \times c_1}
	\end{pmatrix} \mid y_{ij} \in \Mat{c_2}{c_1}\left(R\right) \right\}.
\end{equation*}
Let $$\mathcal{X}\left(R\right) = \left\{ \begin{pmatrix}
	\IdentityMatrix{k c_1} & X\\
	& \IdentityMatrix{k c_2}
\end{pmatrix} \mid X =\begin{pmatrix}
	0_{c_1 \times c_2} \\
	0_{c_1 \times c_2} & x_{12} &\\
	0_{c_1 \times c_2} & x_{13} & x_{23}\\
	\vdots & \vdots & \vdots & \ddots\\
	0_{c_1 \times c_2} & x_{1k} & x_{2 k} & \cdots & x_{k-1,k}
\end{pmatrix}  \mid x_{ij} \in \Mat{c_1}{c_2}\left(R\right) \right\}.$$
Let $\mathcal{Y}_{ij}\left(R\right)$ (respectively, $\mathcal{X}_{ij}\left(R\right)$) be the subgroup of $\mathcal{Y}\left(R\right)$ (respectively, $\mathcal{X}\left(R\right)$) consisting of elements such that $y_{i'j'} = 0_{c_2 \times c_1}$ (respectively, $x_{i'j'} = 0_{c_1 \times c_2}$) for $\left(i',j'\right) \ne \left(i,j\right)$. Let $r^{+}_{i,j} \colon \Mat{c_1}{c_2}\left(R\right) \to \mathcal{X}_{ij}\left(R\right)$, $r^{-}_{i,j} \colon \Mat{c_2}{c_1}\left(R\right) \to \mathcal{Y}_{ij}\left(R\right)$ be the obvious maps. 

Let $\left(i_1, j_1\right), \dots, (i_{\binom{k}{2}}, j_{\binom{k}{2}})$ be the sequence of pairs $\left(k-1,k\right)$, $\left(k-2, k-1\right)$, $\dots$, $\left(1,2\right)$, $\left(k-2, k\right)$, $\left(k-3, k-1\right)$, $\dots$, $\left(1,3\right)$, $\dots$, $\left(1, k\right)$.

Recall \Cref{prop:integrand-for-k-c-functional}.

\begin{proposition}
	For any $\Phi \in \SpehRepresentation{\depthZeroRepresentation}{c} \subset \abs{\det}^{-\frac{c_2}{2}} \SpehRepresentation{\depthZeroRepresentation}{c_1} \times \abs{\det}^{\frac{c_1}{2}} \SpehRepresentation{\depthZeroRepresentation}{c_2}$, such that $\Phi$ is right invariant under $\IdentityMatrix{kc} + \squareMatrix_{kc}\left(\maximalIdeal^{r+1}\right)$ for $r \ge 0$, we have $$\standardForm{\Phi}{\gShortSpehWhittakerFunctional{\depthZeroRepresentation}{k}{c}} = \int_{\mathcal{Y}\left(\uniformizer^{-r} \ringOfIntegers \right)} \standardForm{\Phi\left( y \kappa \right)}{\gShortSpehWhittakerFunctional{\depthZeroRepresentation}{k}{c_1} \otimes \gShortSpehWhittakerFunctional{\depthZeroRepresentation}{k}{c_2}} \differential y.$$
\end{proposition}

The proof of this proposition follows from the root exchange algorithm by taking $r_{\binom{k}{2} + 1 - m} = r^{-}_{i_m, j_m}$ and $r^{\ast}_{\binom{k}{2} + 1 - m} = r^{+}_{i_m, j_m}$ for every $1 \le m \le \binom{k}{2}$. We use the fact that the images of $r^{\ast}_m$ lie in $\mathcal{X}\left(\localField\right)$, which is a commutative group, and hence by the first remark in the previous section, we can use the same domain for all variables. The hardest thing to check is that the relation \eqref{eq:condition-for-root-exchange} is satisfied. The identity follows the fact that $\gShortSpehWhittakerFunctional{\depthZeroRepresentation}{k}{c_1}$ and $\gShortSpehWhittakerFunctional{\depthZeroRepresentation}{k}{c_2}$ are $\left(k, c_1\right)$ and $\left(k, c_2\right)$ $\fieldCharacter$-Whittaker functionals, respectively, and from computing $$\left(\prod_{t = s}^{\binom{k}{2}} r_{i_t, j_t}^{-}\left(y_{i_t, j_t}\right)\right) r_{i_s, j_s}^{+}\left(x_{i_s, j_s}\right) \left(\prod_{t = s}^{\binom{k}{2}} r_{i_t, j_t}^{-}\left(y_{i_t, j_t}\right)\right)^{-1},$$
where one needs to change variables in the integral for some of the variables. See also a demonstration at \url{https://elad.zelingher.com/mathapps/gln/rootexchange}.

\subsection{Reduction of integration domain in the dual Ginzburg--Kaplan integral}
We explain another application of the root exchange algorithm, used in the proof of \Cref{prop:dual-zeta-gk-evaluation-for-depth-zero}.

\begin{lemma}\label{lem:reduce-integration-over-unipotent-to-ring-of-integers}
	Suppose that $W \in \Whittaker\left(\SpehRepresentation{\depthZeroRepresentation}{c}, \fieldCharacterkc{k}{c}\right)$ is invariant under right translations of $\IdentityMatrix{kc} + \squareMatrix_{kc}\left(\maximalIdeal\right)$. Then for any $h \in \GL_c\left(\localField\right)$, \begin{align*}
		& \int_{\Mat{c}{\left(k-2\right)c}\left(\localField\right)} W \left(\diag\left(\IdentityMatrix{\left(k-1\right)c}, h\right) \begin{pmatrix}
			\IdentityMatrix{c}\\
			& \IdentityMatrix{\left(k-2\right)c}\\
			& Y & \IdentityMatrix{c}
		\end{pmatrix}\right) \differential Y \\
		= & \int_{\Mat{c}{\left(k-2\right)c}\left(\ringOfIntegers\right)} W \left(\diag\left(\IdentityMatrix{\left(k-1\right)c},h\right)\begin{pmatrix}
			\IdentityMatrix{c}\\
			& \IdentityMatrix{\left(k-2\right)c}\\
			& Y & \IdentityMatrix{c}
		\end{pmatrix}\right) \differential Y.
	\end{align*}
\end{lemma}
This lemma is proved using the root exchange algorithm where for $1 \le i \le r$, the maps $r_i \colon \squareMatrix_c\left(\localField\right) \to \GL_{kc}\left(\localField\right)$ and $r^{\ast}_i \colon \squareMatrix_c\left(\localField\right) \to \GL_{kc}\left(\localField\right)$ are given by $$r_i\left(y\right) = \begin{pmatrix}
	\IdentityMatrix{c}\\
	& \IdentityMatrix{\left(k-2\right)c}\\
	& y e_i & \IdentityMatrix{c}
\end{pmatrix} \,\,\,\,\,\,\text{and}\,\,\,\,\,\, r_i^{\ast}\left(x\right) = \begin{pmatrix}
\IdentityMatrix{\left(k-2\right)c} & & -\transpose{e_{i+1}} x\\
& \IdentityMatrix{c} & \\
& & \IdentityMatrix{c}
\end{pmatrix}.$$
Here $e_i = \left(0_c,\dots,0_c,\IdentityMatrix{c},0_c,\dots,0_c\right) \in \Mat{c}{\left(k-2\right)c}\left(\localField\right)$ and $\IdentityMatrix{c}$ is located at the $i$th position.

\bibliographystyle{abbrv}
\bibliography{references}
\end{document}